\documentclass[oneside,a4paper,11pt,reqno]{amsart}

\usepackage[T1]{fontenc}
\usepackage[utf8]{inputenc}
\usepackage{graphicx}

\usepackage{amsmath}
\usepackage{amsfonts}
\usepackage{amssymb}
\usepackage{amsthm}
\usepackage{array}
\usepackage{color}
\usepackage{enumerate}
\usepackage{tikz}

\usetikzlibrary{arrows, automata}

\newtheorem{theorem}{Theorem}[section]

\newtheorem{corollary}[theorem]{Corollary}
\newtheorem{definition}[theorem]{Definition}

\newtheorem{hypothesis}[theorem]{Hypothesis}
\newtheorem{lemma}[theorem]{Lemma}
\newtheorem{proposition}[theorem]{Proposition}

\newtheorem{remark}[theorem]{Remark}
\newtheorem{properties}[theorem]{Properties}

\numberwithin{equation}{section}

\newcommand{\N}{\mathbb{N}}
\newcommand{\Z}{\mathbb{Z}}

\newcommand{\R}{\mathbb{R}}
\newcommand{\C}{\mathbb{C}}

\newcommand{\Ebb}{\mathbb{E}}

\newcommand{\Lbb}{\mathbb{L}}

\newcommand{\Pbb}{\mathbb{P}}
\newcommand{\Sbb}{\mathbb{S}}
\newcommand{\Tbb}{\mathbb{T}}

\newcommand{\Bcal}{\mathcal{B}}
\newcommand{\Ccal}{\mathcal{C}}

\newcommand{\Ecal}{\mathcal{E}}
\newcommand{\Fcal}{\mathcal{F}}
\newcommand{\Gcal}{\mathcal{G}}

\newcommand{\Ical}{\mathcal{I}}
\newcommand{\Jcal}{\mathcal{J}}

\newcommand{\Lcal}{\mathcal{L}}

\newcommand{\Ncal}{\mathcal{N}}

\newcommand{\Pcal}{\mathcal{P}}

\newcommand{\Scal}{\mathcal{S}}

\newcommand{\Ycal}{\mathcal{Y}}
\newcommand{\Zcal}{\mathcal{Z}}

\newcommand{\Afrak}{\mathfrak{A}}

\newcommand{\norm}[2]{\left\| #1 \right\|_{#2}}
\newcommand{\dd}{\;{\rm d}}
\newcommand{\de}{{\rm d}}

\DeclareMathOperator{\Leb}{Leb}

\DeclareMathOperator{\Lip}{Lip}

\DeclareMathOperator{\Einf}{Einf}

\DeclareMathOperator{\sgn}{sgn}

\usepackage[top=3cm, bottom=2cm, left=1.5cm, right=1.5cm]{geometry}

\title{Potential kernel, hitting probabilities and distributional asymptotics}
\author{Fran\c{c}oise P\`ene}
\address{Universit\'e de Brest and Institut Universitaire de France,
Laboratoire de Math\'ematiques de Bretagne Atlantique, UMR CNRS 6205, 29238 Brest Cedex, France}
\email{francoise.pene@univ-brest.fr}
\author{Damien Thomine}
\address{D\'epartement de Math\'ematiques d'Orsay, Universit\'e Paris-Sud, 
UMR CNRS 8628, F-91405 Orsay Cedex, France}
\email{damien.thomine@math.u-psud.fr}

\date\today

\begin{document}

\maketitle

\begin{abstract}
$\Z^d$-extensions of probability-preserving dynamical systems are themselves dynamical systems 
preserving an infinite measure, and generalize random walks. 
Using the method of moments, we prove a generalized central limit theorem for additive functionals of 
the extension of integral zero, under spectral assumptions. As a corollary, we get the fact that 
Green-Kubo's formula is invariant under induction. This allows us to relate the hitting probability 
of sites with the symmetrized potential kernel, giving an alternative proof and generalizing 
a theorem of Spitzer. Finally, this relation is used to improve in turn the asumptions of the 
generalized central limit theorem. Applications to Lorentz gases in finite horizon and to the 
geodesic flow on abelian covers of compact manifolds of negative curvature are discussed.
\end{abstract}

\section*{Introduction}

Given a recurrent random walk $(S_n)$ on $\Z^d$, with $d \in \{1,2\}$, a natural question is how much time the walker 
spends in any region of the space -- the so-called occupation times. More generally, one may 
choose an observable $f : \Z^d \to \R$, and consider the Birkhoff averages $n^{-1} \sum_{k=0}^{n-1} f(S_k)$. 
When $f$ is summable and the walk is well-behaved, it is known that $a_n^{-1} \sum_{k=0}^{n-1} f(S_k)$ 
converges in distribution to a Mittag-Leffler random variable, for well-chosen coefficients 
$(a_n)_{n \geq 0}$~\cite{Levy:1940}. This behaviour generalizes to 
null-recurrent Markov processes~\cite{DarlingKac:1957, Aaronson:1997}.

\smallskip

When $f$ has integral zero, this family of results is not sharp enough, and we must look at a higher 
order. In the same way that a central limit theorem replaces the weak law of large numbers, one 
can get a generalized central limit theorem for observables of null-recurrent Markov processes. 
Typically, $a_n^{-1/2} \sum_{k=0}^{n-1} f(S_k)$ converges in distribution, with an explicit limit. 
The story of these central limit theorems starts from Dobrushin~\cite{Dobrushin:1955} where 
$(S_n)$ is the simple random walk on $\Z$. Then these results were generalized to Markov 
processes~\cite{Kesten:1962, Kasahara:1981, Kasahara:1985}, and later included invariance 
principles~\cite{Kasahara:1984, Borodin:1986a, Borodin:1986b}.

\smallskip

In this article, we are interested not in Markov processes, but in a family of dynamical systems 
preserving an infinite measure: $\Z^d$-extensions, which are a generalization of random walks. 
Starting from a dynamical system preserving a probability measure $(A, \mu, T)$ and a function $F : A \to \Z^d$, 
we work with the transformation $\widetilde{T}: (x,p) \mapsto (T(x), p+F(x))$ on $A \times \Z^d$. This class of systems 
include random walks on $\Z^d$, as well as, for instance, Lorentz gases~\cite{BunimovichSinai:1981, BunimovichChernovSinai:1991} 
and the geodesic flow on abelian covers of complete manifolds~\cite{KatsudaSunada:1990, Sharp:1993, PollicottSharp:1994}.
Given an observable $f: A \times \Z^d \to \R$, we want to understand the limit in distribution 
of $\sum_{k=0}^{n-1} f \circ \widetilde{T}^k$.

\smallskip

In two previous works by the second-named author~\cite{Thomine:2014, Thomine:2015}, adapting previous 
methods~\cite{CsakiCsorgoFoldesRevesz:1992, CsakiFoldes:1998, CsakiFoldes:2000}, the case 
where $(A, \mu, T)$ is a Gibbs-Markov map was investigated. In the current article, we are able to get a generalized 
central limit theorem under spectral hypotheses on the transfer operator of the system $(A, \mu, T)$, 
which has much wider applications. The downside is that we need the observable $f(x,p)$ to depend only on 
$p$ and decay fast enough at infinity. This is Theorem~\ref{thm:gene}, which we prove using the method 
of moments (an approach which, to our knowledge, is new for this problem), and apply to Lorentz gases 
with finite horizon.

\smallskip

An interesting corollary of Theorem~\ref{thm:gene} and~\cite[Theorem~6.8]{Thomine:2015} is that, for $\Z^d$-extensions 
of Gibbs-Markov maps, Green-Kubo's formula -- which appears as the asymptotic variance in the central limit theorem -- 
is invariant under induction. This is the content of Corollary~\ref{cor:InvarianceInduction}. 
By choosing the observable $f$ carefully, in Theorem~\ref{thm:ConvergenceNp} 
we are able to related the probability that an excursion from $A \times \{0\}$ 
hits a site $A \times \{p\}$, and the symmetrized potential kernel associated to the $\Z^d$-extension.
Our proof relies on the first hitting time of small target statistics. This method provides a new proof of 
an earlier proposition by Spitzer~\cite[Chap.~III.11, P5]{Spitzer:1976}, and generalizes it 
to $\Z^d$-extensions (for which harmonic analysis as used in~\cite{Spitzer:1976} does not make sense). 
Finally, the estimates from Theorem~\ref{thm:ConvergenceNp} are used to relax the assumptions from~\cite{Thomine:2015}: 
in Theorem~\ref{thm:GM}, the observables need only to decay polynomially at infinity, instead of having 
bounded support. We apply it to the geodesic flow on abelian covers of compact manifolds with negative curvature.

\smallskip

This article is organized as follow. We present our setting and our results in Section~\ref{sec:Main}, 
as well as our applications to Lorentz gases (Sub-subsection~\ref{subsubsec:IntroBillard}) and to geodesic flows 
(Sub-subsection~\ref{subsubsec:IntroFlotGeodesique}). In Section~\ref{sec:operators} we present our spectral assumptions, 
and prove Theorem~\ref{thm:gene} using the method of moments. In Section~\ref{sec:proofGM} we prove 
Theorems~\ref{thm:ConvergenceNp} and~\ref{thm:GM}, and in Section~\ref{sec:Applications} the two 
applications mentioned above. We discuss Green-Kubo's formula in the Appendix.

\section*{Acknowledgements}

The second author is indebted to S\'ebastien Gou\"ezel, both for it guidance during his PhD and for his discussions 
on the subject, as well as to P\'eter N\'andori for suggesting part of the argument of Section~\ref{sec:proofGM}.

\smallskip

We also wish to thank the Erwin Schr\"odinger Institute for the program ``Mixing flow and averaging methods'' (Vienna, 2016), 
the CNRS for the grant PEPS ``Moyennisation et analyse stochastique des syst\`emes lents-rapides'' (2016), and the IUF 
for their support.

\section{Main results}
\label{sec:Main}

\subsection{Setting and goals}

We consider conservative ergodic dynamical systems given by $\Z^d$-extensions of 
probability-preserving dynamical systems, where the underlying dynamical system is 
sufficiently hyperbolic and $d \in \{1,2\}$. We shall deem a system hyperbolic enough if its transfer 
operator satisfies good properties. For some applications, we use the stronger assumption that 
the underlying dynamical system is Gibbs-Markov.

\smallskip

Let $(A,\mu,T)$ be a probability-preserving dynamical system. 
Let $F: A \to \Z^d$ with $d\in\{1,2\}$ be a $\mu$-integrable
function such that $\int_A F\dd \mu=0$. The $\Z^d$-\textit{extension}
$(\widetilde{A},\tilde{\mu},\widetilde{T})$ of $(A,\mu,T)$ with step function $F$
is the dynamical system given by:
\begin{itemize}
\item $\widetilde{A} := A \times \Z^d$;
\item $\tilde{\mu} := \sum_{p\in\Z^d} \mu \otimes \delta_p$;
\item $\widetilde{T}(x,p)=(x,p+F(x))$.
\end{itemize}
Note that $\widetilde{T}$ preserves the infinite measure $\tilde{\mu}$. 
We shall always assume that $(\widetilde{A}, \tilde{\mu}, \widetilde{T})$ is 
ergodic. If $(A, \mu, T)$ has a Markov partition $\pi$, we may also asume that 
the step function $F$ is $\sigma (\pi)$-measurable -- that is, constant almost everywhere 
on elements of the partition. We then say that $(\widetilde{A},\tilde{\mu},\widetilde{T})$ 
is a \textit{Markov extension} of $(A, \mu, T)$.

\smallskip

Let $S_n := S_n^T F := \sum_{k=0}^{n-1} F \circ T^k$ be the second coordinate 
of $\widetilde{T}^n (x,0)$. Heuristically, the sequence $(S_n)_{n \geq 0}$, under the distribution 
$\mu$, behaves much like a random walk, the randomness being generated by the dynamical 
system $(A,\mu,T)$. Indeed, this family of extensions includes every random walk on $\Z^d$, as well 
as some physically or geometrically interesting systems such as Lorentz gases (Sub-subsection~\ref{subsubsec:IntroBillard})
or the geodesic flow on $\Z^d$-periodic manifolds of negative curvature (Sub-subsection~\ref{subsubsec:IntroFlotGeodesique})
\footnote{Up to some lengthy, but in our case not particularly challenging, legwork to go from discrete time to continuous time.}.

\smallskip

In the present paper, we will make assumptions ensuring the convergence in distribution 
of $(S_n/\mathfrak{a}_n)_n$ to a L\'evy stable distribution, for some normalizing sequence $(\mathfrak{a}_n)_n$. 
Our main goals are the following:
\begin{itemize}
\item[(A)] Given $f:\widetilde{A} \rightarrow \R$ integrable
and such that $\int_{\widetilde{A}} f \dd \tilde\mu=0$, we are interested 
in the asymptotic behaviour of the ergodic sum
\begin{equation*}
S_n^{\widetilde{T}} f
= \sum_ {k=0}^{n-1} f \circ \widetilde{T}^k,
\end{equation*}
as $n\rightarrow +\infty$. More precisely, we are looking for a non-trivial 
strong convergence in distribution:
\begin{equation}\label{TCLMI}
\frac{S_n^{\widetilde{T}} f}{\Afrak_n} 
\stackrel{\text{dist.}}{\Rightarrow} \sigma(f) \Ycal, 
\mbox{ with } \Afrak_n := \sqrt{\sum_{k=1}^n \mathfrak{a}_k^{-d}},
\end{equation}
where $\sigma (f)$ is some constant which depends on the pushforward of the measure $\tilde{\mu}$ by $(f \circ \widetilde{T}^n)_{n \geq 0}$, 
whereas the random variable $\Ycal$ depends only 
on\footnote{Up to a constant, $\Ycal$ actually depends only on the index $\alpha$ of the L\'evy stable distribution 
that is the limit of $(S_n/\mathfrak{a}_n)_n$.} the distribution of $(F\circ T^k)_k$ (with respect to $\mu$). 
\item[(B)] In the context of Gibbs-Markov maps, we consider the probability
$\alpha(p)^{-1}$, starting from the $A \times \{0\}$ endowed with the measure $\mu$, 
to visit $A \times \{p\}$ before coming back to $A \times \{0\}$. 
By applying the limit theorems we have proved before to $f_p (x,q)
:= (\mathbf{1}_{\{p\}}-\mathbf{1}_{\{0\}})(q)$, we are able to prove that
\begin{equation*}
\alpha (p) 
\sim \frac{\sigma(f_p)}{2} \text{ as } p \text{ goes to } \infty,
\end{equation*}
which provides a new proof of~\cite[Chap.~III.11, P5]{Spitzer:1976}, and generalizes it to systems which are not random walks.
\end{itemize}
The next sub-sections present in more details these two goals, and the precise statements we get.

\subsection{Distributional limit theorems}
\label{subsec:TheoremesLimiteDistribution}

\subsubsection{Convergence and limit distributions}
\label{subsubsec:ConvergenceLoisLimite}

When working with spaces endowed with an infinite measure, there is no natural notion 
of convergence in distribution. We shall instead use the notion of \textit{strong convergence in distribution}. 
The reader may consult e.g.\ \cite[Chapter~3.6]{Aaronson:1997} for an introduction to this notion 
and applications to ergodic dynamical systems whose invariant measure is infinite.

\begin{definition}[Strong convergence in distribution]

Let $(\widetilde{A}, \tilde{\mu})$ be a measured space. Let $(X_n)_{n \geq 0}$ be a sequence of 
measurable functions from $\widetilde{A}$ to $\R$. Let $X$ be a real-valued random variable. 
We say that $(X_n)$ converges strongly in distribution to $X$ if, for all probability measures 
$\nu \ll \tilde{\mu}$,
\begin{equation*}
X_n \to_{n \to + \infty} X \text{ in distribution on } (\widetilde{A}, \nu).
\end{equation*}
\end{definition}

Now that we have defined our mode of convergence, we introduce our limit objects: 
Mittag-Leffler random variables, and Mittag-Leffler -- Gaussian mixtures.

\begin{definition}[$ML$ and $MLGM$ random variables]

Let $\gamma \in [0,1]$. Let $X$ be a non-negative real-valued random variable. We say that $X$ 
follows a standard Mittag-Leffler distribution of index $\gamma$ if, for all $z \in \C$ 
(or all $z \in B(0,1)$ if $\alpha = 0$),
\begin{equation*}
\Ebb [e^{z X}] 
= \sum_{n=0}^{+ \infty} \frac{\Gamma(1+\gamma)^n z^n}{\Gamma (1+n \gamma)}.
\end{equation*}
If this is the case, we shall write that $X$ has a $ML (\gamma)$ distribution.

\smallskip

Let $X$ be a real-valued random variable. We say that $X$ 
follows a standard Mittag-Leffler -- Gaussian mixture distribution of index $\gamma$ 
if $X=\sqrt{Y}\cdot Z$, where $Y$ and $Z$ are two independent random variables 
with respective distribution $ML(\gamma)$ and standard normal $\Ncal (0,1)$.
If this is the case, we shall write that $X$ has a $MLGM (\gamma)$ distribution. 
See~\cite[Chapitre~1.4]{Thomine:2013} for a partial description of the MLGM distributions.
\end{definition}

For $\gamma = 0$, these distributions take more common forms: a $ML(0)$ distribution 
is an exponential distribution of parameter $1$, while a  $MLGM(0)$ distribution 
is a Laplace distribution of parameter $1/\sqrt{2}$. A $ML(1/2)$ random variable is 
the absolute value of a centered Gaussian of variance $\pi/2$.

\subsubsection{Main distributional theorem}
\label{subsubsec:TheoremDistriutionnelPrincipal}

Mittag-Leffler distributions appear when one studies the distributional convergence 
of the local time of null recurrent Markov processes, or chaotic enough $\sigma$-finite 
ergodic dynamical systems. For the Brownian motion, the result goes back to P.~L\'evy~\cite{Levy:1940}, 
and to Darling-Kac's theorem for Markov chains~\cite{DarlingKac:1957}. 
We refer the reader to~\cite{MolchanovOstrovskii:1969} for $\alpha$-stable L\'evy processes, 
and to~\cite[Corollary~3.7.3]{Aaronson:1997} for dynamical systems in infinite ergodic theory. 
For instance~\cite[Corollary~3.7.3]{Aaronson:1997} and Hopf's ergodic 
theorem~\cite[\S$14$, Individueller Ergodensatz f\"ur Abbildungen]{Hopf:1937} yield:

\begin{proposition}

Let $(\widetilde{A},\tilde{\mu},\widetilde{T})$ be a measure-preserving transformation of a $\sigma$-finite measure space. 
Assume that $\widetilde{T}$ is pointwise dual ergodic with return sequence $(a_n)_n$ (see~\cite[Chapter~3.5]{Aaronson:1997} 
for definitions). Assume that $(a_n)_n$ has regular variation of index $\alpha \in [0,1]$, i.e.\ $a_n = n^{1/\alpha} L(n)$ 
for some sequence $L$ which varies slowly at infinity. 
Then, for all $f \in \Lbb^1 (\widetilde{A}, \tilde{\mu})$,
\begin{equation*}
\frac{S_n^{\widetilde{T}} f}{a_n}
\Rightarrow \int_{\widetilde{A}} f \dd \tilde{\mu} \cdot \Ycal,
\end{equation*}
where $\Ycal$ is a standard $ML(\alpha)$ random variable and the convergence is strong in distribution.
\end{proposition}

\smallskip

However, this kind of result is not sharp enough when the integral of the observable $f$ 
is zero. We want to get more precise asymptotics, that is to say, some kind of 
central limit theorem for observables of $\sigma$-finite ergodic dynamical systems whose integral is $0$. 
We need to add some regularity condition on the observable $f$, as well as stronger 
integrability conditions -- as is usual in ergodic theory, for instance to get a central limit theorem 
\cite{Nagaev:1957, Nagaev:1961}. In this article, we shall prove the following result:

\begin{theorem}
\label{thm:gene}

Let $(\widetilde{A}, \tilde{\mu}, \widetilde{T})$ be an ergodic and aperiodic $\Z^d$-extension 
of $(A, \mu, T)$ with step function $F$ and  $\alpha \in [d,2]$. Assume Hypothesis~\ref{hyp:HHH}.
Let $(\mathfrak{a}_n)_n$ be an $\alpha^{-1}$-regularly varying
sequence of positive numbers and $Y$ be an $\alpha$-stable random variable $Y$
such that
\begin{equation*}
S_n/\mathfrak{a}_n
\stackrel{distrib.}{\Rightarrow} Y.
\end{equation*}
Let $\Afrak_n:=\sqrt{\sum_{k=1}^n \mathfrak{a}_k^{-d}}$. Let $\beta: \Z^d \rightarrow \R$ be such that:
\begin{itemize}
\item $\sum_{p\in\Z^d}|p|^{\frac{\alpha-d}{2}+\varepsilon}|\beta(p)|<+\infty$ for some $\varepsilon >0$;
\item $\sum_{p\in\Z^d} \beta(p)=0$.
\end{itemize}
Let $f(x,p):=\beta(p)$. Then the following sums over $k$ are absolutely convergent:
\begin{equation}
\label{AAA2}
\sigma_{GK}^2 (f, \widetilde{A}, \tilde{\mu}, \widetilde{T})
= \int_{\widetilde{A}} f^2 \dd \tilde{\mu} + 2\sum_{k \geq 1} \int_{\widetilde{A}} f \cdot f \circ \widetilde{T}^k \dd \tilde{\mu}.
\end{equation}
Moreover, 
\begin{equation}\label{eq:ConvergenceDistributionGene}
\frac{S_n^{\widetilde{T}} f}{\sqrt{\Phi(0)}\Afrak_n}
\Rightarrow \sigma_{GK} (f, \widetilde{A}, \tilde{\mu}, \widetilde{T}) \Ycal,
\end{equation}
where $\Ycal$ is a standard $MLGM(1-\frac{d}{\alpha})$ random variable and the convergence is strong in distribution, 
and where $\Phi$ is the continuous version of the density function of $Y$.
\end{theorem}

Under the hypotheses of Theorem~\ref{thm:gene}, we have in addition:
\begin{equation}
\label{AAA1}
\sigma_{GK}^2 (f, \widetilde{A}, \tilde{\mu}, \widetilde{T}) 
= \sum_{a \in \Z^d} \beta(a)^2 + 2 \sum_{k \geq 1}\sum_{a,b \in \Z^d} \beta(a)\beta(b)\mu (S_k=a-b).
\end{equation}

\begin{remark}
\label{rmk:Aperiodicite}

For a definition of aperiodicity, see Definition~\ref{def:Aperiodicite}. It is not necessary 
in the statement in the theorem, but appears as a result of Hypothesis~\ref{hyp:HHH}, and we prefer 
to make this assumption explicit.

\smallskip

We do not expect aperiodicity to be necessary in the statement of Theorem~\ref{thm:gene}, 
up to the necessary modification in Hypotheses~\ref{hyp:HHH}. 
Proving this generalization would be straightforward if $f$ were allowed to depend on $x$; 
however, allowing such a dependence would make the proof of Theorem~\ref{thm:gene} 
much more difficult. We choose to leave the non-aperiodic case aside, except for 
a couple of later results, Theorem~\ref{thm:ConvergenceNp} and Theorem~\ref{thm:GM}.
\end{remark}

Theorem~\ref{thm:gene} shall be proved in Section~\ref{sec:operators} with the method of moments 
and is based on refinements of the local limit theorem for $S_n$, which says that $\Pbb (S_n=0)\sim\Phi(0)\mathfrak{a}_n^{-d}$.
Under our hypotheses, the normalization $\sqrt{\Phi(0)} \Afrak_n$ is equivalent to $\sqrt{\sum_{k=0}^{n-1} \mu (S_k=0)}$.
See e.g.\ \cite{AaronsonDenker:2001} for a spectral proof of the local limit theorem, which holds under 
Hypotheses~\ref{hyp:HHH}, and implies the equivalence of the normalizations.

\smallskip

In special cases, the normalization $\Afrak_n$ can be made explicit:
\begin{equation*}
\Afrak_n \sim \left\{
\begin{array}{c}
\sqrt{\frac{\alpha}{\alpha-1}\frac{n}{\mathfrak{a}_n}} \quad \text{ if } d=1 \text{ and } \alpha>1 \\
\sqrt{\log n} \quad \text{ if } d=\alpha \text{ and } a_n \sim n^{1/\alpha} \\
\sqrt{\log\log n} \quad \text{ if } d=\alpha \text{ and } a_n \sim (n \log n)^{1/\alpha}
\end{array}
\right. .
\end{equation*}

\subsubsection{Symmetrized potential kernel}
\label{subsubsec:NoyauPotentielSymetrise}

The case when $f=f_p$ of Theorem~\ref{thm:gene} is especially interesting. 
Then the computation of $\sigma_{GK} (f_p, \widetilde{A}, \tilde{\mu}, \widetilde{T})$ boils down to an estimation 
of the symmetrized potential kernel $g$ of the $\Z^d$-extension:
\begin{equation*}
\sigma_{GK}^2 (f_p, \widetilde{A}, \tilde{\mu}, \widetilde{T}) 
= 2g(p)-2,
\end{equation*}
with
\begin{equation*}
g(p) 
:=\sum_{n\geq 0} \left(2\, \mu(S_n=0)-\mu(S_n=p)-\mu(S_n=-p) \right),
\end{equation*}
which is well defined under the assumptions of Theorem \ref{thm:gene}. 
We estimate the asymptotic growth of $g(p)$ in Subsection~\ref{subsec:renewal}, 
adapting the methods of~\cite{Spitzer:1976} to dynamical systems. We get:

\begin{proposition}\label{prop:renouvellement}

Let $(\widetilde{A}, \tilde{\mu}, \widetilde{T})$ be an ergodic, aperiodic $\Z^d$-extension 
of $(A, \mu, T)$ with step function $F$. Let $(\Bcal,\norm{\cdot}{\Bcal})$ be a 
complex Banach space of functions defined on $A$. Assume Hypothesis~\ref{hyp:HHH} 
holds with $(\Bcal,\norm{\cdot}{\Bcal})$ and $\alpha \in [d,2]$. 
If $\alpha = d$, let $I$ be the functions defined by Equation~\eqref{eq:FormuleI}.

\smallskip

If $d=1$ and $\alpha \in (1,2]$,
\begin{equation*}
g(p)
\sim_{p\rightarrow \infty} \frac{1}{\vartheta (1+\zeta^2)\Gamma (\alpha)\sin \left( \frac{(\alpha-1) \pi}{2} \right)} \frac{|p|^{\alpha-1}}{L(|p|)}.
\end{equation*}

If $d=\alpha=1$, 
\begin{equation*}
g(p)
\sim_{p\rightarrow \infty} \frac{2}{\pi \vartheta (1+\zeta^2)} I(|p|^{-1}).
\end{equation*}

\smallskip

If $d=\alpha=2$, 
\begin{equation*}
g(p)
\sim_{p\rightarrow \infty} \frac{2}{\pi \sqrt{\det (\Sigma)}} I(|p|^{-1}).
\end{equation*}
\end{proposition}

\subsection{Hitting probabilities of excursions}
\label{sec:ProbabiliteAtteinteExcursions}

We leave aside for a moment the distributional asymptotics of the Birkhoff sums, 
and focus on the probability that an excursion hits a given site (Section~\ref{sec:proofGM}). 
We now assume that $(A, \mu, T)$ is a Gibbs-Markov map. The leading theme of this 
section is the study of the probability that an excursion from $A \times \{0\}$ hits 
$A \times \{p\}$, and its asymptotics as $p$ goes to infinity.

\subsubsection{Induced transformations}
\label{subsubsec:TransformationsInduites}

Let us describe the terminology.
We define $\varphi_{\{0\}}:A \rightarrow \N_+\cup\{\infty\}$, where
$\varphi_{\{0\}}(x)$ is the length of an excursion starting from $(x,0)$:
\begin{equation*}
\varphi_{\{0\}} (x)
:= \inf \{k>0\ :\ S_k(x)=0\}.
\end{equation*}
Then, define the corresponding induced map by $\widetilde{T}_{\{0\}} (x) := T^{\varphi_{\{0\}}(x)} (x)$, 
which is well-defined for almost every $x$. Note that $(A, \mu, \widetilde{T}_{\{0\}})$ is a measure-preserving 
ergodic dynamical system~\cite{Kakutani:1943}. For any observable $f : \widetilde{A} \to \R$, let:
\begin{equation*}
f_{\{0\}} (x) 
:=\sum_{k=0}^{\varphi_{\{0\}}(x)-1} f \circ \widetilde{T}^k (x,0).
\end{equation*}

\smallskip

Let us introduce a few more objects: the time $N_p$ that 
an excursion from $A \times \{0\}$ spends at $A \times \{p\}$, and the inverse probability 
$\alpha(p)$ that an excursion from $A \times \{0\}$ hits $A \times \{p\}$, and the number 
of times $N_{0,p}$ that the system goes back to $A \times \{0\}$ before hitting $A \times \{p\}$. 
Formally,
\begin{equation*}
N_p(x)
:= \#\left\{k=0,\ldots,\varphi_{\{0\}} (x)-1\ :\ S_k (x)= p\right\} 
= 1+f_{p,\{0\}} (x),
\end{equation*}
\begin{equation*}
\alpha(p)
:= \mu (N_p>0)^{-1} 
= \mu (\exists \ 0 \leq k < \varphi(x): \ S_k (x)=p)^{-1},
\end{equation*}
and:
\begin{equation*}
N_{0,p} (x)
:= \inf \{n \geq 0: \ T_{\{0\}}^n (x) \in \{N_p >0\} \}.
\end{equation*}

The following theorem explains how these quantities are related in the limit $p \to \infty$.

\begin{theorem}\label{thm:ConvergenceNp}

Let $(\widetilde{A},\tilde{\mu},\widetilde{T})$ be a conservative and ergodic\footnote{The extension need not be aperiodic for this theorem.} 
Markov $\Z^d$-extension of a Gibbs-Markov map $(A, \mu, T)$. Then:
\begin{itemize}
\item As $p\rightarrow +\infty$,
\begin{equation*}
\alpha(p) 
\sim \alpha(-p)
\sim \Ebb_\mu [N_p|N_p>0] 
\sim \Ebb_\mu [N_{0,p}] 
\sim \frac{\sigma_{GK}^2 (f_{p,\{0\}}, A, \mu, \widetilde{T}_{\{0\}})}{2}
\sim \frac{\Ebb_\mu [f_{p,\{0\}}^2]}{2}.
\end{equation*}
\item The conditional distributions $\alpha(p)^{-1} N_p | \{N_p > 0\}$ have exponential tails, 
uniformly in $p$: there exist $C \geq 0$ and $\kappa>0$ such that, for every $t>0$,
\begin{equation*}
\sup_{p \in \Z^d} \mu \left((\alpha(p))^{-1} N_p > t | N_p > 0 \right)
\leq Ce^{-\kappa t};
\end{equation*}
\item $\alpha(p)^{-1} N_p | \{N_p > 0\}$ converges in distribution and in moments to an exponential 
random variable of parameter $1$ as $p$ goes to infinity. In particular, for all $q > 1$,
\begin{equation*}
\Ebb_\mu \left[|f_{p,\{0\}}|^q \right] 
\sim \Gamma (1+q) \alpha (p)^{q-1}.
\end{equation*}
\end{itemize}
\end{theorem}

The proof of Theorem~\ref{thm:ConvergenceNp} rests on two main points: the exponential tightness of 
$\alpha(p)^{-1} N_p | \{N_p > 0\}$ (Subsection~\ref{subsec:Tightness}), and its convergence to 
an exponential random variable (Subsection~\ref{subsec:ConvergenceLoiPetitesBoules}). The later 
point is an interesting application of the general fact that, for many hyperbolic dynamical systems, 
the hitting time of small balls, once renormalized, converges in distribution to an exponential 
random variable (see e.g.\ the reviews~\cite{Coelho:2000, Saussol:2009, Haydn:2013}). Once we have tightness and 
convergence in distribution, we can evaluate the moments of $N_p$. 

\smallskip

For random walks, many estimates are more explicit. For instance, the distribution of 
$N_p | \{N_p > 0\}$ is geometric, so its moments are exactly known (as functions of $\alpha (p)$). 
With this improvement, one can recover part of~\cite[Chap.~III.11, P5]{Spitzer:1976} -- that is, the 
equivalents in Theorem~\ref{thm:ConvergenceNp} and Corollary~\ref{cor:AlphaG} can be made into equalities:
\begin{equation*}
\alpha(p) 
= \alpha(-p)
= \Ebb_\mu [N_p|N_p>0] 
= 1+\Ebb_\mu [N_{0,p}] 
= 1+\frac{\sigma_{GK}^2 (f_{p,\{0\}}, A, \mu, \widetilde{T}_{\{0\}})}{2}
= 1+\frac{\Ebb_\mu [f_{p,\{0\}}^2]}{2} 
= g(p).
\end{equation*}

\subsubsection{Induction invariance of the Green-Kubo formula}
\label{subsubsec:InductionInvariance}

While Theorem~\ref{thm:ConvergenceNp} gives asymptotic relationships between many quantities, it does not provide 
any way to effectively compute them. For random walks, $\alpha (p)$ and $g(p)$ are related through 
a probabilistic interpretation of the symmetrized potential kernel:

\begin{proposition}\cite[Chap.~III.11, P5]{Spitzer:1976}
\label{prop:Spitzer}

Consider an ergodic aperiodic random walk on $\Z^2$. For all $p \in \Z^2$,
\begin{equation*}
\alpha(p) 
= g(p).
\end{equation*}
\end{proposition}

We are able to generalize this proposition to a larger class of dynamical systems. 
To our knowledge, our proof of Proposition~\ref{prop:Spitzer} is new even for random walks.
We leverage Theorem~\ref{thm:gene} and \cite[Theorem~6.8]{Thomine:2015}. Whenever the hypotheses 
of theses theorems coincide, their conclusions must be the same. Hence, the scaling factors before 
the $MLGM$ distribution must be the same, that is,
\begin{equation}\label{eq:InvarianceGKRestreinte}
\sigma_{GK}^2 (f, \widetilde{A}, \tilde{\mu}, \widetilde{T})
= \sigma_{GK}^2 (f_{\{0\}}, A, \mu, \widetilde{T}_{\{0\}}).
\end{equation}
If we apply this observation to $f = f_p$, we get:

\begin{corollary}\label{cor:AlphaG}

Let $(\widetilde{A}, \tilde{\mu}, \widetilde{T})$ be an aperiodic Markov $\Z^d$-extension 
of a Gibbs-Markov map $(A,\pi,\lambda,\mu,T)$ with step function $F$. Assume that the extension is 
ergodic, conservative, and either of the following hypotheses:
\begin{itemize}
\item $d=1$ and $F$ is in the domain of attraction of an $\alpha$-stable distribution, with $\alpha \in (1,2]$.
\item $d=1$ and $\int_A e^{iuF}\dd \mu = e^{-\vartheta|u|[1-i\zeta\sgn(u)]L(|u|^{-1})} + o\left(|u| L(|u|^{-1})\right)$ at $0$, 
for some real numbers $\vartheta>0$ and $\zeta\in\R$ and some function $L$ with slow variation. 
\item $d=2$ and $F$ is in the domain of attraction of a non-degenerate Gaussian random variable.
\end{itemize}

Then, as $p\rightarrow +\infty$,
\begin{equation*}
\alpha(p) 
\sim g(p).
\end{equation*}
\end{corollary}

\begin{remark}[$1$-stable laws]

The description of the distributions in the basin of attraction of a $1$-stable law is 
notoriously difficult~\cite{AaronsonDenker:1998}. As in Hypothesis~\ref{hyp:HHH}, we choose to make a spectral assumption. 
It does not capture all such distributions, but includes e.g.\ symmetric distributions. 
We believe that this assumption can be significantly weakened if needed.
\end{remark}

\subsubsection{An improved distributional limit theorem}
\label{subsubsec:TheoremeLimiteAmeliore}

Proposition~\ref{prop:renouvellement} provides a first-order estimate of $\alpha (p)$, 
depending on the tails of $F$. We can use this estimate to run an (improved version of) 
an argument by Cs\'aki, Cs\"org\H{o}, F\"oldes and R\'ev\'esz~\cite[Lemma~3.1]{CsakiCsorgoFoldesRevesz:1992}. 
we get more explicit integrability conditions than in~\cite[Theorem~6.8]{Thomine:2015} 
for observables of $\Z^d$-extensions, which yields a new distributional limit theorem. 
Note that aperiodicity is not required for this result.

\begin{theorem}
\label{thm:GM}

Let $(\widetilde{A}, \tilde{\mu}, \widetilde{T})$ be a Markov $\Z^d$-extension 
of a Gibbs-Markov map $(A,\pi,\lambda,\mu,T)$ with step function $F$. Assume that the extension is 
ergodic, conservative, and either of the following hypotheses:
\begin{itemize}
\item $d=1$ and $F$ is in the domain of attraction of an $\alpha$-stable distribution, with $\alpha \in (1,2]$.
\item $d=1$ and $\int_A e^{iuF}\dd \mu = e^{-\vartheta|u|[1-i\zeta\sgn(u)]L(|u|^{-1})} + o\left(|u| L(|u|^{-1})\right)$ at $0$, 
for some real numbers $\vartheta>0$ and $\zeta\in\R$ and some function $L$ with slow variation. 
\item $d=2$ and $F$ is in the domain of attraction of a non-degenerate Gaussian random variable.
\end{itemize}

\smallskip

Let $f:\widetilde{A} \rightarrow \R$ be such that:
\begin{itemize}
\item the family of function $(f (\cdot, p))_{p \in \Z^d}$ is uniformly locally $\eta$-H\"older for some $\eta > 0$;
\item $\int_{\widetilde{A}} (1+|p|)^{\frac{\alpha-d}{2}+\varepsilon} \norm{f(\cdot,p)}{\Lbb^q (A,\mu)} \dd \tilde{\mu}(x,p) <+\infty$ for some $\varepsilon >0$ and $q>2$; 
\item $\int_{\widetilde{A}} f \dd \tilde{\mu}=0$.
\end{itemize}
Then:
\begin{equation}\label{eq:ConvergenceDistributionExcursions}
\frac{S_n^{\widetilde{T}} f}{\sqrt{\sum_{k=0}^{n-1} \mu (S_k=0)}}
\Rightarrow \sigma_{GK} (f_{\{0\}}, A, \mu, \widetilde{T}_{\{0\}}) \Ycal,
\end{equation}
where $\Ycal$ is a standard $MLGM(1-\frac{d}{\alpha})$ random variable, the convergence is strong in distribution, and:
\begin{equation*}
\sigma_{GK}^2 (f_{\{0\}}, A, \mu, \widetilde{T}_{\{0\}})
:= \lim_{n \to + \infty} \int_A f_{\{0\}}^2 \dd\mu + 2 \sum_{k=1}^n \int_A f_{\{0\}} \cdot f_{\{0\}} \circ \widetilde{T}_{\{0\}}^k \dd \mu,
\end{equation*}
where the limit is taken in the Ces\`aro sense.
\end{theorem}

\begin{remark}[Optimal exponent in the summability assumption]

We consider the case when $d=1$ and $\alpha=2$. In~\cite{CsakiCsorgoFoldesRevesz:1992} and some 
subsequent works by the same authors, the condition required for $f$ is:
\begin{equation}
\sum_{p \in G} |p|^{1+\varepsilon} |\beta(p)|
< + \infty.
\end{equation}
The reason is that the authors used Jensen's inequality in their proof~\cite[Lemma~2.1]{CsakiCsorgoFoldesRevesz:1992}, 
which is in this context less efficient than Minkowski's inequality, which we used in the proof of Lemma~\ref{lem:sommabi}. 
This small modification can be implemented in their proof, which improves by a factor $2$ some requirements in their works, 
e.g.\ \cite[Theorem~1]{CsakiCsorgoFoldesRevesz:1992} and \cite[Example~3.3]{CsakiFoldes:1998}.
\end{remark}

Finally, the hypotheses of Theorem~\ref{thm:gene} and of Theorem~\ref{thm:GM} 
have a greater overlap than those of Theorem~\ref{thm:gene} and \cite[Theorem~6.8]{Thomine:2015}, 
so we can improve the observation in Equation~\eqref{eq:InvarianceGKRestreinte}:

\begin{corollary}[Induction Invariance of the Green-Kubo formula]\label{cor:InvarianceInduction}

Let $(A, \pi, \lambda, \mu, T)$ be an ergodic Gibbs-Markov map. Assume that the step function $F : \ A \to \Z^d$ 
is $\sigma (\pi)$-measurable, integrable, aperiodic, and that $\int_A F \dd \mu=0$. We also assume that the distribution of $F$ 
with respect to $\mu$ is in the domain of attraction of an $\alpha$-stable distribution, 
and that the Markov extension $(\widetilde{A},\tilde{\mu},\widetilde{T})$ is conservative and ergodic.

\smallskip

Let $\beta: \Z^d \rightarrow \R$ be such that:
\begin{itemize}
\item $\sum_{p\in\Z^d}|p|^{\frac{\alpha-d}{2}+\varepsilon}|\beta(p)|<+\infty$ for some $\varepsilon >0$;
\item $\sum_{p\in\Z^d} \beta(p)=0$.
\end{itemize}
Let $f(x,p):=\beta(p)$. Then:
\begin{equation}
\sigma_{GK} (f, \widetilde{A}, \tilde{\mu}, \widetilde{T})
= \sigma_{GK} (f_{\{0\}}, A, \mu, \widetilde{T}_{\{0\}}).
\end{equation}
\end{corollary}

See Appendix~\ref{sec:GreenKubo} for a discussion of this corollary.

\subsection{Applications}
\label{subsec:Applications}

To finish this introduction, we present some applications of our results to 
more concrete dynamical systems : the geodesic flow on abelian covers in negative 
curvature, and Lorentz gases (i.e.\ periodic planar billiards). 
The proofs can be found in Section~\ref{sec:Applications}.

\subsubsection{Periodic planar billiard systems}
\label{subsubsec:IntroBillard}

Lorentz gases -- that is, periodic or quasi-periodic convex billiards -- 
are classical dynamical systems, whose initial motivation comes from the modelization 
of a gas of electrons in a metal. The electron is then seen as bouncing on the atoms 
of the metal, which act as scatterers.

\smallskip

In the plane and with a finite horizon, Lorentz gases exhibit classical diffusion, and the trajectory of a particle 
behaves much like a random walk in the Euclidean space. For instance, the trajectories 
are chaotic~\cite{Sinai:1970}, satisfy a central limit theorem~\cite{BunimovichSinai:1981, BunimovichChernovSinai:1991}, 
a local limite theorem~\cite{SzaszVarju:2004}, 
an almost sure invariance principle~\cite{Gouezel:2010} (i.e.\ the renormalized trajectories converge 
in a strong sense to the trajectories of a Brownian motion), etc. We refer the reader 
to~\cite{ChernovMarkarian:2006} for more informations of billiards. 
While the infinite horizon case is also well-known~\cite{SzaszVarju:2007,DolgopyatSzaszVarju:2008}, 
it presents many non-trivial difficulties, so we shall restrict ourselves to finite horizon planar billiards.

\smallskip

Choose a $\Z^2$-periodic locally finite configuration of obstacles 
$(p+O_i: \ i\in\Ical, \ p \in \Z^2)$, where $\Ical$ is a finite set. 
We assume that the obstacles $O_i+p$ are convex open sets, with pairwise disjoint closures (so that there is no cusp), 
that their boundaries are $\Ccal^3$ and have non-vanishing curvature. We assume moreover that the horizon is finite: 
every line in $\R^2$ meets at least one obstacle. The billiard domain is the complement in $\R^2$
of the union of the obstacles $Q:=\R^2\setminus\bigcup_{i\in\Ical,\ p\in \Z^2} (p+O_i)$.

\begin{figure}[h!]
\vspace{1cm}
\centering
\includegraphics[scale=1.7, trim= 10mm 20mm 10mm 20mm, clip]{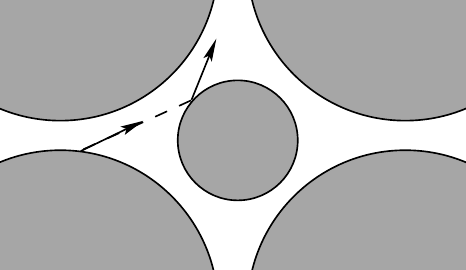}
\caption{A Sinai billiard with finite horizon.}
\end{figure}

We consider a point particle moving at unit speed in the billiard domain $Q$, 
bouncing on obstacles with the classical Descartes reflexion law: the incident angle equals the reflected angle 
and going straight on between two collisions. This is the billiard flow, whose configuration space 
is (up to a set of zero measure) $Q \times \Sbb_1$. Now, consider this model at collision times; 
the configuration space is then given by $\Omega := \partial Q \times [-\pi/2, \pi/2]$. The space 
$\Omega$ is endowed with the Liouville measure $\tilde{\nu}$, which has density $\cos (\phi)$ in $(x, \phi)$ 
with respect to the Lebesgue measure (see the picture), and is invariant under the collision map.

\begin{figure}[h!]
\vspace{1cm}
\centering
\includegraphics[scale=1.6, trim= 10mm 25mm 35mm 25mm, clip]{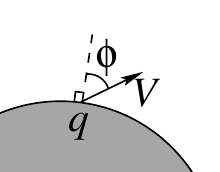}
\caption{A single collision.}
\end{figure}

For every $p\in \Z^2$, we call \textit{cell} any set $\Ccal_p:=\bigcup_{i\in\Ical} (p+\partial{O_i})$
and attribute to this cell a value $\beta(p)$ given by a function $\beta: \Z^2 \rightarrow \R$. 
We assume that the particle wins the value $\beta(p)$ associated to $\Ccal_p$ each time it touches it.
We are interested in the behaviour, as $n\rightarrow +\infty$, of the total amount $\Ycal_n$ won by
the particle after the $n$-th reflection.

\begin{figure}[h!]
\vspace{1cm}
\centering
\includegraphics[scale=0.5, trim= 10mm 40mm 10mm 50mm, clip]{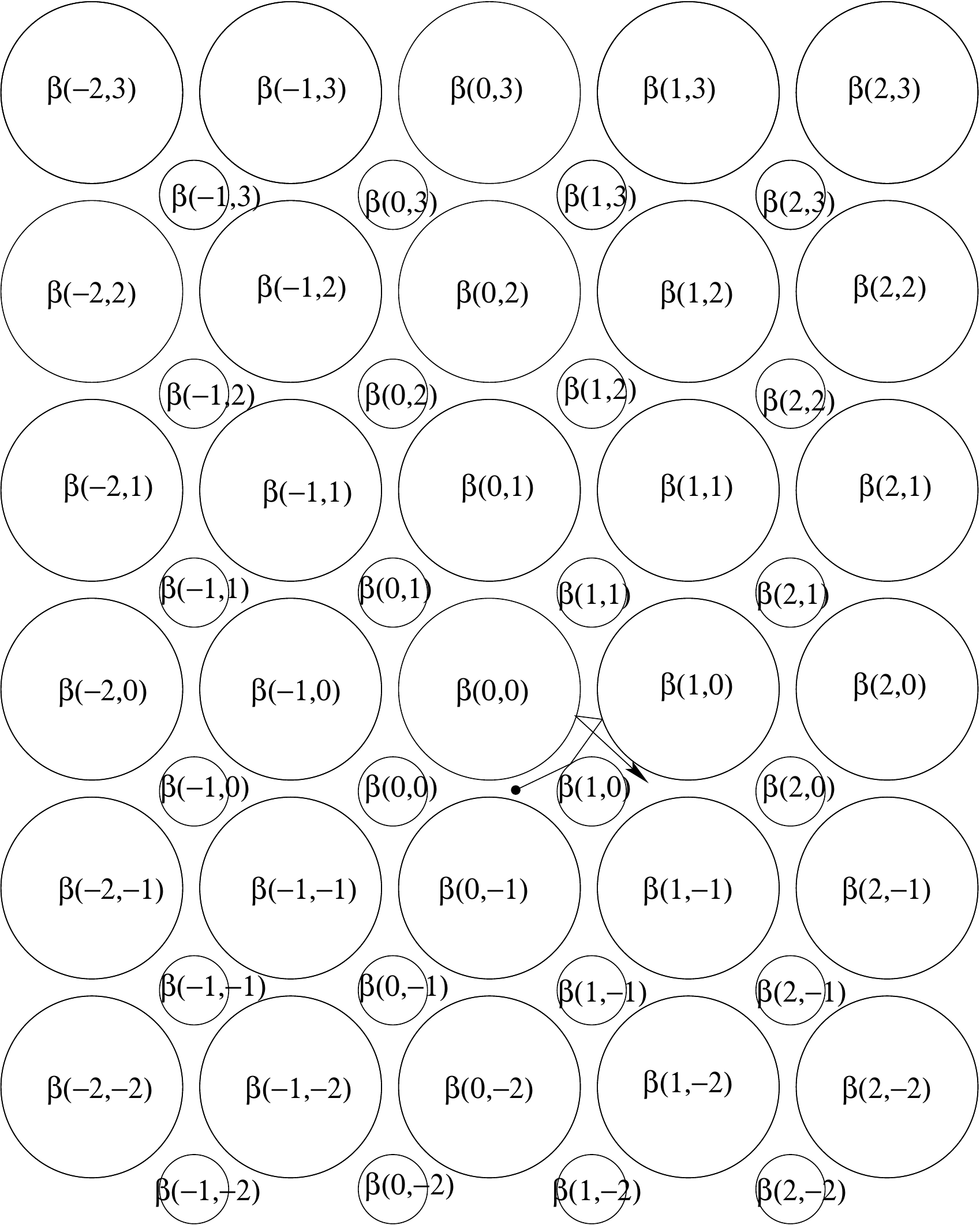}
\caption{A periodic billiard table, and the observable $\beta$.}
\end{figure}

We write $S_n(x)$ for the index in $\Z^2$ of the cell touched at the $n$-th reflection time by a particle 
starting from $x\in \Omega$. Recall that $(S_n/\sqrt{n})_n$ converges strongly in distribution 
(with respect to the Lebesgue measure on $\Omega$) to a centered gaussian random variable with 
positive definite covariance matrix $\Sigma$ \cite{BunimovichSinai:1981,BunimovichChernovSinai:1991,Young:1998}.

\smallskip

If $\beta$ is summable and $\sum_{p \in \Z^2} \beta(p) \neq 0$, then $\Ycal_n/ \log (n)$ 
converges strongly in distribution to $\left( \sum_{p \in \Z^2} \beta(p) \neq 0 \right) \Ecal$, 
where $\Ecal$ has a non-degenerate exponential distribution. This follows e.g.\ from 
\cite[Corollary~3.7.3]{Aaronson:1997} and Young's construction~\cite{Young:1998}, and is also 
done in~\cite{DolgopyatSzaszVarju:2008}. In another direction, if $(\beta(p))_{p \in \Z^2}$ 
is a sequence of independent identically distributed random variables independent of 
the billiard, the asymptotic behavior of $(\Ycal_n)$ is markedly different~\cite{Pene:2009a}.

\smallskip

We present two applications of Theorem~\ref{thm:gene}, the first for (hidden) $\Z$-extensions, 
and the second for $\Z^2$-extensions.

\begin{corollary}\label{coro:bill1}

With the above notations, assume that:
\begin{itemize}
\item $\beta (a,b) = \tilde{\beta} (a)$ for some function $\tilde{\beta}$;
\item there exists $\varepsilon >0$ such that $\sum_{p \in \Z} |p|^{\frac{1}{2}+\varepsilon} |\tilde{\beta}|(p) < + \infty$;
\item $\sum_{p \in \Z} \tilde{\beta}(p) =0$.
\end{itemize}

\smallskip

Then:
\begin{equation*}
\lim_{n \to + \infty} \frac{1}{n^{\frac{1}{4}}} \Ycal_n 
= \sigma(f) \Ycal,
\end{equation*}
where the convergence is strong in distribution on $(\Omega, \Leb)$, the random variable 
$\Ycal$ follows standard $MLGM(1/2)$ distribution, and:
\begin{equation*}
\sigma (f)^2 = \sqrt{\frac{2}{\pi \Sigma_{1,1}}} \sum_{k\in\Z} \sum_{a,b\in\Z^2} \beta(a)\beta(b) \tilde{\nu} (S_k=a-b|\Ccal_0).
\end{equation*}
In addition, $\sigma(f) = 0$ if and only if $f$ is a coboundary.
\end{corollary}

\begin{corollary}\label{coro:bill2}

With the above notations, assume that:
\begin{itemize}
\item there exists $\varepsilon >0$ such that $\sum_{p \in \Z^2} |p|^\varepsilon |\beta|(p) < + \infty$;
\item $\sum_{p \in \Z^2} \beta(p) =0$.
\end{itemize}

\smallskip

Then:
\begin{equation*}
\lim_{n \to + \infty} \frac{1}{\sqrt{\log (n)}} \Ycal_n 
= \sigma(f) \Ycal,
\end{equation*}
where the convergence is strong in distribution on $(\Omega, \Leb)$, the random variable 
$\Ycal$ follows a Laplace distribution of parameter $1/\sqrt{2}$, and:
\begin{equation*}
\sigma (f)^2 = \frac{1}{2\pi \sqrt{\det (\Sigma)}} \sum_{k\in\Z} \sum_{a,b\in\Z^2} \beta(a)\beta(b) \tilde{\nu} (S_k=a-b|\Ccal_0).
\end{equation*}
In addition, $\sigma(f) = 0$ if and only if $f$ is a coboundary.
\end{corollary}

\subsubsection{Geodesic flow on abelian covers}
\label{subsubsec:IntroFlotGeodesique}

The geodesic flow on a connected, compact manifold with negative sectional curvature is a well-known 
example of a hyperbolic dynamical system. The geodesic flow on abelian covers of such manifolds 
provides a class of dynamical systems which preserve a $\sigma$-finite measure, for instance the Liouville measure. 
They are also more tractable than billiards, as they do not have singularities. These geodesic flows 
have been studied by extensively, for instance to count periodic orbits on the basis manifold of given length 
in a given homology class~\cite{KatsudaSunada:1990, Sharp:1993, PollicottSharp:1994}. 
There are extensions to Anosov flows~\cite{BabillotLedrappier:1998} as well as to surfaces with cusps 
\cite{AaronsonDenker:1999}. Finally, let us mention that the geodesic flow on periodic manifolds is also used to 
study the horocycle flow on the same manifolds \cite{BabillotLedrappier:1996, LedrappierSarig:2006, LedrappierSarig:2007, LedrappierSarig:2008}.

\smallskip

Limit theorems for observables with integral zero have already 
been obtained in this context~\cite{Thomine:2015}, but the improvement we get with Theorem~\ref{thm:GM} 
translates into a limit theorem which is valid for a wider class of observables. Instead 
of having compact support, the observables need only to decay polynomially fast at infinity.

\smallskip

Let $M$ be a compact, connected manifold with a Riemannian metric of 
negative sectional curvature. Let $\varpi:N \to M$ be a connected $\Z^d$-cover of $M$. 
Given a Gibbs measure $\mu_M$ on $T^1 M$, we endow $T^1 N$ with a $\sigma$-finite 
measure $\mu_N$ by lifting $\mu_M$ locally. 
We refer the reader to~\cite[Chapter~11.6]{PaulinPollicottSchapira:2015} 
for more details about Gibbs measures in this context.

\smallskip

Let $(g_t)_{t \in \R}$ be the geodesic flow on $T^1 N$. 
In Subsection~\ref{subsec:ApplicationFlotGeodesique}, we shall prove the following proposition, 
which is a generalization of~\cite[Proposition~6.12]{Thomine:2015}:

\begin{proposition}
\label{prop:FlotGeodesique}

Let $\mu_N$ be the lift of a Gibbs measure $\mu_M$ corresponding to a H\"older potential. 
Assume that the extension $(N, (g_t), \mu_N)$ is ergodic and recurrent. 
Fix $x_0 \in T^1 N$. Let $f$ be a real-valued H\"older function on $T^1 N$. Assume that:
\begin{itemize}
\item there exists $\varepsilon >0$ such that $\int_{T^1 N} d(x_0,x)^{1-\frac{d}{2}+\varepsilon} |f|(x) \dd \mu_N (x) < + \infty$;
\item $\int_{T^1 N} f \dd \mu_N = 0$.
\end{itemize}

\smallskip

If $d=1$, there exists $\sigma(f) \geq 0$ such that:
\begin{equation*}
\lim_{t \to + \infty} \frac{1}{t^{\frac{1}{4}}} \int_0^t f \circ g_s (x,v) \dd s 
= \sigma(f) \Ycal_{1/2},
\end{equation*}
where the convergence is strong in distribution on $(T^1 N, \mu_N)$, and 
$\Ycal_{1/2}$ follows a standard $MLGM (1/2)$ distribution.

\smallskip

If $d=2$, there exists $\sigma(f) \geq 0$ such that:
\begin{equation*}
\lim_{t \to + \infty} \frac{1}{\sqrt{\log (t)}} \int_0^t f \circ g_s (x,v) \dd s 
= \sigma(f) \Ycal_0,
\end{equation*}
where the convergence is strong in distribution on $(T^1 N, \mu_N)$, and 
$\Ycal_0$ follows a standard $MLGM (0)$ distribution.

\smallskip

In both cases, $\sigma(f) = 0$ if and only if $f$ is a measurable coboundary.
\end{proposition}

\begin{remark}[Recurrent extensions of Gibbs measures]

Given a H\"older potential $F : T^1 M \to \C$, let 
$\hat{F} (x,v) := F(x,-v)$. We say that the potential is \textit{reversible} 
if $F$ and $\hat{F}$ are cohomologuous, that is, if there exists a 
H\"older function $u$ such that $\int_0^t f \circ g_s \dd s = u \circ g_t -u$ 
for all $t$. In this case, we also say that $F-\hat{F}$ is a H\"older coboundary. 
We say that a Gibbs measure is \textit{reversible} if it associated to 
a reversible potential.

\smallskip

For instance, both the Liouville measure and the maximal entropy measure are 
reversible, because the associated potentials (constants for the maximal entropy measure, 
and the log-Jacobian of the flow restricted to the unstable direction for the Liouville 
measure) are reversible. 

\smallskip

If $\mu_M$ is a reversible Gibbs measure and $d \in \{1, 2\}$, then the geodesic flow on $(T^1 N, \mu_N)$ 
is both ergodic and recurrent (see~\cite{Rees:1981} for the constant curvature case, although the proof 
works as well in variable curvature).
\end{remark}

The only difference between Proposition~\ref{prop:FlotGeodesique} and~\cite[Proposition~6.12]{Thomine:2015} 
is that the assumption that $f$ has compact support is relaxed to 
$\int_{T^1 N} d(x_0,x)^{1-\frac{d}{2}+\varepsilon} |f|(x) \dd \mu_N (x) < + \infty$ for some $\varepsilon >0$.

\smallskip

Note that our work gives us more information on this system; for instance, 
Theorem~\ref{thm:ConvergenceNp} can be adapted to yield an asymptotic equivalence of the probability that, 
starting from some nice Poincar\'e section $A_1$, the geodesic flow reaches a faraway Poincar\'e 
section $A_2$ before returning to $A_1$. However, the geometric interpretation 
of these sections is less evident than for others systems, such as billiards.

\section{Theorem~\ref{thm:gene}: assumptions and proof}
\label{sec:operators}

This Section is mostly devoted to the proof of Theorem~\ref{thm:gene}. 
It is organized as follows. The spectral hypotheses are presented in Subsection~\ref{subsec:HypothesesSpectrales}. 
The following three subsections contain respectively a sketch of the proof of Theorem~\ref{thm:gene}, 
the full proof of the theorem, and a proof of the more technical estimates we use. Finally, 
in Subsection~\ref{subsec:renewal} we prove Proposition~\ref{prop:renouvellement}.

\subsection{General spectral assumptions}
\label{subsec:HypothesesSpectrales}

Let $P$ be the transfer operator associated to $h \mapsto h \circ T$, that is,
\begin{equation*}
\int_A P f \cdot g \dd \mu 
= \int_A f \cdot g \circ T \dd \mu 
\quad \forall f \in \Lbb^1 (A,\mu), \ \forall g \in \Lbb^\infty (A,\mu).
\end{equation*}
We consider the family $(P_u)_u$ of operators defined by $P_u: h \mapsto P (e^{i \langle u, F \rangle} h)$, 
where $\langle \cdot, \cdot \rangle$ stands for the usual scalar product in $\R^d$. Note that:
\begin{equation}\label{Puk}
P_u^k (h)
= P^k \left(e^{i \langle u, S_k \rangle} h \right).
\end{equation}

We make the following assumptions. Thanks to perturbation theorems (see namely \cite{Nagaev:1957,Nagaev:1961,GuivarchHardy:1988,KellerLiverani:1999,HennionHerve:2001} 
for the general method, and~\cite{AaronsonDenker:2001} for an application to Gibbs-Markov maps), 
they hold for a wide variety of hyperbolic dynamical systems.

\begin{hypothesis}[Spectral hypotheses]\label{hyp:HHH}
The stochastic process $(S_n)_n$ is recurrent. There exists an integer $M\geq 1$ and a $\mu$-essential partition of $A$
in $M$ measurable subsets $(A_j)_{j \in \Z / M\Z}$ such that $T(A_j)=A_{j+1}$
for all $j\in\Z/M\Z$ ($M=1$ if $T$ is mixing).

\smallskip

There exists a complex Banach space $(\Bcal,\norm{\cdot}{\Bcal})$ of functions defined on $A$, on which 
$P$ acts continuously, and such that:

\begin{itemize}
\item $\Bcal \hookrightarrow \Lbb^1 (A, \mu)$.
\item $\mathbf{1} \in \Bcal$ and for every $j$, the multiplication by $\mathbf{1}_{A_j}$ belongs to $\Lcal(\Bcal,\Bcal)$, 
where $\left(\Lcal(\Bcal,\Bcal),\norm{\cdot}{\Lcal(\Bcal,\Bcal)}\right)$ stands for the Banach space of linear continuous endomorphisms of $\Bcal$. 
\item There exist a neighbourhood $U$ of $0$ in $\Tbb^d$, two constants $C>0$ and $r \in(0,1)$, two
continuous functions $\lambda_\cdot:U \to \C$ and $\Pi_\cdot:U \to \Lcal(\Bcal,\Bcal)$ 
such that
\begin{equation}\label{eq:decomp}
P_u =\lambda_u\Pi_u+R_u,
\end{equation}
with:
\begin{align}
\Pi_u R_u     & = R_u \Pi_u = 0,							     \label{eq:PiuRu} \\
\Pi_u^{M+1} & = \Pi_{u}, 								     \label{eq:PiM+1} \\
\lambda_0     & = 1,       								     \label{eq:lambda0} \\
\Pi_0         & = M \sum_{j\in \Z/M\Z} \Ebb_\mu [\mathbf{1}_{A_j}\cdot]\mathbf{1}_{A_{j+1}}, \label{eq:Pi0} \\
\sup_{u\in U} \norm{R_u^k}{\Lcal(\Bcal,\Bcal)} & \leq C r^k,				     \label{eq:majoRu} \\
\sup_{u\in[-\pi,\pi]^d \setminus U} \norm{P_u^k}{\Lcal(\Bcal,\Bcal)} & \leq C r^k. 	     \label{eq:aperiodicity}
\end{align}
\item If $d=1$, there exists $\alpha \in [1,2]$ such that, for all $u \in U$,
\begin{equation*}
\lambda_u
= e^{-\psi(u)L(|u|^{-1})} + o\left(|u|^\alpha L(|u|^{-1})\right),
\end{equation*}
as $u$ goes to $0$, where $\psi(u)=\vartheta|u|^\alpha[1-i\zeta\sgn(u)]$ for some 
real numbers $\vartheta>0$ and $\zeta\in\R$ such that $|\zeta| \leq \tan\frac{\pi\alpha}2$ if $\alpha>1$. 
We set $\Sigma:=1$.
\item If $d=2$, there exists an invertible positive symmetric matrix
$\Sigma$ such that, for all $u \in U$,
\begin{equation*}
\lambda_u
= e^{-\psi(\sqrt{\Sigma} u)L(|\sqrt{\Sigma} u|^{-1})} + o\left(|u|^2 L(|u|^{-1})\right),
\end{equation*}
as $u$ goes to $0$, where $\psi(u):=\frac{|u|^2}{2}$ and $L$ is slowly varying at infinity. We set $\vartheta:=1/2$.
\end{itemize}
\end{hypothesis}

Hypothesis~\ref{hyp:HHH} implies the ergodicity of $T$ and the mixing of $(T^M)_{|A_j}$ for all $j \in \Z/M\Z$ 
as soon as $\Bcal$ is dense in $\Lbb^1 (A, \mu)$. If the system is not mixing, then it is expected that the transfer 
operators has multiple eigenvalues of modulus $1$. The following proposition asserts that, in this case, the standard 
spectral techniques yield a decomposition as in Equation~\eqref{eq:decomp}.

\begin{proposition}\label{prop:CSHHH}
Assume the begining of Hypothesis~\ref{hyp:HHH} and its first two items, 
and that $(A_0,\mu(\cdot|A_0), (T^M)_{|A_0})$ is mixing. Assume in addition that 
there exist a neighbourhood $U$ of $0$ in $\Tbb^d$, two constants $C>0$ and $r \in(0,1)$ and continuous functions
$\tilde\lambda_\cdot,\lambda_{0,\cdot},\ldots,\lambda_{K-1,\cdot}: U \to \C$ and $\tilde{\Pi}_\cdot, \Pi_{0,\cdot},\ldots,\Pi_{K-1,\cdot}, \tilde{R}_\cdot, R_\cdot:U \to \Lcal(\Bcal,\Bcal)$ 
such that:
\begin{align*}
P_u & = \sum_{j\in\Z/K\Z}\lambda_{j,u} \Pi_{j,u}+R_u, \\
\Pi_{j,u} R_u & = R_u P_{j,u} = 0, \\
\Pi_{j,u} \Pi_{j',u} & = \delta_{j,j'}\Pi_{j,u}, \\
|\lambda_{j,0}| & = 1 \\
\sup_{u\in U} \norm{R_u^k}{\Lcal(\Bcal,\Bcal)} & \leq C r^k,
\end{align*}
and $\mathbf{1}_{A_0} P_u^M (\mathbf{1}_{A_0}) = \tilde{\lambda}_u \tilde{\Pi}_u + \tilde{R}_u$, with:
\begin{align*}
\tilde{\lambda}_0 & = 1, \\
\tilde{\Pi}_u^2 & = \tilde{\Pi}_u, \\
\tilde{\Pi}_0 & = \mu(\cdot | A_0) \mathbf{1}_{A_0}, \\
\tilde{\Pi}_u \tilde{R}_u & = \tilde{R}_u \tilde{\Pi}_u = 0, \\
\norm{\tilde R_u^k}{\Lcal(\Bcal,\Bcal)} & \leq C r^k.
\end{align*}

Then $P_u =\lambda_u\Pi_u+R_u$, and the equations \eqref{eq:PiuRu}, \eqref{eq:PiM+1}, \eqref{eq:lambda0}, \eqref{eq:Pi0}, \eqref{eq:majoRu} 
are satisfied. If moreover $u \mapsto P_u$ is continuous on $\Tbb^d$ and $P_u$ admits no eigenvalue of modulus $1$ 
for $u\neq 0$, then Equation~\eqref{eq:aperiodicity} is also satisfied, up to increase of $C>0$ and $r\in(0,1)$.
\end{proposition}

\begin{proof}
Up to taking a smaller $U$, we assume that $|\lambda_{j,\cdot}|>C^{1/M} r$ and $|\tilde{\lambda}_\cdot|>Cr^M$.
Then $\tilde{\lambda}_u=\lambda_{j,u}^M$ for every $j \in \Z/M\Z$, and $\tilde{\Pi}_u=\sum_{j \in \Z/M\Z} \mathbf{1}_{A_0} \Pi_{j,u} (\mathbf{1}_{A_0} \cdot)$.
Hence we can take $K=M$ and, up to a permutation of indices, we assume 
that $\lambda_{j,u}=\lambda_u \xi^j$ with $\xi:=e^{2i\pi /M}$ and $\lambda_0=1$
($P\mathbf{1}=\mathbf{1}$ ensures that $1$ is an eigenvalue of $P_0$, and this convention yields Equation~\eqref{eq:lambda0}). 
Hence $P_u=\lambda_u\Pi_u+R_u$, with:
\begin{align*}
\lambda_u & :=\lambda_{0,u}, \\
\Pi_u & :=\sum_{j\in\Z/M\Z} \xi^j \Pi_{j,u}.
\end{align*}
Note that $\Pi_u R_u=R_u\Pi_u=0$ and that $\Pi_u^{M+1}= \sum_{j\in\Z/M\Z}\xi^{j(M+1)} \Pi_{j,u}=\Pi_u$, 
which proves equations~\eqref{eq:PiuRu} and~\eqref{eq:PiM+1}.
In the general case, it remains to prove \eqref{eq:Pi0}.

\smallskip

Let $f$ be an eigenvector for the eigenvalue $\xi^j$ of $P$. 
For all $k \in \Z/M\Z$, 
\begin{equation}
\label{eq:FonctionProprePeriodique}
P(\mathbf{1}_{A_k}f) 
= \xi^j \mathbf{1}_{A_{k+1}}f,
\end{equation}
so that $P^M(\mathbf{1}_{A_k} f)= \mathbf{1}_{A_k} f$. Since 
$T^M$ is mixing, $f$ must be constant on each $A_k$; using Equation~\eqref{eq:FonctionProprePeriodique}, 
we get that $f$ is proportional to $\sum_{k\in \Z/M\Z}\xi^{-jk}\mathbf{1}_{A_k}$. 
We conclude that:
\begin{equation*}
\Pi_{j,0} 
= \sum_{k\in \Z/M\Z}\xi^{-jk} \mathbf{1}_{A_k} \Ebb_\mu \left[\sum_{\ell \in \Z/M\Z} \xi^{j\ell} \mathbf{1}_{A_\ell}  \cdot\right],
\end{equation*}
and from there that $\Pi_0 = M \sum_{j\in\Z/M\Z}\xi^j \Pi_{j,0} =
\sum_{k \in \Z/M\Z} \mathbf{1}_{A_{k+1}} \Ebb_\mu[ \mathbf{1}_{A_k} \cdot]$.

\smallskip

Finally, Equation~\eqref{eq:aperiodicity} comes from~\cite{AaronsonDenker:2001}.
\end{proof}

For every $n$, we set 
\begin{equation}
\mathfrak{a}_n:= \inf\{x>0\, :\, n |x|^{-\alpha} L(x) \geq 1\},
\end{equation}
so that $nL(\mathfrak{a}_n)\sim \mathfrak{a}_n^\alpha$. The sequence $(\mathfrak{a}_n)$ is then regularly varying of index $\frac{1}{\alpha}$. 
Under Hypothesis~\ref{hyp:HHH}, $\Ebb_\mu [e^{i\langle t,S_n\rangle/\mathfrak{a}_n}] \sim (\lambda_{t/\mathfrak{a}_n})^n \sim e^{-\psi(\sqrt{\Sigma}t)}$ for every $t\in\R^d$. 
Thus, the sequence $(S_n/\mathfrak{a}_n)_n$ converges in distribution to an $\alpha$-stable 
random variable with characteristic function $e^{-\psi(\sqrt{\Sigma}\cdot)}$.

\subsection{Strategy of the proof}
\label{subsec:StrategieThmGene}

Given the length of the proof and the technicality of some of its parts, 
we give here a brief outline of how the method of moments can be applied 
to our problem.

\smallskip

The proof consists in showing the convergence for every $m$
of the following quantity:
\begin{align*}
\Ebb_\mu \left[ \left(\frac{\Zcal_n}{\Afrak_n}\right)^n \right]
& = \Ebb_\mu \left[\left(\frac{\sum_{k=1}^n \sum_{a\in\Z^d} \beta(a) \mathbf{1}_{\{S_k=a\}}}{\Afrak_n}\right)^m\right] \\
& = \Afrak_n^{-m} \sum_{k_1,\ldots,k_m=1}^n \beta(a_1) \cdots \beta(a_m) \mu (S_{k_1}=a_1,\ldots,S_{k_q}=a_q).
\end{align*}
Hence we have to deal with quantities of the following form:
\begin{equation*}
\sum_{1 \leq k_1 < \cdots <k_q \leq n} \sum_{a_1,\ldots, a_q \in \Z^d} \beta(a_1)^{N_1} \cdots \beta(a_q)^{N_q} \mu (S_{k_1}=a_1,\ldots,S_{k_q}=a_q), 
\end{equation*}
where $N_1+\ldots+N_q=m$. Let us write $A_{n;q;N_1,\ldots,N_q}$ for this quantity, which behaves roughly as:
\begin{equation*}
\sum_{1 \leq k_1 < \cdots<k_q \leq n} \sum_{a_1, \ldots, a_q \in \Z^d} \beta(a_1)^{N_1} \cdots \beta(a_q)^{N_q} \mu(S_{k_1}=a_1) \mu (S_{k_2-k_1}=a_2-a_1) \cdots \mu (S_{k_q-k_{q-1}}=a_q-a_{q-1}).
\end{equation*}
This equation would actually be exact if $(S_n)_n$ were a random walk. Then, put $k_0 := 0$ and $\ell_i := k_i-k_i-1$, so that:
\begin{equation*}
A_{n;q;N_1,\ldots,N_q}
\sim \sum_{\ell_1+\ldots+\ell_q \leq n} \sum_{a_1,\ldots,a_q \in \Z^d} \prod_{i=1}^q \left(\beta(a_i)^{N_i}  \mu(S_{\ell_i}=a_i-a_{i-1})\right).
\end{equation*}

\smallskip

We prove that
\begin{equation*}
A_{n;q;N_1,\ldots,N_q}
= O(\Afrak_n^m)
\end{equation*}
and even that
\begin{equation*}
A_{n;q;N_1,\ldots,N_q}
= o(\Afrak_n^m)
\end{equation*}
except if $(N_1,\ldots,N_q)$ is made of $2$s and of pairs of consecutive $1$s and of nothing else, 
which implies that $m$ is even. In particular, for all odd $m$,
\begin{equation*}
\Ebb_\mu [\Zcal_n^m] 
= o(\Afrak_n^m).
\end{equation*}
This is the content of Lemma~\ref{lem:normeB}, which is by far the most technical part 
of our proof. This is also the point where we use the fact that $\beta$ has zero sum; otherwise, 
we would get $A_{n;q;N_1,\ldots,N_q} = \Theta (\Afrak_n^{2m})$ for $(N_1, \ldots, N_q)=(1, \dots, 1)$.

\smallskip

If $(N_1,\ldots,N_q)$ is made of $2$s and disjoint pairs of consecutive $1$, 
then it contains $(m-q)$ times the value $2$ and $(q-m/2)$ pairs $(N_i,N_{i+1})=(1,1)$. 
Then, we shall prove that:
\begin{align*}
A_{n;q;N_1,\ldots,N_q}
& \sim \sum_{\ell_1+\ldots+\ell_{m/2}\leq n} \prod_{i=1}^{m-q} \left(\sum_{a \in \Z^d} \beta(a)^2 \mu(S_{\ell_i}=a) \right) \\
& \hspace{2em} \times \prod_{i=m-q+1}^{m/2} \left(\sum_{\ell \geq 1} \sum_{a,b \in \Z^d} \beta(a) \beta(b) \mu (S_\ell=b-a) \mu (S_{\ell_i}=a)\right)\\
& \sim \sum_{\ell_1+\ldots+\ell_{m/2} \leq n} \prod_{i=1}^{m-q} \left(\sum_{a \in \Z^d} \beta(a)^2 c \mathfrak{a}_{\ell_i}^{-d}\right) \prod_{i=m-q+1}^{m/2} \left(\sum_{\ell\geq 1} \sum_{a,b \in \Z^d} \beta(a)\beta(b) \mu(S_\ell=b-a) c \mathfrak{a}_{\ell_i}^{-d}\right)\\
& \sim c^{m/2} \left(\sum_{a\in\Z^d} \beta(a)^2 \right)^{m-q} \left(\sum_{\ell\ge 1} \sum_{a,b\in\Z^d} \beta(a)\beta(b) \mu(S_\ell=b-a)\right)^{q-m/2} \sum_{\ell_1+\ldots+\ell_{m/2}\leq n} \prod_{i=1}^{m/2}\mathfrak{a}_{\ell_i}^{-d} \\
& \sim K_{m,q} \left(\sum_{\ell=1}^n \mathfrak{a}_\ell^{-d}\right)^{m/2}
 = K_{m,q} \Afrak_n^m,
\end{align*}
where the constants $c$ and $K_{m,q}$ are explicit and yield the $MLGM$ random variables.

\subsection{Proof of Theorem~\ref{thm:gene}}
\label{subsec:PreuveGene}

In this section we prove Theorem~\ref{thm:gene}. To prove the strong convergence in distribution, 
it is actually sufficient to prove the convergence in distribution with 
respect to some absolutely continuous probability measure \cite[Theorem~1]{Zweimuller:2007}. At first, we prove the 
convergence of $(S_n^{\widetilde{T}} f /\Afrak_n)_n$ under the probability measure $\mu_0 := \mu \otimes \delta_0$, 
i.e.\ the convergence of $(\Zcal_n (\beta)/\Afrak_n)_n$ under the probability measure $\mu$, where:
\begin{equation*}
\Zcal_n (\beta)
= \sum_{k=0}^{n-1} \beta (S_n^T F).
\end{equation*}

\smallskip

We use the method of moments. Let $m \geq 0$ be an integer, which is fixed for the remainder of this proof. Then, for all $n$:
\begin{align*}
\Ebb_\mu \left[ \Zcal_n(\beta)^m \right] 
& = \Ebb_\mu \left[ \left(\sum_{k=1}^n \beta (S_k) \right)^m \right] \\
& = \sum_{k_1,\ldots,k_m = 1}^n \sum_{d_1,\ldots,d_m \in \Z^d} \Ebb_\mu \left[ \prod_{s=1}^m \beta (d_s) \mathbf{1}_{\{S_{k_s}=d_s\}} \right].
\end{align*}
We delete the terms which are null, and regroup those which are equal. Let us consider one of the terms 
$\prod_{s=1}^m \beta (d_s)\mathbf{1}_{\{S_{k_s}=d_s\}}  $. We may assume that $d_s = d_{s'}$ as soon as $k_s = k_{s'}$; 
otherwise, $\mathbf{1}_{\{S_{k_s}=d_s\}} \mathbf{1}_{\{S_{k_{s'}}=d_{s'}\}} = 0$ and the whole product is zero.
Let $q:=\#\{k_1,\ldots, k_m\}$. Then $\{k_1,\ldots, k_m\}=\{n_1,\ldots,n_q\}$ with $1 \leq n_1 < \ldots < n_q \leq n$. 
We set $N_j:=\#\{i=1,\ldots,m\ :\ k_i=n_j\}$ for the multiplicity of $n_j$ in $(k_1,\ldots,k_m)$, and $a_j=d_i$ if $k_i=n_j$.
We write $\mathbf{a}:=(a_1,\ldots,a_q)$,  $\mathbf{N}:=(N_1,\ldots,N_q)$ and $\mathbf{n}:=(n_1,\ldots,n_q)$, and set, by convention, 
$n_0 := 0$ and $a_0 := 0$. Observe that
\begin{equation*}
\prod_{s=1}^m \beta (d_s) \mathbf{1}_{\{S_{k_s}=d_s\}}  
=\prod_{j=1}^q \beta(a_j)^{N_j}\mathbf{1}_{\{S_{n_j}=a_j\}}
\end{equation*}
and that the number of $m$-uplets $(k_1,\ldots,k_m)$ giving the same pair $(\mathbf{n},\mathbf{N})$ 
is equal to the number $c_{\mathbf{N}}$ of maps $\phi: \ \{1,\ldots,m\} \to \{1,\ldots,q\}$ 
such that $| \phi^{-1} (\{j\}) | = N_j$ for all $j\in\{1,\ldots,q\}$. Hence:
\begin{equation*}
\Ebb_\mu \left[ \Zcal_n (\beta)^m \right] 
= \sum_{q=1}^m \sum_{\substack{N_j \geq 1 \\ N_1+\ldots+N_q=m}} c_{\mathbf{N}} \sum_{1 \leq n_1 < \ldots < n_q \leq n} \sum_{\mathbf{a} \in (\Z^d)^q} \Ebb_\mu \left[ \prod_{j=1}^q \left( \beta (a_j)^{N_j} \mathbf{1}_{\{S_{n_j}=a_j\}} \right) \right].
\end{equation*}
For all $n \geq 1$, for all $1 \leq q \leq m$ and for all $\mathbf{N}=(N_j)_{1 \leq j \leq q}$ such that $N_j \geq 1$ and $\sum_{j=1}^q N_j = m$, we define:
\begin{align*}
A_{n;q;\mathbf{N}} 
& := \sum_{1 \leq n_1< \ldots <n_q \leq n} \sum_{\mathbf{a} \in (\Z^d)^q} \Ebb_\mu \left[ \prod_{j=1}^q \left( \beta (a_j)^{N_j} \mathbf{1}_{\{S_{n_j}=a_j\}} \right) \right] \\
& = \sum_{1 \leq n_1< \ldots <n_q \leq n} \sum_{\mathbf{a} \in (\Z^d)^q} \Ebb_\mu \left[ \prod_{j=1}^q \left(\beta (a_j)^{N_j} \mathbf{1}_{\{S_{n_j}-S_{n_{j-1}}=a_j-a_{j-1}\}} \right) \right],
\end{align*}
so that:
\begin{equation}\label{momentAn}
\Ebb_\mu \left[ \Zcal_n (\beta)^m \right] 
= \sum_{q=1}^m \sum_{\substack{N_j \geq 1 \\ N_1+\cdots+N_q=m}} c_{\mathbf{N}} A_{n;q,\mathbf{N}}.
\end{equation}

Instead of working with a sequence of times $(n_j)$ and positions $(a_j)$, it shall be more convenient to work with time 
increments and position increments. Let $1 \leq n_1 < \ldots < n_q \leq n$. We can describe this sequence with integers 
$(\ell_1, \dots,\ell_q)$ by taking $\ell_1 = n_1$ and $\ell_j = n_j - n_{j-1}$ for all $2 \leq j \leq q$. Let $E_{q,n}$ be the set defined by:
\begin{equation*}
E_{q,n}
= \left\{ \boldsymbol{\ell}=({\ell}_1,\ldots,\ell_q) \in \{1, \ldots ,n\}^q\ :\ \sum_{j=1}^q {\ell}_j \leq n \right\}.
\end{equation*}
Then summing over all $\mathbf{n}=(n_1,\ldots,n_q)$ such that $1 \leq n_1 < \ldots <n_q \leq n$ is the same as summing over all $\boldsymbol{\ell}$ in $E_{q,n}$, whence:
\begin{equation}
\label{eq:AExpressionEll}
A_{n;q;\mathbf{N}} 
= \sum_{\mathbf{a} \in (\Z^d)^q} \left[ \left( \prod_{j=1}^q \beta (a_j)^{N_j} \right) \sum_{\boldsymbol{\ell} \in E_{q,n}} \Ebb_\mu \left[ \prod_{j=1}^q \mathbf{1}_{\{S_{\ell_j} = a_j-a_{j-1}\}} \circ T^{n_{j-1}} \right] \right].
\end{equation}

A single coefficient $A_{n;q;\mathbf{N}}$ is the contribution to the $m$th moment of $\Zcal_n (\beta)$ 
by paths of length $q$ and with weights $(N_j)$. Our goal is to find a sub-family 
of such weighted paths which is manageable enough so that we can estimate the behaviour of the 
$A_{n;q;\mathbf{N}}$, and large enough so that it makes for almost all $\Ebb_\mu \left[ \Zcal_n (\beta)^m \right]$ 
as $n$ goes to infinity. However, in order to benefit from the fact that $\sum_{a\in\Z^d}\beta(a) = 0$, 
we use transfer operators, and a decomposition which leverages this equality to make some further 
simplifications\footnote{If $\beta$ has non-zero integral, different terms dominate, and the moments 
grow faster. It is thus essential to cancel out these ``first order terms''.}.

\smallskip

For all $\ell \in \N$ and $a \in \Z^d$, we define an operator $Q_{\ell,a}$ acting on $\Bcal$ by:
\begin{equation*}
Q_{\ell,a} (h)
:=P^\ell \left(\mathbf{1}_{\{S_\ell=a\}}\, h\right)
=\frac{1}{(2 \pi)^d} \int_{[-\pi,\pi]^d} e^{-i \langle u, a \rangle} P_u^\ell (h) \dd u \, .
\end{equation*}
 where we used \eqref{Puk}
to establish the second formula. 
For $1 \leq q' \leq q$, we write:
\begin{equation*}
D_{q'} 
:= \prod_{j=1}^{q'} \left(\mathbf{1}_{\{ S_{\ell_j}=a_j-a_{j-1} \}} \circ T^{n_{j-1}} \right).
\end{equation*}
Recall that $P^k (g \circ T^k \cdot h) = g P^k (h)$. 
Hence, by induction,
\begin{align*}
P^{n_q} \left( D_q \right)
& = P^{n_q} \left( \mathbf{1}_{\{ S_{\ell_q}=a_q-a_{q-1} \}} \circ T^{n_{q-1}} \cdot D_{q-1} \right) \\
& = P^{n_q-n_{q-1}}(\mathbf{1}_{\{S_{\ell_q}=a_q-a_{q-1}\}} P^{n_{q-1}}(D_{q-1}))\\
& = Q_{\ell_q, a_q - a_{q-1}} \left( P^{n_{q-1}} \left( D_{q-1} \right) \right) \\
& = \ldots \\
& = Q_{\ell_q, a_q-a_{q-1}} \cdots Q_{\ell_1, a_1-a_0} (\mathbf{1}).
\end{align*}
Plugging this identity into Equation~\eqref{eq:AExpressionEll} yields: 
\begin{equation}\label{eq:AExpressionQ}
A_{n;q;\mathbf{N}} 
= \sum_{\mathbf{a} \in(\Z^d)^q} \left[ \left( \prod_{j=1}^q \beta (a_j)^{N_j} \right) \sum_{\boldsymbol{\ell} \in E_{q,n}} \Ebb_\mu \left[ Q_{\ell_q,a_q-a_{q-1}} \cdots Q_{\ell_1,a_1-a_0}(\mathbf{1}) \right] \right].
\end{equation}

We further split the operators $Q_{\ell,a}$. Let us write:
\begin{equation}\label{devQ}
Q_{\ell,a} = Q_{\ell,a}^{(0)} + Q_{\ell,a}^{(1)},
\end{equation}
with:
\begin{align*}
Q_{\ell,a}^{(0)} & := \Phi(0)\frac{\Pi_0^\ell}{\mathfrak{a}_\ell^d} \\
Q_{\ell,a}^{(1)} & = \varepsilon_{\ell,a} + \frac{\Phi(a/\mathfrak{a}_\ell)-\Phi(0)}{\mathfrak{a}_\ell^d} 
\Pi_0^\ell \quad \mbox{with}\ \norm{\varepsilon_{\ell,a}}{}=o(\mathfrak{a}_\ell^{-d}),
\end{align*}
which we know is possible thanks to Lemma~\ref{lem:0}. 

\smallskip

We introduce these operators $Q_{\ell,a}^{(0)}$ and $Q_{\ell,a}^{(1)}$ 
into~\eqref{eq:AExpressionQ}, creating new data we need to track: 
the index of the operator we use at each point in the weighted path. 
Fix $n$, $q$ and $\mathbf{N}$. Given $\boldsymbol{\varepsilon} = (\varepsilon_1, \ldots, \varepsilon_q) \in \{0,1\}^q$ and $s \in \Z^d$, write:
\begin{align*}
B_{s,\boldsymbol{\ell}, \mathbf{N}}^{\boldsymbol{\varepsilon}}
& := \sum_{\substack{a_0,\ldots,a_q \in \Z^d \\ a_0=s}} \left( \prod_{i=1}^q \beta (a_i)^{N_i} \right) Q_{\ell_q,a_q-a_{q-1}}^{(\varepsilon_q)} \cdots Q_{\ell_1,a_1-a_0}^{(\varepsilon_1)}, \\
b_{s, \boldsymbol{\ell}, \mathbf{N}}^{\boldsymbol{\varepsilon}} (\cdot)
& := \sum_{\mathbf{a} \in (\Z^d)^q} \left[ \left( \prod_{i=1}^q \beta (a_i)^{N_i} \right) \Ebb_\mu \left[ Q_{\ell_q,a_q-a_{q-1}}^{(\varepsilon_q)} \ldots Q_{\ell_1,a_1-s}^{(\varepsilon_1)} (\cdot) \right] \right]
= \Ebb_\mu \left[B_{s,\boldsymbol{\ell},\mathbf{N}}^{\boldsymbol{\varepsilon}} (\cdot)\right], \\
A_{n;q;\mathbf{N}}^{\boldsymbol{\varepsilon}} 
& := \sum_{\boldsymbol{\ell} \in E_{q,n}} b_{0,\boldsymbol{\ell}, \mathbf{N}}^{\boldsymbol{\varepsilon}}(\mathbf{1}),
\end{align*}
so that:
\begin{equation*}
A_{n;q;\mathbf{N}} 
= \sum_{\boldsymbol{\varepsilon}\in\{0,1\}^q} A_{n;q;\mathbf{N}}^{\boldsymbol{\varepsilon}}
= \sum_{\boldsymbol{\varepsilon}\in\{0,1\}^q} \sum_{\boldsymbol{\ell} \in E_{q,n}} b_{0, \boldsymbol{\ell}, \mathbf{N}}^{\boldsymbol{\varepsilon}}(\mathbf{1}).
\end{equation*}

\smallskip

Now, the main question is: for which data $(q,\mathbf{N},\boldsymbol{\varepsilon})$ do the coefficients $A_{n;q;\mathbf{N}}^{\boldsymbol{\varepsilon}}$, 
seen as functions of $n$, grow the fastest? One would want to use the larger operator $Q_{\ell, a}^{(0)}$ whenever possible, 
and to use the lowest possible weights whenever possible (because lower weights means larger value of $q$, so a faster combinatorial growth). 
\textit{A priori}, the best possible choice would be $\boldsymbol{\varepsilon} = (0, \ldots, 0)$ and $\mathbf{N} = (1, \ldots, 1)$. 
That is indeed true for observables $\beta$ with non-zero integral. However, in our case, the fact that $\sum_{a \in \Z^d} \beta (a) = 0$ 
induces a cancellation, which makes the corresponding coefficient vanish. This can be seen with the following elementary properties.

\begin{properties}\label{rmk:proprietesbn}

Consider a single linear form $b_{s, \boldsymbol{\ell}, \mathbf{N}}^{\boldsymbol{\varepsilon}}$. 
For all $1 \leq i \leq q$, the terms on the right side of $Q_{\ell_i,a_i-a_{i-1}}^{(\varepsilon_i)}$ depend only 
on $a_1,\ldots,a_{i-1}$, and the terms on its left side only depend on $a_i,\ldots,a_q$. Hence:
\begin{enumerate}[(I)]
\item Since $Q_{\ell,a}^{(0)}$ does not depend on $a$, the value of $b_{s, (\ell_0,\boldsymbol{\ell}),(N_0,\mathbf{N})}^{(0,\boldsymbol\varepsilon)}$ 
does not depend on $s$. Without loss of generality, we shall choose $s$ to be $0$ when $\varepsilon_1=0$.

\smallskip

\item $b_{s, (\ell), (1)}^{(0)} (\cdot) = \Phi(0) \mathfrak{a}_\ell^{-d} \sum_{a \in \Z^d} \beta(a) \Ebb_\mu[\cdot] = 0$ and 
$b_{s, (\ell), (N)}^{(0)} (\cdot) = \Phi(0) \mathfrak{a}_\ell^{-d} \sum_{a \in \Z^d} \beta(a)^N \Ebb_\mu[\cdot]$ for all $\ell$, $N \geq 1$.

\smallskip

\item $b_{s, (\boldsymbol{\ell},\ell_0,\boldsymbol{\ell}'),(\mathbf{N},N_0,\mathbf{N}')}^{(\boldsymbol{\varepsilon},0,\boldsymbol{\varepsilon}')} 
= \sum_{j \in \Z/M\Z} \Ebb_\mu [\mathbf{1}_{A_j} B_{s, \boldsymbol{\ell},\mathbf{N}}^{\boldsymbol\varepsilon}(\cdot)] \Ebb_\mu[B_{0, (\ell_0,\boldsymbol{\ell}'),(N_0,\mathbf{N}')}^{(0,\boldsymbol\varepsilon')}(\mathbf{1}_{A_{j+\ell_0}})]$,
i.e.
\begin{equation*}
b_{s, (\boldsymbol{\ell},\ell_0,\boldsymbol{\ell}'),(\mathbf{N},N_0,\mathbf{N}')}^{(\boldsymbol{\varepsilon},0,\boldsymbol{\varepsilon}')} (\cdot)
= \sum_{j\in \Z/M\Z} b_{s, \boldsymbol{\ell},\mathbf{N}}^{\boldsymbol\varepsilon}(\mathbf{1}_{A_j}\cdot) b_{0, (\ell_0,\boldsymbol{\ell}'),(N_0,\mathbf{N}')}^{(0,\boldsymbol\varepsilon')}(\mathbf{1}_{A_j+|\boldsymbol\ell|_1}),
\end{equation*}
since $Q^{(\varepsilon_i)}_{\ell_i,a_i-a_{i-1}}(\mathbf{1}_{A_j}\cdot) = \mathbf{1}_{A_{j+\ell_i}}Q^{(\varepsilon_i)}_{\ell_i,a_i-a_{i-1}}(\cdot)$.

\smallskip

\item In particular, $b_{s, (\boldsymbol{\ell},\ell_0),(\mathbf{N},1)}^{(\boldsymbol{\varepsilon},0)}= 0$, and:
\begin{equation*}
b_{s, (\boldsymbol{\ell},\ell_0,\ell_0',\boldsymbol{\ell}'),(\mathbf{N},1,N_0',\mathbf{N}')}^{(\boldsymbol{\varepsilon},0,0,\boldsymbol{\varepsilon}')} 
= \sum_{j \in \Z/M\Z} b_{s, \boldsymbol{\ell},\mathbf{N}}^{\boldsymbol\varepsilon}(\mathbf{1}_{A_j} \cdot) 
b_{0, (\ell_0),(1)}^{(0)}(\mathbf{1}_{A_j+|\boldsymbol{\ell}|_1}) b_{0, (\ell_0',\boldsymbol{\ell}'),(N'_0,\mathbf{N}')}^{(0,\boldsymbol\varepsilon')} (\mathbf{1}_{A_j+|\boldsymbol{\ell}|_1+\ell_0})
= 0.
\end{equation*}

\smallskip

\item \begin{equation*}
b_{s, (\ell_1, \ldots,\ell_q),(N_1,N_2, \ldots,N_q)}^{(0,1,\ldots,1)} (\mathbf{1}_{A_j})
= \Phi(0) \mathfrak{a}_{\ell_1}^{-d} \sum_{a_1 \in \Z^d} \beta(a_1)^{N_1} b_{a_1, (\ell_2, \ldots,\ell_q),(N_2, \ldots,N_q)}^{(1,\ldots,1)} (\mathbf{1}_{A_{j+\ell_1}}).
\end{equation*}

\smallskip

\item Applying Point~$(V)$ and the fact that $\sum_{a,b\in\Z^d} \beta(a) \beta(b) \Ebb_\mu \left[Q_{\ell',a-b}^{(0)} (\mathbf{1})\right] = 0$, 
we get:
\begin{align*}
b_{s, (\ell,\ell'),(1,1)}^{(0,1)}(\mathbf{1}_{A_j})
& = \Phi(0) \mathfrak{a}_\ell^{-d} \sum_{a,b\in\Z^d} \beta(a) \beta(b) \Ebb_\mu \left[Q_{\ell',a-b}(\mathbf{1}_{A_{j+\ell}})\right] \\
& = \Phi(0) \mathfrak{a}_\ell^{-d} \sum_{a,b\in\Z^d} \beta(a) \beta(b) \mu(A_{j+\ell} ; \ S_{\ell'}=a-b).
\end{align*}
\end{enumerate}
\end{properties}

Given a sequence $\boldsymbol\varepsilon \in \{0, 1\}^q$, we can iterate Point~$(III)$ above to cut $b_{s, \boldsymbol{\ell},\mathbf{N}}^{\boldsymbol\varepsilon}$ 
into smaller pieces, for which $0$ may only appear at the beginning of the associated sequences of indices, and then use Point~$(V)$ to 
transform the heading $\varepsilon_i = 0$. Write $m_1<m_2<\cdots<m_K$ for the indices $i\in\{1,\ldots,q\}$ such that $\varepsilon_i=0$. 
We use the conventions that $m_{K+1}:=q+1$ and $\varepsilon_{q+1} := 0$, that 
$b_{s,\boldsymbol{\ell}, \mathbf{N}}^{\boldsymbol{\varepsilon}} \equiv 1$ if $q=0$, 
and that an empty product is also equal to $1$.
Then:

\begin{align*}
b_{s,\boldsymbol{\ell}, \mathbf{N}}^{\boldsymbol{\varepsilon}}(\mathbf{1})
& = \sum_{j \in \Z/M\Z} b_{s,(\ell_1,\ldots,\ell_{m_1-1}), (N_1,\ldots,N_{m_1-1})}^{(1,\ldots,1)}(\mathbf{1}_{A_j}) \prod_{i=1}^K b_{0,(\ell_{m_i}, \ldots,\ell_{m_{i+1}-1}),(N_{m_i}, \ldots,N_{m_{i+1}-1})}^{(0,1, \ldots ,1)}(\mathbf{1}_{A_{j+\ell_1+\ldots+\ell_{m_i}}}) \\
& = (M \Phi(0))^K \sum_{j \in \Z/M\Z} b_{s,(\ell_1,\ldots,\ell_{m_1-1}), (N_1,\ldots,N_{m_1-1})}^{(1,\ldots,1)}(\mathbf{1}_{A_j}) \\
& \hspace{2em} \times \prod_{i=1}^K \mathfrak{a}_{\ell_{m_i}}^{-d} \sum_{a \in \Z^d} \beta(a)^{N_{m_i}} b_{a, (\ell_{m_i+1}, \ldots,\ell_{m_{i+1}-1}),(N_{m_i+1}, \ldots,N_{m_{i+1}-1})}^{(1, \ldots ,1)}(\mathbf{1}_{A_{j+\ell_1+\ldots+\ell_{m_i}}}) .
\end{align*}
We sum over $\boldsymbol{\ell} \in E_{q,n}$, and get:
\begin{align}
\left| A_{n,q, \mathbf{N}}^{\boldsymbol{\varepsilon}} \right|
& \leq \sum_{\boldsymbol{\ell} \in \{1, \ldots, n\}^q} \left| b_{0,\boldsymbol{\ell}, \mathbf{N}}^{\boldsymbol{\varepsilon}}(\mathbf{1}) \right| \nonumber\\
& \leq (M\Phi(0))^K \left( \sum_{\substack{(\ell_1,\ldots,\ell_{m_1-1}) \\ \in \{1, \ldots, n\}^{m_1-1}}}\sup_{j \in \Z/M\Z} \left| b_{0,(\ell_1,\ldots,\ell_{m_1-1}), (N_1,\ldots,N_{m_1-1})}^{(1,\ldots,1)}(\mathbf{1}_{A_j}) \right| \right) \label{majoA1}\\
& \hspace{1em} \times \prod_{i=1}^K \left( \Afrak_n^2 \sum_{\substack{(\ell_{m_i+1}, \ldots,\ell_{m_{i+1}-1}) \\ \in \{1, \ldots, n\}^{m_{i+1}-m_i-1}}} \sup_{j \in \Z/M\Z} \left| \sum_{a \in \Z^d} \beta(a)^{N_{m_i}} b_{a, (\ell_{m_i+1}, \ldots,\ell_{m_{i+1}-1}),(N_{m_i+1}, \ldots,N_{m_{i+1}-1})}^{(1, \ldots ,1)}(\mathbf{1}_{A_j}) \right|\right). \label{majoA2}
\end{align}

Fix $\omega\in(0,1]$ such that $\frac{\alpha-d}{2} < \omega < \frac{\alpha-d}{2}+\varepsilon$, 
and $\eta \in (0,\omega]$ such that $\omega + \eta \leq \frac{\alpha-d}{2}+\varepsilon$.
The control of~\eqref{majoA1} and of~\eqref{majoA2} shall be done with the following technical lemma, the proof of which 
is postponed until Subsection~\ref{subsec:technical}.

\begin{lemma}
\label{lem:normeB}
Under the assumptions of Theorem~\ref{thm:gene} and with the previous notations,
for every $q\geq 1$ and $\mathbf{N}=(N_1,\ldots,N_q) \in \N_+^q$,
for every $j \in\Z/M\Z$,
\begin{align}
\sup_{a \in \Z^d} \frac{1}{1+|a|^\eta} \sum_{\boldsymbol{\ell} \in \{1,\ldots,n\}^q}  \left| b_{a,\boldsymbol \ell,\mathbf{N}}^{(1,\ldots,1)} (\mathbf{1}_{A_j})\right| & = o \left(\Afrak_n^{N_1+\ldots+N_q}\right), \label{eq:normeB1} \\
\sum_{\ell=1}^n \left| \sum_{a \in \Z^d} \beta(a) b_{a,(\ell),(N)}^{(1)} (\mathbf{1}_{A_j})\right| & = \left\{ \begin{array}{lll} O \left(1\right) & \text{ if } & q=1, \ N=1 \\ o \left(\Afrak_n\right) & \text{ if } & q=1, \ N\geq 2 \end{array}\right. , \label{eq:normeB3} \\
\sum_{\boldsymbol{\ell} \in \{1,\ldots,n\}^q} \left| \sum_{a \in \Z^d}\beta(a) b_{a,\boldsymbol \ell,\mathbf{N}}^{(1,\ldots,1)}(\mathbf{1}_{A_j}) \right| & = o \left(\Afrak_n^{N_1+\ldots+N_q-1}\right) \text{ if } q\geq 2. \label{eq:normeB4}
\end{align}
\end{lemma}

We consider the following condition on the sequences $\boldsymbol{\varepsilon}$ and $\mathbf{N}$:
\begin{equation}\label{dominant}
\begin{array}{l}
m_1=1, \\
  \forall i\in\{1,\ldots,K\},\  \left\{ 
    \begin{array}{lll}
      N_{m_i} \in \{1,2\}& & \\ 
      N_{m_i}=1 & \Rightarrow & m_{i+1}=m_i+2,\ N_{1+m_i}=1 \\ 
      N_{m_i}=2 & \Rightarrow & m_{i+1}=m_i+1 
    \end{array} \right..\\ 
\end{array}
\end{equation}
Note that this condition imply that $m=2K$.

\begin{corollary}

Use the assumptions of Theorem~\ref{thm:gene} and the previous notations. 
Let $m\geq 1$, $q\geq 1$ and $N_1,\ldots,N_q \in\N_+$ be such that $N_1+\ldots+N_q=m$. 
If Condition~\eqref{dominant} holds, then 
\begin{equation*}
\left|A_{n,q, \mathbf{N}}^{\boldsymbol{\varepsilon}}\right|
= O(\Afrak_n^m) ;
\end{equation*}
otherwise:
\begin{equation*}
\left|A_{n,q, \mathbf{N}}^{\boldsymbol{\varepsilon}}\right|
= o(\Afrak_n^m).
\end{equation*}
In particular,
\begin{equation}\label{momentsimpairs}
\left\{
\begin{array}{llll}
\Ebb_\mu \left[\Zcal_n (\beta)^m \right] & = & O \left( \Afrak_n^m \right) & \forall m \in 2\N \\
\Ebb_\mu \left[\Zcal_n (\beta)^m \right] & = & o \left( \Afrak_n^m \right) & \forall m \in 2\N+1
\end{array}
\right..
\end{equation}
\end{corollary}

\begin{proof}

Due to Equation~\eqref{eq:normeB1}, the term~\eqref{majoA1} is an
$o(\Afrak_n^{N_1+\ldots+N_{m_1-1}})$ if $m_1\neq 1$, and 
an $O(1)=O(\Afrak_n^{N_1+\ldots+N_{m_1-1}})$ if $m_1=1$.

\smallskip

Let us now estimate the term~\eqref{majoA2}. Due to Equation~\eqref{eq:normeB1}, 
since $\sum_{a\in\Z^d} |a|^\eta |\beta(a)| < +\infty$, for all $N_0\geq 2$ and $q \geq 1$, 
for all $N_1, \ldots, N_q\geq 1$,
\begin{equation*}
\Afrak_n^2 \sum_{\boldsymbol\ell\in\{1,\ldots,n\}^q}\sup_{j \in \Z/M\Z} \left|\sum_{a\in\Z^d}\beta(a)^{N_0} b_{a,\mathbf{\ell},(N_1,\ldots,N_q)}^{(1, \ldots ,1)}(\mathbf{1}_{A_j}) \right|
= o\left(\Afrak_n^{N_0+N_1+\ldots+N_q}\right).
\end{equation*}
Due to Equation~\eqref{eq:normeB3}, this estimate holds when $N_0\geq 3$ and $q=0$;
due to Equation~\eqref{eq:normeB4}, this estimate holds when $N_0=1$ and $q \geq 2$. 

\smallskip

The two remaining cases are $N_0=2$, $q=0$ and $N_0=q=1$. When $N_0=2$ and $q=0$, we 
have a upper bound in $O(\Afrak_n^2)=O \left(\Afrak_n^{N_0+\ldots+N_q} \right)$. When $q=N_0=N_1=1$, 
the same upper bound is given by Equation~\eqref{eq:normeB3}. If $q=N_0=1$ and $N_1 \geq 2$, 
then Equation~\eqref{eq:normeB3} yields a upper bound in 
$o(\Afrak_n^3) = o\left(\Afrak_n^{N_0+\ldots+N_q}\right)$.

\smallskip

Hence, the term~\eqref{majoA2} is in $O\left(\Afrak_n^{2K}\right)
= O\left(\Afrak_n^{N_{m_1}+\ldots+N_q}\right)$ if, for every
$i\in\{1,\ldots,K\}$, we are in one of two cases:
\begin{itemize}
\item $N_{m_i}=1$, $m_{i+1}=m_i+2$ and $N_{1+m_i}=1$;
\item $N_{m_i}= 2$, $m_{i+1}=m_i+1$.
\end{itemize}
Otherwise, \eqref{majoA2} is in $o\left(\Afrak_n^{N_{m_1}+\ldots+N_q}\right)$.
In particular,
\begin{equation*}
\left|A_{n,q, \mathbf{N}}^{\boldsymbol{\varepsilon}}\right|
= O(\Afrak_n^m).
\end{equation*}
Furthermore, if Condition~\eqref{dominant} is not satisfied, either~\eqref{majoA1} 
is an $o(\Afrak_n^{N_1+\ldots+N_{m_1-1}})$ or one of the terms in~\eqref{majoA2} 
is an $o \left(\Afrak_n^{N_{m_1}+\ldots+N_{m_{i+1}-1}} \right)$, 
so $\left|A_{n,q, \mathbf{N}}^{\boldsymbol{\varepsilon}}\right| = o(\Afrak_n^m)$. 
This is the case, in particular, if $m$ is odd.
\end{proof}

Condition~\eqref{dominant} can be rewritten:
\begin{itemize}
\item $\max_i N_i \leq 2$; 
\item $\varepsilon_i=0$ as soon as $N_i=2$; 
\item there exists $\Jcal \subset\{1,\ldots,q\}$ such that $\{i\, :\, N_i=1\}=\bigsqcup_{j\in \Jcal}\{j,j+1\}$; 
\item $\varepsilon_j=0$ and $\varepsilon_{j+1}=1$ for all $j\in\Jcal$.
\end{itemize}

\smallskip

Assume now that $m \geq 0$ is even. Let us write $\Gcal(q)$ for the set of 
$\mathbf{N}=(N_1,\ldots,N_q) \in\{1,2\}^q$ such that $N_1+\ldots+N_q=m$ and 
$\{i\in\{1,\ldots,q\}\, :\, N_i=1\}$ is the disjoint union of pairs of the form $\{j,j+1\}$. 
Given $\mathbf{N} \in \Gcal (q)$, there exists a unique $\boldsymbol{\varepsilon}(\mathbf{N}) \in \{0,1\}^q$
such that $(\boldsymbol{\varepsilon} (\mathbf{N}),\mathbf{N})$ satisfies Condition~\eqref{dominant}. 
Note that $q = |\{i \, : \, \varepsilon_i =0, \ N_i=2\}|+2|\{i \, : \, \varepsilon_i =0, \ N_i=1\}|$ and 
$m/2 = |\{i \, : \, \varepsilon_i =0, \ N_i=2\}|+|\{i \, : \, \varepsilon_i =0, \ N_i=1\}|$, so that $|\{i \, : \, \varepsilon_i =0, \ N_i=1\}| = q-m/2$ and $|\{i \, : \,  N_i=2\}| = m-q$.
Then:
\begin{equation*}
b_{0;\boldsymbol{\ell},\mathbf{N}}^{\boldsymbol{\varepsilon} (\mathbf{N})}(\mathbf{1})
= \sum_{j \in \Z/M\Z} \left(\prod_{i\, :\, N_i=2} b_{(\ell_i),(2)}^{(0)}
  (\mathbf{1}_{A_{j+\ell_1+\ldots+\ell_i}}) \right)
   \left(\prod_{i\, :\, N_i=1,\varepsilon_i=0} b_{(\ell_i,\ell_{i+1}),(1,1)}^{(0,1)}(\mathbf{1}_{A_{j+\ell_1+\ldots+\ell_i}})\right).
\end{equation*}
Let $\tilde{E}_{q,n}$ be the set of $q$-uplets of integers $(\ell_1,\ldots,\ell_q)\in\{1,\ldots,n\}^q$
such that $M \sum_{i=1}^q\lceil \ell_i/M\rceil\leq n$. Using Points~$(II)$ and~$(VI)$ in Properties~\ref{rmk:proprietesbn}, 
we get:
\begin{align}\label{FORMULEA}
A_{n;q;\mathbf{N}}^{\boldsymbol\varepsilon (\mathbf{N})}
& = (M\Phi(0))^{\frac{m}{2}} \sum_{j \in \Z/M\Z} \sum_{\boldsymbol{\ell}\in E_{q,n}} \left( \prod_{i:N_i = 2} \frac{\sum_{a\in\Z^d} \beta(a)^2}{M \mathfrak{a}_{\ell_i}^d} \right) \nonumber \\
& \hspace{2em} \times \left( \prod_{i:N_i=1,\varepsilon_i=0} \frac{\sum_{a,b\in\Z^d} \beta(a)\beta(b) \mu(A_{j+\ell_1+\ldots+\ell_i};\ S_{\ell_{i+1}}=a-b)}{\mathfrak{a}_{\ell_i}^d}\right) \nonumber \\
& = o(\Afrak_n^m) + \Phi(0)^{\frac{m}{2}} \sum_{j \in \Z/M\Z} \sum_{\boldsymbol{\ell}\in \tilde E_{q,n}} \left( \prod_{i:N_i = 2} \frac{\sum_{a\in\Z^d}\beta(a)^2}{\mathfrak{a}_{\ell_i}^d} \right) \nonumber \\
& \hspace{2em} \times \left( \prod_{i:N_i=1,\varepsilon_i=0}\frac{\sum_{a,b\in\Z^d} \beta(a)\beta(b) \sum_{k=1}^M \mu(A_{j+\ell_1+\ldots+\lfloor\ell_i/M\rfloor M+k};\ S_{\ell_{i+1}}=a-b)}{\mathfrak{a}_{\lceil\ell_j/M\rceil M}^d}\right) \nonumber \\
& = o(\Afrak_n^m) + \Phi(0)^{\frac{m}{2}} \sum_{\boldsymbol{\ell}\in \tilde{E}_{q,n}} \left( \prod_{i:N_i = 2} \frac{\sum_{a\in\Z^d}\beta(a)^2}{\mathfrak{a}_{\ell_i}^d} \right) \left( \prod_{i:N_i=1,\varepsilon_i=0}\frac{\sum_{a,b\in\Z^d} \beta(a)\beta(b) \mu(S_{\ell_{i+1}}=a-b)}{\mathfrak{a}_{\lceil\ell_i/M\rceil M}^d}\right) \nonumber \\
& = o(\Afrak_n^m) + \Phi(0)^{\frac{m}{2}} \left( \sum_{a\in\Z^d}\beta(a)^2 \right)^{m-q} \sum_{\ell_1,\ldots,\ell_{q-m/2} \geq 1} \left[ \left( \prod_{i=1}^{q-m/2} \sum_{a,b\in\Z^d} \beta(a)\beta(b) \mu(S_{\ell_i}=a-b) \right) \right. \nonumber \\
& \hspace{2em} \left. \times \left( \sum_{\boldsymbol{\ell}' \in E_{m/2,n-\sum_{i=1}^{q-m/2} \ell_i}} \prod_{i=1}^{m/2} \frac{1}{\mathfrak{a}_{\ell'_i}^d} \right) \right]. \nonumber
\end{align}
The sequence $(\Afrak_n)$ has regular variation. Due to Lemma~\ref{lem:Integrale}, for all $\ell_1,\ldots,\ell_{q-m/2} \geq 1$,
\begin{equation*}
\sum_{\mathbf{\ell}' \in E_{m/2,n-\sum_{i=1}^{q-m/2} \ell_i}} \prod_{j=1}^{m/2} \mathfrak{a}_{\ell'_j}^{-d} 
\sim \Afrak_n^m \frac{\Gamma \left(1+\frac{\alpha-d}{\alpha} \right)^{\frac{m}{2}}}{\Gamma \left(1+\frac{m}{2} \frac{\alpha-d}{\alpha} \right)} \text{ as } n \rightarrow +\infty.
\end{equation*}
Hence, by the dominated convergence theorem, 
\begin{equation*}
A_{n;q;\mathbf{N}}^{\boldsymbol\varepsilon (\mathbf{N})} 
\sim \Afrak_n^m \Phi(0)^{\frac{m}{2}} \left( \sum_{a\in\Z^d}\beta(a)^2 \right)^{m-q} 
\left( \sum_{\ell \geq 1} \sum_{a,b\in\Z^d} \beta(a)\beta(b) \mu(S_\ell=a-b) \right)^{q-\frac{m}{2}}
\frac{\Gamma \left(1+\frac{ \alpha-d}{\alpha} \right)^{\frac{m}{2}}}{\Gamma \left(1+\frac{m}{2} \frac{\alpha-d}{\alpha} \right)}.
\end{equation*}
If $\mathbf{N} \notin \Gcal (q)$, or $\mathbf{N} \in \Gcal (q)$ but $\boldsymbol\varepsilon \neq \boldsymbol\varepsilon (\mathbf{N})$, 
we have already seen that $A_{n;q;\mathbf{N}}^{\boldsymbol\varepsilon (\mathbf{N})} \ll \Afrak_n^m$. Therefore,
by Equation~\eqref{momentAn},
\begin{equation*}
\Ebb_\mu \left[ \Zcal_n (\beta)^m \right]
\sim \sum_{q=1}^m \sum_{\mathbf{N} \in \Gcal(q)} c_{\mathbf{N}} A_{n;q;\mathbf{N}}^{\boldsymbol\varepsilon (\mathbf{N})}.
\end{equation*}
For fixed $q$, the value of $c_{\mathbf{N}}$ does not depend on $\mathbf{N}$, as the multiset of weights is the same. 
There are $2^{-(m-q)} m!$ maps from $\{1, \ldots, m\}$ to $\{1, \ldots, q\}$ such that $1, \ldots, m-q$ each have $2$ preimages, 
and $m-q+1, \ldots, q$ each have $1$ preimage. Thus,
\begin{equation*}
\forall m\in 2\Z,\quad \Ebb_\mu \left[ \Zcal_n (\beta)^m \right]
\sim m! \sum_{q=1}^m 2^{-(m-q)} \sum_{\mathbf{N} \in \Gcal(q)} A_{n;q;\mathbf{N}}^{\boldsymbol\varepsilon (\mathbf{N})}.
\end{equation*}
For fixed $q$, there are $\binom{m/2}{q-m/2}$ sequences $\mathbf{N} \in \Gcal (q)$: each such sequence is the concatenation 
of $m/2$ blocs of two different kinds, with $r := q-m/2$ blocs of one kind. Thus, for even $m$,
\begin{align*}
\Ebb_\mu \left[ \Zcal_n (\beta)^m \right]
& \sim \Afrak_n^m m! \Phi(0)^{\frac{m}{2}} \frac{\Gamma \left(1+\frac{\alpha-d}{\alpha} \right)^{\frac{m}{2}}}{\Gamma \left(1+\frac{m}{2} \frac{\alpha-d}{\alpha} \right)} \\
& \hspace{2em} \times \sum_{r=0}^{\frac{m}{2}} \binom{m/2}{r} \left( \frac{\sum_{a\in\Z^d}\beta(a)^2}{2} \right)^{\frac{m}{2}-r} \left( \sum_{\ell \geq 1} \sum_{a,b\in\Z^d} \beta(a)\beta(b) \mu(S_\ell=a-b) \right)^r \\
& \sim \Afrak_n^m m! \frac{\Gamma \left(1+\frac{\alpha-d}{\alpha} \right)^{\frac{m}{2}}}{\Gamma \left(1+\frac{m}{2} \frac{\alpha-d}{\alpha} \right)}  \left[\frac{\Phi(0)}{2} \left( \sum_{a\in\Z^d}\beta(a)^2 + 2 \sum_{\ell \geq 1} \sum_{a,b\in\Z^d} \beta(a)\beta(b) \mu(S_\ell=a-b) \right) \right]^{\frac{m}{2}}\\
& = \Afrak_n^m \frac{m!\Gamma \left(1+\frac{\alpha-d}{\alpha} \right)^{\frac{m}{2}}}{2^\frac{m}{2} \Gamma \left(1+\frac{m}{2} \frac{\alpha-d}{\alpha} \right)}  \Phi(0)^{\frac{m}{2}} \sigma_{GK} (\beta, \widetilde{A}, \tilde{\mu}, \widetilde{T})^m
\end{align*}
Let $\Ycal$ be a random variable with a standard $MLGM(1-\alpha/d)$ distribution. Its distribution function is 
even, so all its odd moments are $0$. Let $Y$ have a standard Mittag-Leffler distribution of parameter 
$1-\alpha/d$ and $N$ be a standard Gaussian random variable. Then the even moments of $\Ycal$ are:
\begin{equation*}
\Ebb[\Ycal^m] 
= \Ebb [Y^{\frac{m}{2}}] \Ebb [N^m] 
= \frac{(m/2)! \Gamma \left(1+\frac{\alpha-d}{\alpha} \right)^{\frac{m}{2}}}{\Gamma \left(1+\frac{m}{2} \frac{\alpha-d}{\alpha} \right)} \frac{m!}{2^{\frac{m}{2}} (m/2)!} 
= \frac{m! \Gamma \left(1+\frac{\alpha-d}{\alpha} \right)^{\frac{m}{2}}}{2^{\frac{m}{2}} \Gamma \left(1+\frac{m}{2} \frac{\alpha-d}{\alpha} \right)},
\end{equation*}
so that, for even $m$:
\begin{equation*}
\Ebb_\mu \left[ \Zcal_n (\beta)^m \right]
\sim \Afrak_n^m \Ebb \left[\left(\sqrt{\Phi(0)} \sigma_{GK} (\beta, \widetilde{A}, \tilde{\mu}, \widetilde{T})\Ycal \right)^m \right].
\end{equation*}
We already know that $\Ebb_\mu \left[ \Zcal_n (\beta)^m \right] \ll \Afrak_n^m$ for odd $m$.
Hence, all the moments of $(\Zcal_n (\beta)/\Afrak_n)_n$ converge to the moments of $\sqrt{\Phi(0)} \sigma_{GK} (\beta, \widetilde{A}, \tilde{\mu}, \widetilde{T}) \Ycal$. Since:
\begin{equation*}
\sum_{m \geq 0} \left[\frac{\Gamma \left(1+ \frac{m}{2}\frac{\alpha-d}{\alpha} \right)}{m!  \Gamma \left(1+ \frac{\alpha-d}{\alpha} \right)^{\frac{m}{2}}} \right]^{\frac{1}{2m}} 
= +\infty,
\end{equation*}
Carleman's criterion is satisfied~\cite[Chap.~XV.4]{Feller:1966}, so $(\Zcal_n (\beta)/\Afrak_n)_n$ converges in distribution to $\sqrt{\Phi(0)} \sigma_{GK} (\beta, \widetilde{A}, \tilde{\mu}, \widetilde{T}) \Ycal$, 
when $A \times \Z$ is endowed with the probability measure $\mu \times \delta_0$.

\smallskip

Finally, remark that:
\begin{equation*}
\left| \frac{\Zcal_n (\beta)}{\Afrak_n} \circ \widetilde{T} - \frac{\Zcal_n (\beta)}{\Afrak_n} \right| 
\leq \frac{2 \norm{\beta}{\infty}}{\Afrak_n} 
\rightarrow_{n \to + \infty} 0,
\end{equation*}
so by~\cite[Theorem~1]{Zweimuller:2007}, the sequence $(\Zcal_n (\beta)/\Afrak_n)_n$ converges strongly in distribution to
$\sqrt{\Phi(0)} \sigma_{GK} (\beta, \widetilde{A}, \tilde{\mu}, \widetilde{T}) \Ycal$.

\subsection{Technical lemmas}
\label{subsec:technical}

In the previous section, we used three technical lemmas, whose proofs would have been to long 
to include into our main line of reasoning. Their statements and proofs follow.

\smallskip

We begin with 
Lemma~\ref{lem:0}, which we used to control each part of the decomposition 
$Q_{\ell, a} = Q_{\ell, a}^{(0)}+Q_{\ell, a}^{(1)}$. Recall that $\Phi$ is the 
continuous version of the density function of the stable distribution with characteristic function $e^{-\psi(\sqrt{\Sigma}\cdot)}$.
Since $\mu (S_\ell=a) = \Ebb_\mu \left[Q_{\ell,a}(\mathbf{1})\right]$ 
for $a\in \Z^d$, the following lemma can be understood as a strong form of the the local 
limit theorem for $(S_\ell)_{\ell \geq 1}$.

\begin{lemma}\label{lem:0}

We assume that the Hypotheses~\ref{hyp:HHH} hold.

\smallskip

Let $a \in \Z^d$. For every positive integer $\ell$,
\begin{equation*}
Q_{\ell,a} (h) 
= \frac{\Phi\left(\frac{a}{\mathfrak{a}_\ell}\right)}{\mathfrak{a}_\ell^d} \Pi_0^\ell (h) + \varepsilon_{\ell,a} (h),
\end{equation*}
with $\sup_{a \in \Z^d} \norm{\varepsilon_{\ell,a} }{\Bcal \to \Bcal} = o \left(\mathfrak{a}_\ell^{-d} \right)$. 

\smallskip

Moreover, for every $\omega \in (0,1]$,
\begin{equation}\label{differenceQ}
\sup_{\substack{a,p \in \Z^d \\ p \neq 0}} |p|^{-\omega} \norm{ Q_{\ell,a}-Q_{\ell,a-p} }{}
= O \left(\mathfrak{a}_k^{-(d+\omega)} \right),
\end{equation}
and:
\begin{equation}\label{differenceQ4}
\norm{ Q_{\ell,a-p}-Q_{\ell,a}-Q_{\ell,-p}+Q_{\ell,0} }{}
= O \left((|a|\, |p|)^\omega \mathfrak{a}_\ell^{-(d+2\omega)}\right).
\end{equation}
\end{lemma}

\begin{proof}[Proof of Lemma \ref{lem:0}]
Recall that $Q_{\ell,a} (h) = \frac 1{(2\pi)^d} \int_{\Tbb^d} e^{-i \langle u, a \rangle} P_u^\ell (h) \dd u$.
From Hypothesis~\ref{hyp:HHH}, and up to taking a smaller neighborhood $U$, 
there exist constants $C_0$, $c_0 > 0$ such that $\norm{P_u}{L (\Bcal)} \leq C_0$ and: 
\begin{equation*}
\max \left\{ |\lambda_u|, \left|e^{-\ell \psi(\sqrt{\Sigma} u) L(|\sqrt{\Sigma} u|^{-1})} \right| \right\}
\leq e^{-c_0 |u|^\alpha L(|u|^{-1})},
\end{equation*}
for all $u \in U$.

\smallskip

Let $\varepsilon \in (0,\alpha)$. Since $L$ is slowly varying at infinity and $\Sigma$ is invertible, 
by Karamata~\cite{Karamata:1933} (or Potter's bound~\cite[Theorem~1.5.6]{BinghamGoldieTeugels:1987}), there exists $\ell_0 \geq 0$ 
such that, for every $\ell \geq \ell_0$ and $v\in U$,
\begin{equation*}
\frac{2}{|v|^\varepsilon} 
\leq \left|\frac{L(\mathfrak{a}_\ell/|v|)}{L(\mathfrak{a}_\ell)}\right| 
\leq \frac{|v|^\varepsilon}{2}.
\end{equation*}
Since $nL(\mathfrak{a}_n) \sim \mathfrak{a}_n^\alpha$, up to choosing a larger $\ell_0$, 
for every $\ell \geq \ell_0$ and $v\in U$,
\begin{equation}
\label{eq:MAJOVARIATIONLENTE}
|v|^{\alpha-\varepsilon}
\leq \ell\frac{|v|^\alpha}{\mathfrak{a}_\ell^\alpha} \ L \left(\frac{\mathfrak{a}_\ell}{|v|}\right)
\leq |v|^{\alpha+\varepsilon}.
\end{equation}

\smallskip

We begin with the first point of the lemma. Let $a \in \Z^d$ and $\ell \geq \ell_0$ be an integer. By Hypothesis~\ref{hyp:HHH}, 
\begin{equation}\label{eqnumero0}
Q_{\ell,a}
= \frac{1}{(2\pi)^d}\int_{\Tbb^d} e^{-i \langle u, a \rangle} P_u^\ell \dd u 
= \frac{1}{(2\pi)^d}\int_U e^{-i\langle u, a \rangle} \lambda_u^\ell \Pi_u^\ell \dd u +O(r^\ell),
\end{equation}
and, for every $u \in U$, 
\begin{align}
\norm{\lambda_u^\ell \Pi_u^\ell -e^{-\ell\psi(\sqrt{\Sigma} u) L(|\sqrt{\Sigma} u|^{-1})} \Pi_0^\ell}{} 
& \leq |\lambda_u|^\ell \norm{\Pi_u^\ell-\Pi_0^\ell}{\Lcal (\Bcal,\Bcal)} + \left| \lambda_u^\ell- e^{-\ell\psi(\sqrt{\Sigma} u) L(|\sqrt{\Sigma} u|^{-1})} \right| \norm{\Pi_0^\ell}{\Lcal(\Bcal,\Bcal)} \nonumber \\
& \leq C \left(1+\ell|u|^\alpha L(|u|^{-1})\right) e^{-c_0 \ell|u|^\alpha L(|u|^{-1})}\xi(u) \norm{h}{\Bcal}, \label{eqnumero1}
\end{align}
where $\xi$ is bounded and $\lim_{u\rightarrow 0} \xi(u) = 0$, due to the asymptotic expansion of $u \mapsto \lambda_u$, 
to the continuity of $u \mapsto \Pi_u$ at $0$ and since $\Pi_u^\ell= \Pi_u^{\{\ell/M\}M}$ (see Equation~\eqref{eq:PiM+1}).
Hence:
\begin{align}
\left\| \frac{1}{(2\pi)^d}\int_U e^{-i\langle u, a \rangle} \lambda_u^\ell \Pi_u^\ell \dd u  \right. & 
\left. - \frac{1}{(2\pi)^d}\int_U e^{-i\langle u, a \rangle} e^{-\ell\psi(\sqrt{\Sigma} u) L(|\sqrt{\Sigma} u|^{-1})} \Pi_0^\ell \dd u \right\| \nonumber \\
& \leq C \int_U \left(1+\ell|u|^\alpha L(|u|^{-1})\right) e^{-c_0 \ell|u|^\alpha L(|u|^{-1})} \xi(u) \dd u \nonumber\\
& \leq C \mathfrak{a}_\ell^{-d} \int_{\mathfrak{a}_\ell U} \left(1+\ell\frac{|v|^\alpha}{\mathfrak{a}_\ell^\alpha} L \left(\frac{\mathfrak{a}_\ell}{|v|}\right)\right) e^{-c_0 \ell \frac{|v|^\alpha}{\mathfrak{a}_\ell^\alpha} L\left(\frac{\mathfrak{a}_\ell}{|v|}\right)} \xi \left(\frac{v}{\mathfrak{a}_\ell}\right) \dd v \nonumber\\
& \leq C \mathfrak{a}_\ell^{-d}\int_{\mathfrak{a}_\ell U} \left(1+|v|^{\alpha+\varepsilon}\right) e^{-c_0 |v|^{\alpha-\varepsilon}} \xi \left(\frac{v}{\mathfrak{a}_\ell}\right) \dd v \nonumber\\
& = o(\mathfrak{a}_\ell^{-d}), \label{eqnumero5}
\end{align}
due to \eqref{eq:MAJOVARIATIONLENTE} and to the Lebesgue dominated convergence theorem. Finally, 
\begin{align}
\left| \frac{1}{(2\pi)^d} \int_U \right. & \left. e^{-i\langle u, a \rangle} e^{-\ell\psi(\sqrt{\Sigma} u) L(|\sqrt{\Sigma} u|^{-1})} \dd u - \frac {1}{\mathfrak{a}_\ell^d}\Phi\left(\frac{a}{\mathfrak{a}_\ell}\right) \right| \nonumber\\
& = \left| \frac{1}{(2\pi)^d \mathfrak{a}_\ell^d} \int_{\mathfrak{a}_\ell U} e^{-i\frac{\langle v, a \rangle}{\mathfrak{a}_\ell}} e^{-\ell \psi\left(\frac{v}{\mathfrak{a}_\ell}\right) L\left(\frac{\mathfrak{a}_\ell}{|v|}\right)}- \frac {1}{(2\pi)^d\mathfrak{a}_\ell^d}\int_{\R^d} e^{-i\frac{\langle v, a \rangle}{\mathfrak{a}_\ell}} e^{-\psi(v)} \dd v \right| \nonumber\\
& = \left| \frac{1}{(2\pi)^d \mathfrak{a}_\ell^d} \int_{\mathfrak{a}_\ell U} e^{-i\frac{\langle v, a \rangle}{\mathfrak{a}_\ell}} \left( e^{-\ell \psi\left(\frac{v}{\mathfrak{a}_\ell}\right) L\left(\frac{\mathfrak{a}_\ell}{|v|}\right)}-e^{-\psi(v)}\right) \dd v\right| + o \left(\mathfrak{a}_\ell^{-d} \right)\nonumber\\
& \leq \frac{1}{(2\pi)^d \mathfrak{a}_\ell^d} \int_{\mathfrak{a}_\ell U} \left|e^{-\ell \frac{\psi(v)}{\mathfrak{a}_\ell^\alpha}  L\left(\frac{\mathfrak{a}_\ell}{|v|}\right)}-e^{-\psi(v)}\right| \dd v + o \left(\mathfrak{a}_\ell^{-d}\right)\nonumber\\
& = o \left(\mathfrak{a}_\ell^{-d} \right) \label{eqnumero6},
\end{align}
using again the Lebesgue dominated convergence theorem (with~\eqref{eq:MAJOVARIATIONLENTE} 
for the necessary upper bound). Note that the majorations we used are independent of $a$, whence:
\begin{equation*}
\sup_{a \in \Z^d} \left| \frac{1}{(2\pi)^d} \int_U e^{-i \langle u, a \rangle} e^{-\ell\psi(\sqrt{\Sigma} u) L(|\sqrt{\Sigma} u|^{-1})} \dd u - \frac {1}{\mathfrak{a}_\ell^d}\Phi\left(\frac{a}{\mathfrak{a}_\ell}\right) \right| 
= o \left(\mathfrak{a}_\ell^{-d} \right).
\end{equation*}
This ends the proof of the first point.

\smallskip

Let $\beta > -1$. Let $F : \Tbb^d \to \C$ be a measurable function, with $|F(u)| \leq K |u|^\beta$ for all $u \in U$. 
Then, for all large enough $\ell$, 
\begin{align}
\norm{\frac{1}{(2\pi)^d} \int_{\Tbb^d} F(u) P_u^\ell \dd u}{} 
& \leq  \norm{\frac{1}{(2\pi)^d} \int_U F(u) \lambda_u^\ell \Pi_u^\ell \dd u}{} + \norm{F}{\Lbb^1} O(r^\ell) \nonumber \\
& \leq \frac{KC_0}{(2\pi)^d} \int_U |u|^\beta e^{-c_0 \ell |u|^\alpha L(|u|^{-1})} \dd u + K O(r^\ell) \nonumber \\
& \leq \frac{KC_0}{(2\pi)^d \mathfrak{a}_\ell^{d+\beta}} \int_{\mathfrak{a}_\ell U} |v|^\beta e^{-c_0 \ell \frac{|v|^\alpha}{\mathfrak{a}_\ell^\alpha} L\left(\frac{\mathfrak{a}_\ell}{|v|}\right)} \dd u +K O(r^\ell) \nonumber \\
& \leq \frac{KC_0}{(2\pi)^d \mathfrak{a}_\ell^{d+\beta}} \int_\R |v|^\beta e^{-c_0 |v|^{\alpha-\varepsilon}} \dd u +K O(r^\ell) \nonumber \\
& = K. O\left(\mathfrak{a}_\ell^{-(d+\beta)}\right), \label{eq:BorneGeneraleQF}
\end{align}
where the $O\left(\mathfrak{a}_\ell^{-(d+\beta)}\right)$ depends on $\beta$ but not on $K$.

\smallskip

With $|F(u)| = \left|e^{-i\langle u, a \rangle}-e^{-i \langle u, a-p\rangle}\right| \leq \min(2,|u|\,|p|) \leq 2^{1-\omega} |p|^\omega |u|^\omega$, 
Equation~\eqref{eq:BorneGeneraleQF} yields:
\begin{equation*}
\sup_{a \in \Z^d} \norm{Q_{\ell,a}-Q_{\ell,a-p}}{} 
= O\left(|p|^\omega \mathfrak{a}_\ell^{-(d+\omega)}\right), 
\end{equation*}
which is Equation~\eqref{differenceQ}.

\smallskip

With $|F(u)| = \left|e^{-i\langle u, a\rangle}-1 \right| \left|e^{i\langle u, p\rangle}-1\right| \leq \min(2,|u|\,|p|)\cdot \min(2,|u|\,|a|) \leq 4^{1-\omega} |a|^\omega |p|^\omega |u|^{2\omega}$, 
Equation~\eqref{eq:BorneGeneraleQF} yields:
\begin{equation*}
\norm{Q_{\ell,a-p}-Q_{\ell,a}-Q_{\ell,-p}+Q_{\ell,0}}{}
= O \left( |a|^\omega |p|^\omega \mathfrak{a}_\ell^{-(d+2\omega)}\right),
\end{equation*}
which is Equation~\eqref{differenceQ4}.
\end{proof}

We now give a proof of Lemma~\ref{lem:normeB}, which was stated in the previous section. 
This lemma allowed us to control various sums involving the coefficients $b_{a,\boldsymbol \ell,\mathbf{N}}^{(1,\ldots,1)}$, 
depending on $\mathbf{N}$, and was central in the proof of the main theorem. For the convenience of the reader, 
what we have to prove is reformulated at the begining of the proof.

\begin{proof}[Proof of Lemma~\ref{lem:normeB}]
Let us introduce the following operators on $\Bcal$:
\begin{equation*}
C_{b, a,(\ell_1,\ldots,\ell_q),(N_1,\ldots,N_{q-1})}
:= \sum_{\substack{a_0,\ldots,a_q \in \Z^d \\ a_0=a,\, a_q=b}} \beta(a_{q-1})^{N_{q-1}} Q_{\ell_q,a_q-a_{q-1}}^{(1)} \cdots \beta(a_1)^{N_1} Q_{\ell_2,a_2-a_1}^{(1)} Q_{\ell_1,a_1-a_0}^{(1)},
\end{equation*}
and
\begin{equation*}
D_{a,(\ell_1,\ldots,\ell_q),(N_1,\ldots,N_q)}
:= \sum_{\substack{a_0,\ldots,a_q \in \Z^d \\ a_0=a}} \beta(a_q)^{N_q} Q_{\ell_q,a_q-a_{q-1}}^{(1)} \cdots \beta(a_1)^{N_1} Q_{\ell_1,a_1-a_0}^{(1)}.
\end{equation*}
Note that: 
\begin{equation*}
b_{a;\boldsymbol{\ell},\mathbf{N}}^{(1,\ldots,1)} (\cdot)
= \Ebb_\mu [D_{a,\boldsymbol{\ell},\mathbf{N}}(\cdot)]
\quad\mbox{and}\quad 
\sum_{a\in\Z^d} \beta(a) b_{a,\boldsymbol \ell,\mathbf{N}}^{(1,\ldots,1)} (\cdot)
= \sum_{a\in\Z^d}\beta(a)\Ebb_\mu \left[D_{a,\boldsymbol{\ell},\mathbf{N}}(\cdot) \right]\, .
\end{equation*}
Hence, it is sufficient to prove that:
\begin{align}
\sup_{a\in\Z^d} (1+|a|^\eta)^{-1} \sum_{\boldsymbol{\ell} \in \{1, \ldots, n\}^q} {\norm{D_{a,\boldsymbol{\ell},\mathbf{N}}}{}} & = o \left(\Afrak_n^{N_1+\ldots+N_q}\right), \label{formuleBn} \\
\sum_{\ell = 1}^n \norm{ \sum_{a \in \Z^d} \beta(a)  D_{a,(\ell),(N)}  }{}
 & = \left\{ \begin{array}{lll} O \left(1\right) & \text{ if } & q=1, \ N=1 \\ o \left(\Afrak_n\right) & \text{ if } & q=1, \ N\geq 2 \end{array}\right. , \label{formulediffBnq=1} \\
\sum_{\boldsymbol{\ell} \in \{1, \ldots, n\}^q} \norm{ \sum_{a \in \Z^d} \beta(a)  D_{a,\boldsymbol{\ell},\mathbf{N}}}{}
 & = o \left(\Afrak_n^{N_1+\ldots+N_q-1}\right) \text{ if } q \geq 2. \label{formulediffBn}
\end{align}

\medskip
\begin{itemize}
\item \textbf{Restriction of the problem.}
We first observe that we can restrict our study to the case where all the $N_j$'s are equal to 1. 
The price to pay will be that we will have to consider both $D_{a,\boldsymbol{\ell},(1,\ldots ,1)}$ and $C_{b,a,\boldsymbol{\ell},(1,\ldots,1)}$. 
Equation~\eqref{formulediffBnq=1} shall be proved separately with the next step (Case $q=1$). 

\smallskip

We shall prove the estimates~\eqref{formuleBn} and~\eqref{formulediffBn} in the particular case where
$(N_1,\ldots,N_q)=(1,.,1)$ (or equivalently $N_1+\ldots+N_q=q$), that is:
\begin{align}
\sup_{a\in\Z^d} (1+|a|^\eta)^{-1} \sum_{\boldsymbol{\ell} \in \{1, \ldots, n\}^q} \norm{D_{a,\boldsymbol{\ell},(1,\ldots,1)}}{} & = o \left(\Afrak_n^q\right), \label{formuleBnMajo} \\ 
\sum_{\boldsymbol{\ell} \in \{1, \ldots, n\}^q} \norm{ \sum_{a \in \Z^d} \beta(a)  D_{a,\boldsymbol{\ell},(1,\ldots,1)} }{}
 & = o \left(\Afrak_n^{q-1}\right) \text{ if } q \geq 2, \label{formulediffBn1}
\end{align}
together with the following estimates:
\begin{equation}
\label{eq:FormuleCn}
\sup_{a \in \Z^d} (1+|a|^\eta)^{-1} \sum_{\boldsymbol{\ell} \in \{1, \ldots, n\}^q} \sum_{b \in \Z^d} |\beta(b)| \norm{ C_{b,a,\boldsymbol{\ell},(1,\ldots,1)}}{} 
= o \left(\Afrak_n^{q+1}\right), 
\end{equation}
and:
\begin{equation}
\label{eq:FormuleCnDiff}
\sum_{\boldsymbol{\ell} \in \{1, \ldots, n\}^q} \sum_{b \in \Z^d} |\beta(b)|  \norm{ \sum_{a \in \Z^d} \beta(a) C_{b,a,\boldsymbol{\ell},(1,\ldots,1)} }{}
= o(\Afrak_n^q).
\end{equation}

Assume these estimates to be proved. If $(N_1,\ldots,N_q) \neq (1,\ldots,1)$, let $j$ be the largest index such that $N_j \neq 1$. Then:
\begin{equation}
\label{eq:EQ1}
\norm{ D_{a,\boldsymbol{\ell},(N_1,\ldots,N_q)} }{} 
\leq \sum_{a_j \in \Z^d} |\beta(a_j)|^{N_j} \norm{ D_{a_j, (\ell_{j+1},\ldots,\ell_q),(1,\ldots,1)}}{} \, \norm{ C_{a_j,a,(\ell_1,\ldots,\ell_{j}),(N_1,\ldots,N_{j-1})} }{},
\end{equation}
and
\begin{equation}
\label{eq:EQ2}
\norm{ C_{b,a,\boldsymbol \ell,(N_1,\ldots,N_{q-1})} }{}
\leq \sum_{a_j \in \Z^d} |\beta(a_j)|^{N_j} \norm{ C_{b,a_j, (\ell_{j+1},\ldots,\ell_q),(1,\ldots,1)}}{} \, \norm{ C_{a_j,a,(\ell_1,\ldots,\ell_{j}),(N_1,\ldots,N_{j-1})} }{}.
\end{equation}
Let us iterate this decomposition. Given $(N_1, \ldots, N_{q-1}) \neq (1,\ldots,1)$, 
let $\Jcal := \{1 \leq j < q : \, N_j \geq 2\} = \{j_1, \ldots, j_J\}$, with $j_1 < \cdots < j_J$ 
and $J= |\Jcal|$. We also use the convention $j_0=0$. Iterating Equations~\eqref{eq:EQ1} and~\eqref{eq:EQ2} then yields:
\begin{align}
\sup_{a_0\in\Z^d} & (1+|a_0|^\eta)^{-1} \sum_{\boldsymbol{\ell} \in \{1, \ldots, n\}^q} \norm{ D_{a_0,\boldsymbol \ell,(N_1,\ldots,N_{q-1})} }{} \nonumber \\
& \leq \sum_{a_J \in \Z^d} \sum_{\ell_{1+j_J},\ldots,\ell_q=1}^n |\beta (a_J)|^{{N_{j_J}}} \norm{ D_{a_{j_J},(\ell_{1+j_J},\ldots,\ell_q),(1,\ldots,1)} }{} \nonumber \\
& \hspace{2em} \times \prod_{k=2}^J \left( \sum_{a_{j_{k-1}} \in \Z^d} \sum_{\ell_{1+j_{k-1}},\ldots,\ell_{j_k}=1}^n \sum_{a_{j_k} \in \Z^d} |\beta(a_{j_k})|^{{N_{j_k}}} |\beta(a_{j_{k-1}})|^{{N_{j_{k-1}}}} \norm{ C_{a_{j_k},a_{j_{k-1}},(\ell_{1+j_{k-1}},\ldots,\ell_{j_k}),(1,\ldots,1)} }{}\right) \nonumber \\
& \hspace{2em} \times \left( \sup_{a_0\in\Z^d} \sum_{\ell_1,\ldots,\ell_{j_1}=1}^n   (1+|a_0|^\eta)^{-1} \sum_{a_{j_1} \in \Z^d} |\beta(a_{j_1})|^{{N_{j_1}}} \norm{ C_{a_{j_1},a_0,(\ell_1,\ldots,\ell_{j_k}),(1,\ldots,1)} }{}\right). \label{DECOMP1}
\end{align}
Recall that, since $\eta < \frac{\alpha-d}{2}+\varepsilon $ and $\beta$ is bounded, $|\beta (a)|^x = O (|\beta (a)|) = O ((1+|a|^\eta)^{-1})$ for all $x \geq 1$. 
Using~\eqref{formuleBnMajo} on the first term and~\eqref{eq:FormuleCn} on the others, we get~\eqref{formuleBn}:
\begin{align*}
\sup_{a_0\in\Z^d} (1+|a_0|^\eta)^{-1} & \sum_{\boldsymbol{\ell} \in \{1, \ldots, n\}^q} \norm{ D_{a_0,\boldsymbol \ell,(N_1,\ldots,N_{q-1})} }{} \\
& = o \left(\Afrak_n^{q-j_J}\right) \prod_{k=1}^J o \left(\Afrak_n^{j_k-j_{k-1}+1}\right)
= o \left(\Afrak_n^{q+J}\right)
= o \left(\Afrak_n^{N_1+\ldots+N_q}\right).
\end{align*}
We use the same decomposition to get~\eqref{formulediffBn}. The only difference is that the last term in the decomposition becomes:
\begin{equation*}
\sum_{\ell_1,\ldots,\ell_{j_1}=1}^n \sum_{a_{j_1} \in \Z^d} |\beta(a_{j_1})|^{{N_{j_1}}} \norm{ \sum_{a_0\in\Z^d} \beta(a_0) C_{a_{j_1},a_{0},(\ell_1,\ldots,\ell_{j_1}),(1,\ldots,1)}
}{},
\end{equation*}
which by~\eqref{eq:FormuleCnDiff} is an $o \left( \Afrak_n^{j_1} \right)$. The exponent in the estimate is improved by $1$, which is what we wanted.

\medskip
\item \textbf{First estimates.}
We first provide some general inequalities. From Lemma~\ref{lem:0} and the definition of $Q_{\ell, a}^{(1)}$, 
\begin{equation*}
\norm{ Q_{\ell,a}^{(1)} }{} 
= o \left( \mathfrak{a}_\ell^{-d} \right)+O \left( \frac{\Phi (\mathfrak{a}_\ell^{-d} a) -\Phi(0)}{\mathfrak{a}_\ell^d} \right).
\end{equation*}
Since $\Phi$ is proportional to the Fourier transform of $e^{-\psi(\sqrt{\Sigma}\cdot)}$, it is $\eta$-H\"older for all $\eta \in (0,1]$, 
whence:
\begin{equation}
\label{Q0Q1}
\norm{ Q_{\ell,a}^{(1)} }{} 
= o \left( \mathfrak{a}_\ell^{-d} \right) + O\left(|a|^\eta\mathfrak{a}_\ell^{-d-\eta} \right) 
= o \left((1+|a|^\eta) \mathfrak{a}_\ell^{-d} \right). 
\end{equation}

\smallskip

Due to \eqref{differenceQ},
\begin{equation}
\label{Q1-Q1}
\sup_{b \neq 0} |b|^{-\omega} \norm{ Q_{\ell,b-a}^{(1)}-Q_{\ell,-a}^{(1)} }{} 
= \sup_{b \neq 0} |b|^{-\omega} \norm{ Q_{\ell,b-a}-Q_{\ell,-a}}{} 
= O \left(\mathfrak{a}_\ell^{-(d+\omega)}\right). 
\end{equation}
In particular, since $\sum_{b \in \Z^d} |b|^\omega |\beta(b)| < +\infty$,
\begin{equation}
\label{doublesomme1}
\sup_{a \in \Z^d} \norm{ \sum_b \beta(b) Q_{\ell,b-a}^{(1)}}{}
= \sup_{a \in \Z^d} \norm{ \sum_b \beta(b) (Q_{\ell,b-a}^{(1)}-Q_{\ell,-a}^{(1)})}{}
= O \left(\mathfrak{a}_\ell^{-(d+\omega)} \right).
\end{equation}

\smallskip

Due to \eqref{differenceQ4}, and since $\sum_{a \in \Z^d} \beta(a)=0$,
\begin{align}
\sup_{a \neq 0} (|a|\,|b|)^{-\omega} \norm{ \sum_{b \in \Z^d} \beta(b) (Q_{\ell,b-a}^{(1)}-Q_{\ell,b}^{(1)}) }{}
& = \sup_{a\neq 0} (|a|\,|b|)^{-\omega} \norm{ \sum_{b \in \Z^d} \beta(b) (Q_{\ell,b-a}-Q_{\ell,b})}{} \nonumber \\
& = \sup_{a \neq 0} (|a|\,|b|)^{-\omega} \norm{ \sum_{b \in \Z^d} \beta(b) (Q_{\ell,b-a}-Q_{\ell,b}-Q_{\ell,-a}+Q_{\ell,0})}{} \nonumber \\
& = O \left(\mathfrak{a}_\ell^{-(d+2\omega)} \right). \label{2Q1-Q1-Q1}
\end{align}
In particular, using again the fact that $\sum_{b \in \Z^d} |b|^\omega |\beta(b)| < +\infty$,
\begin{align}
\norm{ \sum_{a,b \in \Z^d} \beta(a)\beta(b) Q_{\ell,b-a}^{(1)} }{} 
& = \norm{ \sum_{a,b \in \Z^d} \beta(a)\beta(b)(Q_{\ell,b-a}-Q_{\ell,b}) }{} \nonumber\\
& = O \left(\mathfrak{a}_\ell^{-(d+2\omega)} \right). \label{doublesomme}
\end{align}

\smallskip

We will also repeatedly use the two following facts:
\begin{equation}
\label{negligeable}
\sum_{\ell\geq 1}{\mathfrak{a}_\ell^{-(d+2\omega)}} <+\infty \quad\mbox{and}\quad
\sum_{\ell\geq 1}{\mathfrak{a}_\ell^{-(d+\omega)}}
= o(\Afrak_n),
\end{equation}
since $(\mathfrak{a}_\ell)_{\ell \geq 0}$ is $1/\alpha$-regular and $\omega > (\alpha-d)/2 \geq 0$.

\medskip
\item \textbf{Case $q=1$.}
We prove separately the case $q=1$, which either involves different inequalities, or shall provide the base case for a recursion. 
We have to prove four estimates, which shall be in order: \eqref{formulediffBnq=1}, \eqref{formuleBnMajo}, \eqref{eq:FormuleCn} and~\eqref{eq:FormuleCnDiff}.

\smallskip

We begin with~\eqref{formulediffBnq=1}. Due to~\eqref{doublesomme}, if $N = 1$,
\begin{align*}
\sum_{\ell = 1}^n \norm{ \sum_{a \in \Z^d} \beta(a)  D_{a,(\ell),(1)} }{} 
& = \sum_{\ell = 1}^n \norm{ \sum_{a,b \in \Z^d} \beta(a) \beta(b)
 Q^{(1)}_{\ell,b-a} }{} \\
& = \sum_{\ell = 1}^n  O \left( \mathfrak{a}_\ell^{-(d+2\omega)} \right) 
= O (1).
\end{align*}
If $N \geq 2$, we use~\eqref{doublesomme1} instead:
\begin{align*}
\sum_{\ell = 1}^n \norm{ \sum_{a \in \Z^d} \beta(a)  D_{a,(\ell),(N)} }{} 
& \leq \sum_{\ell = 1}^n \sum_{b \in \Z^d} |\beta(b)|^N \norm{ \sum_{a \in \Z^d} \beta (a) Q^{(1)}_{\ell,b-a} }{} \\
& = \left( \sum_{b \in \Z^d} |\beta(b)|^N \right) \sum_{\ell = 1}^n O \left(\mathfrak{a}_\ell^{-(d+\omega)} \right) 
= o (\Afrak_n).
\end{align*}

Now, consider~\eqref{formuleBnMajo} for $q=1$. Using~\eqref{doublesomme1} and~\eqref{negligeable},
\begin{align}
\sup_{a\in\Z^d} (1+|a|^\eta)^{-1} \sum_{\ell=1}^n \norm{D_{a,(\ell),(1)}}{}
& \leq \sum_{\ell=1}^n \sup_{a\in\Z^d} \norm{\sum_{b \in \Z^d} \beta(b) Q_{\ell,b-a}^{(1)}}{} \nonumber \\
& = \sum_{\ell = 1}^n O \left(\mathfrak{a}_\ell^{-(d+\omega)} \right) 
= o (\Afrak_n). \label{eq:BorneU1}
\end{align}

\smallskip

Next, we prove~\eqref{eq:FormuleCn} for $q=1$. Note that $C_{b,a,(\ell),\emptyset} = Q_{\ell, b-a}^{(1)}$, 
and that $(1+|b-a|^\eta) \leq (1+|a|^\eta)(1+|b|^\eta)$ since $\eta \leq 1$. Hence, by~\eqref{Q0Q1},
\begin{align*}
\sum_{\ell=1}^n \sum_{b\in\Z^d} |\beta(b)| \norm{ C_{b,a,(\ell),\emptyset}}{} 
& = \sum_{\ell=1}^n \sum_{b\in\Z^d} |\beta(b)| o \left((1+|b-a|^\eta) \mathfrak{a}_\ell^{-d} \right) \\
& = \left( \sum_{b\in\Z^d} |\beta(b)| (1+|b|^\eta)\right) (1+|a|^\eta) \sum_{\ell=1}^n o \left(\mathfrak{a}_\ell^{-d} \right) \\
& = o \left((1+|a|^\eta)\Afrak_n^2 \right).
\end{align*}

\smallskip

Finally, we deal with~\eqref{eq:FormuleCnDiff} for $q=1$. Due to~\eqref{doublesomme1} and~\eqref{negligeable},
\begin{equation*}
\sum_{\ell=1}^n \sum_{b \in \Z^d} |\beta(b)|  \norm{ \sum_{a \in \Z^d} \beta(a) C_{b,a,\ell,\emptyset} }{} 
= \left( \sum_{b \in \Z^d} |\beta(b)| \right) \sum_{\ell=1}^n O \left(\mathfrak{a}_\ell^{-(d+\omega)} \right) 
= o \left( \Afrak_n \right).
\end{equation*}

\medskip
\item \textbf{Case $q \geq 2$.}
It remains to check four estimates, which shall be in order: \eqref{formuleBnMajo}, 
\eqref{formulediffBn1}, \eqref{eq:FormuleCn} and~\eqref{eq:FormuleCnDiff}, for $q \geq 2$.
To simplify the notations, we omit $(1,\ldots,1)$ in indices, and use the convention 
$D_{a, \ell, \emptyset}=1$ for all $a$ and $\ell$. 

\smallskip

We shall prove~\eqref{formuleBnMajo} and~\eqref{formulediffBn1} with recursive bounds involving the functions:
\begin{equation*}
u_{q,n}(a) := \sum_{\ell_1,\ldots,\ell_q=1}^n \norm{D_{a,(\ell_1,\ldots,\ell_q)}}{} \quad\mbox{and}\quad
v_{q,n}(a) := \sum_{\ell_1,\ldots,\ell_q=1}^n \norm{D_{a,(\ell_1,\ldots,\ell_q)}-D_{0,(\ell_1,\ldots,\ell_q)}}{}.
\end{equation*}
Note that~\eqref{formuleBnMajo} is equivalent to the statement that $u_{q,n} (a) = o((1+|a|^\eta)\Afrak_n^q)$, 
while~\eqref{formulediffBn1} is implied by the bound $v_{q,n}(a) = o(|a|^\omega \Afrak_n^{q-1})$ for $q \geq 2$ 
(since $\sum_{a\in\Z^d}\beta(a)=0$). We shall express $u_{q,n}$ and $v_{q,n}$ in terms of $u_{q-1,n}$, 
$v_{q-1,n}$, $u_{q-2,n}$ and $v_{q-2,n}$. 

\smallskip

We start with the sequence $(u_{q,n})$. For all $q \geq 2$,
\begin{align}
D_{a_0,(\ell_1,\ldots,\ell_q)} 
& = \sum_{a_1, a_2} \beta(a_1) \beta(a_2) D_{a_2,(\ell_3,\ldots,\ell_q)} Q_{\ell_2,a_2-a_1}^{(1)} Q_{\ell_1,a_1-a_0}^{(1)} \nonumber \\
& = \sum_{a_1,a_2} \beta(a_2) \left[D_{0,(\ell_3,\ldots,\ell_q)}+D_{a_2,(\ell_3,\ldots,\ell_q)}-D_{0,(\ell_3,\ldots,\ell_q)}\right] \nonumber \\
& \hspace{2em} \times \left(\beta(a_1) Q_{\ell_2,a_2-a_1}^{(1)} \left[Q_{\ell_1,-a_0}^{(1)}+Q_{\ell_1,a_1-a_0}^{(1)}-Q_{\ell_1,0-a_0}^{(1)}\right]\right) \nonumber \\
& = D_{0,(\ell_3,\ldots,\ell_q)}\left[ \left(\sum_{a_1,a_2}\beta(a_1)\beta(a_2) Q_{\ell_2,a_2-a_1}^{(1)}\right) Q_{\ell_1,-a_0}^{(1)}\right. \label{decompD} \\
& \hspace{2em} \left.+\sum_{a_1,a_2} \beta (a_1) \beta(a_2) \left(Q_{\ell_2,a_2-a_1}^{(1)}-Q_{\ell_2,-a_1}^{(1)}\right)\left(Q_{\ell_1,a_1-a_0}^{(1)}-Q_{\ell_1,0-a_0}^{(1)}\right)\right] \nonumber \\
& \hspace{2em} +\sum_{a_2}\beta(a_2) \left(D_{a_2,(\ell_3,\ldots,\ell_q)}-D_{0,(\ell_3,\ldots,\ell_q)}\right) \nonumber \\
& \hspace{2em} \times \left[ \left(\sum_{a_1}\beta(a_1)Q_{\ell_2,a_2-a_1}^{(1)}\right)Q_{\ell_1,-a_0}^{(1)}+\sum_{a_1}\beta(a_1)Q_{\ell_2,a_2-a_1}^{(1)}\left(Q_{\ell_1,a_1-a_0}^{(1)}-Q_{\ell_1,0-a_0}^{(1)}\right)\right] \nonumber
\end{align}
since $\sum_{a_2} \beta(a_2) Q^{(1)}_{\ell_2,-a_1}=0$. Note that $\sum_p|p|^{\eta+\omega}|\beta(p)|<+\infty$ and 
$(1+|a_2-a_1|^\eta)|a_1|^\omega \leq 2(1+|a_1|^{\eta+\omega})(1+|a_2|^\eta)$.
Therefore, using in addition \eqref{Q0Q1}, \eqref{Q1-Q1}, \eqref{2Q1-Q1-Q1} and \eqref{doublesomme}, 
we get that, for all $q \geq 2$,
\begin{multline*}
\norm{D_{a_0,(\ell_1,\ldots,\ell_q)}}{}
= \norm{ D_{0,(\ell_3,\ldots,\ell_q)}}{} O \left( (1+|a_0|^\eta) \mathfrak{a}_{\ell_2}^{-(d+2\omega)} o(\mathfrak{a}_{\ell_1}^{-d})+(\mathfrak{a}_{\ell_1}\mathfrak{a}_{\ell_2})^{-(d+\omega)}\right) \\
+ \sum_{a_2}|\beta(a_2)| \norm{ D_{a_2,(\ell_3,\ldots,\ell_q)}-D_{0,(\ell_3,\ldots,\ell_q)} }{}
\, O \left((1+|a_0|^\eta) \mathfrak{a}_{\ell_2}^{-(d+\omega)} o(\mathfrak{a}_{\ell_1}^{-d})+ (1+|a_2|^\eta) o(\mathfrak{a}_{\ell_2}^{-d}) \mathfrak{a}_{\ell_1}^{-(d+\omega)} \right), \label{MAJOB1}
\end{multline*}
uniformly in $a_0$. If $q=2$, this simplifies to:
\begin{equation*}
\norm{ D_{a_0,(\ell_1,\ell_2)}}{}
= O \left( (1+|a_0|^\eta) \mathfrak{a}_{\ell_2}^{-(d+2\omega)} o(\mathfrak{a}_{\ell_1}^{-d})+(\mathfrak{a}_{\ell_1}\mathfrak{a}_{\ell_2})^{-(d+\omega)}\right).
\end{equation*}
These estimates, combined with~\eqref{negligeable}, yield for all $q \geq 3$:
\begin{equation}
\label{recu}
u_{q,n} (a) 
= O \left( (1+|a|^\eta) \left( u_{q-2,n}(0) o(\Afrak_n^2) + \sum_{a_2 \in \Z^d}|\beta(a_2)| (1+|a_2|^\eta) v_{q-2,n}(a_2) o(\Afrak_n^3) \right) \right),
\end{equation}
and, for $q=2$, 
\begin{equation}
\label{u2}
u_{2,n}(a)
= o\left((1+|a|^\eta) \Afrak_n^2 \right).
\end{equation}

\smallskip

Now, let us consider the sequence $(v_{q,n})$. For all $q \geq 2$,
\begin{align}
D_{a_0,(\ell_1,\ldots,\ell_q)}-D_{0,(\ell_1,\ldots,\ell_q)} 
& = \sum_{a_1}\beta(a_1)D_{a_1,(\ell_2,\ldots,\ell_q)}(Q_{\ell_1,a_1-a_0}^{(1)}-Q_{\ell_1,a_1}^{(1)})\nonumber\\
& = D_{0,(\ell_2,\ldots,\ell_q)}\sum_{a_1}\beta(a_1)(Q_{\ell_1,a_1-a_0}^{(1)}-Q_{\ell_1,a_1}^{(1)})\nonumber\\
& \hspace{2em} + \sum_{a_1}\beta(a_1) \left(D_{a_1,(\ell_2,\ldots,\ell_q)}-D_{0,(\ell_2,\ldots,\ell_q)} \right) \left(Q_{\ell_1,a_1-a_0}^{(1)}-Q_{\ell_1,a_1}^{(1)} \right). \label{decompD-D}
\end{align}
From~\eqref{2Q1-Q1-Q1} and~\eqref{Q1-Q1}, we get that, for all $q \geq 2$,
\begin{align*}
\norm{ D_{a_0,(\ell_1,\ldots,\ell_q)}-D_{0,(\ell_1,\ldots,\ell_q)} }{}
& = \norm{ D_{0,(\ell_2,\ldots,\ell_q)} }{} O \left( |a_0|^\omega \mathfrak{a}_{\ell_1}^{-(d+2\omega)} \right) \\
& \hspace{2em} + \sum_{a_1 \in \Z^d} |\beta(a_1)| \norm{ D_{a_1,(\ell_2,\ldots,\ell_q)}-D_{0,(\ell_2,\ldots,\ell_q)} }{} O \left( |a_0|^\omega \mathfrak{a}_{\ell_1}^{-(d+\omega)} \right),
\end{align*}
so that, using~\eqref{negligeable}.
\begin{equation}
\label{recv}
v_{q,n}(a) 
= O\left( |a|^\omega \left( u_{q-1,n}(0) + \sum_{a_1 \in \Z^d}|\beta(a_1)| v_{q-1,n} (a_1)  o(\Afrak_n) \right) \right).
\end{equation}
From~\eqref{2Q1-Q1-Q1} and~\eqref{negligeable}, we also obtain:
\begin{equation}
\label{v1}
v_{1,n} (a) = O \left( |a|^\omega \right).
\end{equation}

\smallskip

Equation~\eqref{formuleBnMajo} can be reformulated as $u_{q,n} (a)=o((1+|a|^\eta)\Afrak_n^q)$ for $q \geq 1$, 
while Equation~\eqref{formulediffBn1} is a straightforward consequence of the fact that 
$\sum_{a \in \Z^d} |\beta(a)|v_{q,n}(a) = o(\Afrak_n^{q-1})$ for $q \geq 1$
(since $\sum_{a\in\Z^d}\beta(a)=0$). We prove these two identities recursively, and more precisely that:
\begin{equation*}
u_{q,n}(a) 
= o\left((1+|a|^\eta)\Afrak_n^q\right) 
\mbox{ and } \sup_{a\neq 0} |a|^{-\omega} v_{q,n}(a)
= \left\{\begin{array}{lll} O(1) & \text{if} & q=1 \\ o(\Afrak_n^{q-1}) & \text{if} & q \geq 2 \end{array}
\right..
\end{equation*}
This follows from \eqref{recu} and \eqref{recv} by an induction of degree $2$ for $u_{q,n}$ and of degree $1$ for $v_{q,n}$. 
The initialization is given by~\eqref{eq:BorneU1}, \eqref{u2} and~\eqref{v1} (for respectively $u_{1,n}$, $u_{2,n}$ and $v_{1,n}$).

\smallskip

It remains to prove Equations~\eqref{eq:FormuleCn} and~\eqref{eq:FormuleCnDiff}.
Note that~\eqref{decompD} and~\eqref{decompD-D} hold true if we replace $D_{\ldots}$ by $C_{a_q,\ldots}$. 
Hence \eqref{recu} and \eqref{recv} also hold if we replace $u_{q,n}$ and $v_{q,n}$ by, respectively, 
$\tilde{u}_{q,n}$ and $\tilde{v}_{q,n}$, which are given by:
\begin{align*}
\tilde{u}_{q,n} (a)
& := \sum_{\ell_1,\ldots,\ell_q=1}^n \sum_{a_q\in\Z^d}|\beta(a_q)|\norm{ C_{a_q,a,(\ell_1,\ldots,\ell_q)}}{}, \\
\tilde{v}_{q,n} (a)
& :=\sum_{\ell_1,\ldots,\ell_q=1}^n \sum_{a_q\in\Z^d}|\beta(a_q)|\norm{ C_{a_q,a,(\ell_1,\ldots,\ell_q)}-C_{a_q,0,(\ell_1,\ldots,\ell_q)}}{}.
\end{align*}
Note that~\eqref{eq:FormuleCn} is equivalent to the statement that $\tilde{u}_{q,n} (a) = o((1+|a|^\eta)\Afrak_n^{q+1})$, 
while~\eqref{eq:FormuleCnDiff} is implied by the bound $\tilde{v}_{q,n}(a) = o(|a|^\omega \Afrak_n^q)$ for $q \geq 2$.

\smallskip

The first terms are the following. For $\tilde{u}_{1,n} (a)$, we get:
\begin{align*}
\tilde{u}_{1,n} (a)
& = \sum_{\ell=1}^n \sum_{b \in \Z^d} |\beta(b)| \norm{Q_{\ell,b-a}^{(1)}}{} \\
& = \sum_{\ell=1}^n \sum_{b \in \Z^d} |\beta(b)| (1+|b|^\eta) o \left( (1+|a|^\eta)\mathfrak{a}_\ell^{-d} \right) \\
& = o \left((1+|a|^\eta) \Afrak_n^2 \right).
\end{align*}
For $\tilde{u}_{2,n} (a)$, we get:
\begin{align*}
\tilde u_{2,n} (a) 
& = \sum_{\ell_1,\ell_2=1}^n \sum_{b\in\Z^d} |\beta(b)| \norm{\sum_{a_1\in\Z^d}\beta(a_1) Q^{(1)}_{\ell_2,b-a_1}Q^{(1)}_{\ell_1,a_1-a}}{} \\
& \leq \sum_{\ell_1,\ell_2=1}^n \sum_{b\in\Z^d} |\beta(b)|\norm{Q^{(1)}_{\ell_2,b} \sum_{a_1\in\Z^d} \beta(a_1) Q^{(1)}_{\ell_1,a_1-a}}{}+\norm{\sum_{a_1\in\Z^d}\beta(a_1) \left(Q^{(1)}_{\ell_2,b-a_1}-Q^{(1)}_{\ell_2,b}\right) Q^{(1)}_{\ell_1,a_1-a}}{} \\
& = \sum_{\ell_1,\ell_2=1}^n \sum_{b\in\Z^d} |\beta(b)| \left[ o\left((1+|b|^\eta)\mathfrak{a}_{\ell_1}^{-(d+\omega)} \mathfrak{a}_{\ell_2}^{-d} \right) + \sum_{a_1\in\Z^d}|\beta(a_1)| o\left((1+|a|^\eta)(1+|a_1|^{\omega+\eta})\mathfrak{a}_{\ell_1}^{-d}\mathfrak{a}_{\ell_2}^{-(d+\omega)} \right) \right]\\
& = o\left((1+|a|^\eta)\Afrak_n^3\right),
\end{align*}
where we used~\eqref{Q1-Q1} and~\eqref{2Q1-Q1-Q1} for the first part, \eqref{Q1-Q1} and~\eqref{doublesomme1} for the second part, 
and~\eqref{negligeable} to finish. Finally, for $\tilde{v}_{1,n} (a)$, we get:
\begin{equation*}
\tilde{v}_{1,n} (a)
= \sum_{\ell=1}^n \sum_{b\in\Z^d} |\beta(b)| \norm{ Q^{(1)}_{\ell_1,b-a}-Q^{(1)}_{\ell_1,b}}{}
= o \left(|a|^\omega \Afrak_n \right),
\end{equation*}
due to~\eqref{Q1-Q1} and~\eqref{negligeable}.

\smallskip

By induction, we obtain
\begin{equation*}
\tilde{u}_{q,n} (a)
= o \left((1+|a|^\eta) \Afrak_n^{q+1} \right) 
\mbox{ and } \sup_{a\neq 0} |a|^{-\omega} \tilde{v}_{q,n} (a)
= o(\Afrak_n^q),
\end{equation*}
which ends the proof of Lemma \ref{lem:normeB}.
\end{itemize}
\end{proof}

The third and last lemma of this sub-section gives a simple formula for the asymptotic growth of 
the quantity $\sum_{\boldsymbol{\ell} \in E_{q,n}} \prod_{j=1}^q \mathfrak{a}_{k_j}^{-d}$.

\begin{lemma}
\label{lem:Integrale}

Let $1 \leq d \leq \alpha \leq 2$ be and integer and a real number respectively. Recall 
that, for every $q \geq 1$,
\begin{equation*}
E_{q,n} 
= \left\{ \boldsymbol{\ell} \in \{1, \ldots ,n\}^q\ :\ \sum_{j=1}^q {\ell}_j \leq n \right\}.
\end{equation*}
Let $(\mathfrak{a}_\ell)_{\ell \geq 0}$ be a sequence of positive real numbers with 
regular variation of index $1/\alpha$, and $\Afrak_n := \sqrt{\sum_{\ell=1}^n \mathfrak{a}_\ell^{-d}}$. 
Assume that $\lim_{n \to + \infty} \Afrak_n = + \infty$.

\smallskip

For every $q \geq 1$,
\begin{equation*}
\sum_{\boldsymbol{\ell} \in E_{q,n}} \prod_{j=1}^q \mathfrak{a}_{k_j}^{-d} 
\sim \Afrak_n^{2p} \frac{\Gamma \left(1+\frac{\alpha-d}{\alpha} \right)^q}{\Gamma \left( 1+q\frac{\alpha-d}{\alpha} \right)}.
\end{equation*}
\end{lemma}

\begin{proof}

We deal separately with the cases $d = \alpha$ (where $(\Afrak_n)$ has slow variation) and $d < \alpha$.
\begin{itemize}
\item \textbf{Case $d = \alpha$.}
If $d = \alpha\in\{1,2\}$, then $\mathfrak{a}_\ell^d$ is $1$-regularly varying, so $\Afrak_n$ has slow variation. 
By the pigeonhole principle, for all $\boldsymbol{\ell} \in \{1,\ldots,n\}^p\setminus E_{p,n}$,  
there is always one $\ell_i$ such that $\ell_i \geq \lceil n/p\rceil$. Hence:
\begin{equation*}
\left| \sum_{\boldsymbol{\ell}\in E_{p,n}} \prod_{j=1}^p \mathfrak{a}_{\ell_j}^{-d} - \sum_{\boldsymbol{\ell} \in \{1, \ldots, n\}^p} \prod_{j=1}^p \mathfrak{a}_{\ell_j}^{-d} \right|
\leq p \left( \sum_{\ell_1=\lceil n/p\rceil}^n  \mathfrak{a}_{\ell_1}^{-d} \right) \left( \sum_{\ell_2, \ldots, \ell_p = 1}^n \prod_{j=2}^p \mathfrak{a}_{\ell_j}^{-d} \right)
= O \left( \sum_{k=1}^n \Afrak_n^{2(p-1)} \right),
\end{equation*}
and thus $\sum_{\boldsymbol{\ell} \in E_{p,n}} \prod_{j=1}^p \mathfrak{a}_{\ell_j}^{-d} \sim \Afrak_n^{2p}$.

\medskip
\item \textbf{Case $d < \alpha$.}
If $d=1<\alpha$, 
\begin{equation*}
\sum_{\boldsymbol{\ell}\in E_{p,n}} \prod_{j=1}^p \left(\frac{\mathfrak{a}_{\ell_j}}{\mathfrak{a}_n}\right)^{-1} 
= \int_{\{\lceil nu_1\rceil+...+\lceil nu_p\rceil\le n\}} \prod_{j=1}^p\left(\frac{\mathfrak{a}_{\lceil nu_j\rceil}}{\mathfrak{a}_n}\right)^{-1}\, du_1 \cdots du_p.
\end{equation*}
The sequence $(\mathfrak{a}_n)_n$ is $1/\alpha$-regular; by the dominated convergence theorem 
(the domination coming e.g.\ from~\cite[Theorem~1.5.6]{BinghamGoldieTeugels:1987}), 
\begin{equation*}
\lim_{n \to +\infty} \frac{1}{n^p} \sum_{\boldsymbol{\ell}\in E_{p,n}} \prod_{j=1}^p \left(\frac{\mathfrak{a}_{\ell_j}}{\mathfrak{a}_n}\right)^{-1} 
= \int_{\Delta_p} \prod_{j=1}^p {u_j}^{-\frac{1}{\alpha}} \dd u_1\cdots \dd u_p,
\end{equation*}
where $\Delta_p =\{(u_1,...,u_p)\in(0,1)^p \, :\, \sum_{j=1}^p u_j \leq 1\}$. Finally, $n\mathfrak{a}_n^{-1} \sim (1-\alpha^{-1}) \Afrak_n^2$ 
by Karamata's theorem~\cite{Karamata:1933}, \cite[Proposition~1.5.8]{BinghamGoldieTeugels:1987}, so that, as $n$ goes to $+\infty$:
\begin{align}
\sum_{\boldsymbol{\ell}\in E_{p,n}} \prod_{j=1}^p\mathfrak{a}_{\ell_j}^{-1}
& = (n\mathfrak{a}_n^{-1})^p\left(n^{-p} \sum_{\boldsymbol{\ell}\in E_{p,n}} \prod_{j=1}^p \left(\frac{\mathfrak{a}_{\ell_j}}{\mathfrak{a}_n}\right)^{-1} \right) \nonumber\\
& \sim \Afrak_n^{2p} \left(\frac{\alpha-1}{\alpha} \right)^p \int_{\Delta_p} \prod_{j=1}^p {u_j}^{-\frac{1}{\alpha}} \dd u_1\cdots \dd u_p. \label{controledim1}
\end{align}
All that remains is to estimate this later integral. Note that, for all $t \geq 0$,
\begin{equation*}
\int_{t \Delta_p} \prod_{j=1}^p {u_j}^{-\frac{1}{\alpha}} \dd u_1 \cdots \dd u_p 
= t^{p\frac{\alpha-1}{\alpha}} \int_{\Delta_p} \prod_{j=1}^p {u_j}^{-\frac{1}{\alpha}} \dd u_1 \cdots \dd u_p.
\end{equation*}
Hence, using Fubini-Tonnelli's theorem,
\begin{align*}
\int_{\Delta_p} \prod_{j=1}^p {u_j}^{-\frac{1}{\alpha}} \dd u_1 \cdots \dd u_p 
& = \frac{1}{\Gamma \left( 1+ p\frac{\alpha-1}{\alpha} \right)} \int_0^{+ \infty} t^{\frac{\alpha-1}{\alpha}p} e^{-t} \int_{\R_+^p} \prod_{j=1}^p {u_j}^{-\frac{1}{\alpha}} \mathbf{1}_{\left\{ \sum_{j=1}^p u_j \leq 1 \right\}} \dd u_1 \cdots \dd u_p \dd t \\
& = \frac{1}{\Gamma \left( 1+ p\frac{\alpha-1}{\alpha} \right)} \int_{\R_+^p} \prod_{j=1}^p {u_j}^{-\frac{1}{\alpha}} \int_0^{+ \infty} e^{-t} \mathbf{1}_{\left\{ \sum_{j=1}^p u_j \leq t \right\}} \dd t \dd u_1 \cdots \dd u_p \\
& = \frac{1}{\Gamma \left( 1+ p\frac{\alpha-1}{\alpha} \right)} \left( \int_0^{+ \infty} u^{-\frac{1}{\alpha}} e^{-u} \dd u \right)^p \\
& = \frac{\Gamma \left( 1- \frac{1}{\alpha} \right)^p}{\Gamma \left( 1+ p\frac{\alpha-1}{\alpha} \right)}.
\end{align*}
Finally, using the identity $\Gamma (z+1) = z \Gamma (z)$, 
\begin{equation*}
\sum_{\boldsymbol{\ell}\in E_{p,n}} \prod_{j=1}^p\mathfrak{a}_{\ell_j}^{-1} 
\sim \Afrak_n^{2p} \frac{\Gamma \left( 1+\frac{\alpha-1}{\alpha} \right)^p}{\Gamma \left( 1+ p\frac{\alpha-1}{\alpha} \right)} \qedhere
\end{equation*}
\end{itemize}
\end{proof}

\subsection{Renewal properties}
\label{subsec:renewal}

The goal of this Subsection is to prove Proposition~\ref{prop:renouvellement}. 
We assume without loss of generality that the function $L$ appearing in Hypothesis~\ref{hyp:HHH} 
is continuous on $(x_0^{-1},+\infty)$ for some $x_0>0$, and that $u\mapsto u L(u^{-1})$ is increasing 
on this set~\cite[Theorem~1.5.3]{BinghamGoldieTeugels:1987}. When $\alpha=d$, we set for all $x\in(0,x_0)$: 
\begin{equation}\label{eq:FormuleI}
I(x) 
:= \int_x^{x_0} \frac{1}{t L(t^{-1})} \dd t.
\end{equation}

We compute the asymptotics of $g(p)$ according to the method in \cite[Chapter~III.12, P3]{Spitzer:1976}, 
which yields Proposition~\ref{prop:renouvellement}. Before starting the proof, though, 
we use the Fourier transform to represent $g$ in an integral form.

\begin{lemma}\label{L000}

For all $u \in \Tbb^d$, let $\Psi(u) := \sum_{n\geq 0} \Ebb_\mu [e^{i \langle u, S_n \rangle}]$. Under 
Hypothesis~\ref{hyp:HHH}, the function $\Psi$ is continuous on $\Tbb^d \setminus\{0\}$, and, 
for every $p \in \Z^d$, 
\begin{equation}\label{eq:FormuleG}
g (p) 
= \frac{2}{(2\pi)^d} \int_{\Tbb^d} (1-\cos(\langle u, p \rangle)) \Psi(u) \dd u.
\end{equation}
In addition, for all small enough neighborhoods $U$ of $0$,
\begin{equation}\label{eq:ApproximationGIntegrale}
\sup_{p \in \Z^d} \left|g(p) - \frac{2}{(2\pi)^d} \Re \int_U \frac{1-\cos(\langle u, p\rangle)}{1-\lambda_u^M} \sum_{k=0}^{M-1}\lambda_u^k \Ebb_\mu [\Pi_u^k(\mathbf{1})] \dd u\right|
<+\infty.
\end{equation}
\end{lemma}

\begin{proof}

Using the Fourier transform, we know that:
\begin{equation*}
g(p) 
= 2\mu (S_n=0)-\mu (S_n=p)-\mu(S_n=-p) 
= \frac{2}{(2\pi)^d} \int_{\Tbb^d} (1-\cos (\langle u, p \rangle)) \Ebb_\mu [e^{i \langle u, S_n \rangle}] \dd u 
\end{equation*}
Thanks to the Lebesgue dominated convergence theorem, it is then enough to prove that:
\begin{equation*}
\sum_{n \geq 0}\int_{\Tbb^d} |1-\cos (\langle u, p \rangle)|\, \left|\Ebb_\mu [e^{i \langle u, S_n\rangle}] \right| \dd u
< +\infty.
\end{equation*}
Note that $\Ebb_\mu [e^{i \langle u, S_n\rangle}] = \Ebb_\mu [P_u^n \mathbf{1}]$. 
Hence, for any small enough neighborhood $U$ of $0$, 
\begin{equation*}
\sup_{u \in U^c} \sum_{n\geq 0} |\Ebb_\mu [e^{i \langle u, S_n\rangle}]|
\leq \sum_{n\geq 0} C r^n \norm{\mathbf{1}}{\Bcal} 
= O(1),
\end{equation*}
which proves the continuity of $\Psi$ on $\Tbb^d \setminus\{0\}$, as it is the uniform limit of 
a sequence of continuous functions. In addition, for every $u \in U$,
\begin{equation*}
\sum_{n\geq 0}| \Ebb_\mu [e^{i \langle u, S_n\rangle}]|
= \sum_{n\geq 0} \left(|\lambda_u|^n \left| \Ebb_\mu [ \Pi_u^n(\mathbf{1})] \right| + C r^n \norm{\mathbf{1}}{\Bcal} \right)
\leq \frac{C'}{1-|\lambda_u|} + O(1).
\end{equation*}
Finally,
\begin{equation*}
\frac{|1-\cos (\langle u, p \rangle)|}{1-|\lambda_u|}
\leq C'' \frac{|p|^2 |u|^{2-\alpha}}{L(|\sqrt{\Sigma} u|^{-1})},
\end{equation*}
since $1-|\lambda_u| \sim \vartheta |\sqrt{\Sigma} u|^\alpha L(|\sqrt{\Sigma} u|^{-1})$
as $u$ goes to $0$, and $|1-\cos (\langle u, p \rangle)| \leq |u|^2 |p|^2$.
Since $\alpha \in [1,2]$ and that $L$ is slowly varying, this yields Equation~\eqref{eq:FormuleG}.
Moreover, due to \eqref{eq:PiM+1}, 
\begin{equation*}
\sum_{n\geq 0} \lambda_u^n\Pi_u^n
= \sum_{n\ge 0} \lambda_u^{Mn}\sum_{k=0}^{M-1}\lambda_u^k\Pi_u^k
= \frac{1}{1-\lambda_u^M}\sum_{k=0}^{M-1}\lambda_u^k\Pi_u^k.
\end{equation*}

\smallskip

As can be seen in this proof, the error terms which come from integrating over $U$ (instead of $\Tbb^d$) 
and using $\Pi_u$ instead of $P_u$ are uniformly bounded in $p$, so that:
\begin{equation*}
\sup_{p \in \Z^d} \left|g(p) - \frac{2}{(2\pi)^d} \int_U \frac{1-\cos(\langle u, p\rangle)}{1-\lambda_u^M}\sum_{k=0}^{M-1} \lambda_u^k\Ebb_\mu [\Pi_u^k(\mathbf{1})] \dd u\right|
<+\infty.
\end{equation*}
This equation stays true \textit{a fortiori} if we take its real part, which yields Equation~\eqref{eq:ApproximationGIntegrale}.
\end{proof}

Now, let us begin the proof of Proposition~\ref{prop:renouvellement} in earnest.

\begin{proof}[Proof of Proposition~\ref{prop:renouvellement}]

We use the same conventions as in the proof of Lemma~\ref{L000}. 
For all small enough $\delta >0$, put $U (\delta) := \sqrt{\Sigma}^{-1} B(0,\delta)$.
By Equation~\eqref{eq:ApproximationGIntegrale}, for any small enough neighborhood $U$ of $0$,
\begin{equation}\label{EQgp1}
\sup_{p \in \Z^d} \left|g(p) - \frac{2}{(2\pi)^d} \Re \int_U \frac{1-\cos(\langle u, p\rangle)}{1-\lambda_u^M},\sum_{k=0}^{M-1} \lambda_u^k\Ebb_\mu [\Pi_u^k(\mathbf{1})] \dd u\right|
<+\infty.
\end{equation}
Fix $\varepsilon \in (0,1)$. Under Hypothesis~\ref{hyp:HHH}, for all small enough $\delta >0$, for all $u \in U(\delta)$,
\begin{equation*}
\left| \lambda_u-1+\psi(\sqrt{\Sigma} u) L(|\sqrt{\Sigma} u|^{-1}) \right|
\leq \varepsilon |\psi (\sqrt{\Sigma} u)| L(|\sqrt{\Sigma} u|^{-1}),
\end{equation*}
and $\max_{0 \leq k \leq M-1}\norm{\Pi_u^k-\Pi_0^k}{L(\Bcal)} \leq \varepsilon$.
Note also that $\sum_{k=0}^{M-1}\lambda_u^k=(1-\lambda _u^M)/(1-\lambda_u)$. Then:
\begin{equation*}
\left| \frac{1}{1-\lambda_u} - \frac{1}{\psi(\sqrt{\Sigma} u) L(|\sqrt{\Sigma} u|^{-1})} \right| 
\leq \frac{\varepsilon}{|1-\lambda_u|} 
\leq \frac{\varepsilon}{1-\varepsilon} \frac{1}{|\psi(\sqrt{\Sigma} u)|  L(|\sqrt{\Sigma} u|^{-1})},
\end{equation*}
and 
\begin{align*}
\left| \int_{U(\delta)} \right. 
& \left. \frac{1-\cos(\langle u, p\rangle)}{1-\lambda_u^M}\sum_{k=0}^{M-1}\lambda_u^k \Ebb [\Pi_u^k(\mathbf{1})] \dd u - \int_{U(\delta)} \frac{1-\cos(\langle u, p\rangle)}{\psi(\sqrt{\Sigma} u) L(|\sqrt{\Sigma} u|^{-1})} \dd u \right| \\
& \leq \varepsilon \left(\frac{(1+\varepsilon) \norm{\mathbf{1}}{\Bcal}}{1-\varepsilon}+1\right) \int_{U(\delta)} \frac{1-\cos(\langle u, p\rangle)}{|\psi(\sqrt{\Sigma} u)|  L(|\sqrt{\Sigma} u|^{-1})} \dd u\\
& \leq 2\varepsilon \left(\frac{(1+\varepsilon) \norm{\mathbf{1}}{\Bcal}}{1-\varepsilon}+1\right) \sqrt{1+\zeta^2} H_\delta(p),
\end{align*}
where:
\begin{equation*}
H_\delta (p) 
:= \frac{2}{(2\pi)^d} \int_{U(\delta)} \Re \left( \frac{1-\cos(\langle u, p\rangle)}{\psi(\sqrt{\Sigma} u) L(|\sqrt{\Sigma} u|^{-1})} \right) \dd u.
\end{equation*}

Hence,
\begin{equation}\label{eq:DifferenceThetaH}
\lim_{\varepsilon\to 0} \sup_{p\in\Z^d} H_\delta(p)^{-1} \left| \int_{U(\delta)} \frac{1-\cos(\langle u, p\rangle)}{1-\lambda_u^M}\sum_{k=0}^{M-1} \lambda_u^k\Ebb [\Pi_u^k(\mathbf{1})] \dd u - H_\delta(p) \right|
=0.
\end{equation}

Assume that there exists a function $h : \Z^d \to \R$ such that, for all $\delta >0$ small enough, $H_\delta (p) \sim h(p)$ as $p$ goes to infinity. 
If in addition $\lim_\infty h = +\infty$, then Equations~\eqref{EQgp1} and~\eqref{eq:DifferenceThetaH} imply that $g (p) \sim h(p)$.

\smallskip

Now, let us simplify those integrals. First, note that: 
\begin{equation*}
\Re \left( \frac{1}{\psi (\sqrt{\Sigma} u)} \right)
= \frac{1}{\vartheta (1+\zeta^2) |\sqrt{\Sigma} u|^\alpha}.
\end{equation*}
Let $e_1 := (1, 0, \ldots, 0)$. Then:
\begin{align*}
H_\delta (p) 
& = \frac{2}{(2\pi)^d \vartheta (1+\zeta^2)} \int_{\sqrt{\Sigma}^{-1} B(0,\delta)} \frac{1-\cos(\langle u, p\rangle)}{|\sqrt{\Sigma} u|^\alpha L(|\sqrt{\Sigma} u|^{-1})} \dd u \\
& = \frac{2}{(2\pi)^d \vartheta (1+\zeta^2) \sqrt{\det (\Sigma)}} \int_{B(0,\delta)} \frac{1-\cos(\langle v, \sqrt{\Sigma}^{-1} p\rangle)}{|v|^\alpha L(|v|^{-1})} \dd v \\
& = \frac{2}{(2\pi)^d \vartheta (1+\zeta^2) \sqrt{\det (\Sigma)}} \int_{B(0,\delta)} \frac{1-\cos(|\sqrt{\Sigma}^{-1} p| \langle v, e_1\rangle)}{|v|^\alpha L(|v|^{-1})} \dd v \\
& = \frac{2 |\sqrt{\Sigma}^{-1}p|^{\alpha-1}}{(2\pi)^d \vartheta (1+\zeta^2) \sqrt{\det (\Sigma)}} \int_{B(0,|\sqrt{\Sigma}^{-1}p|\delta)} \frac{1-\cos(\langle w, e_1\rangle)}{|w|^\alpha L(|\sqrt{\Sigma}^{-1}p| \, |w|^{-1})} \dd w.
\end{align*}
We shall now distinguish between three sub-cases: $d=1$ and $\alpha \in (1,2]$, then $d=\alpha=1$ 
(in the basin of Cauchy distributions), and finally $d=\alpha=2$.

\begin{itemize}
\item \textbf{Case $d=1$, $\alpha \in (1,2]$.} In this case, most of the mass in the integral 
representation of $g(p)$ is present in a small neighborhood of $0$, of size roughly $1/|p|$.

\smallskip

Let $\eta \in (0,\alpha-1)$. By Potter's bound~\cite[Theorem~1.5.6]{BinghamGoldieTeugels:1987}, 
if $\delta$ is small enough, there exists a constant $C$ such that, for all $p \in \Z$ with a large enough absolute value, 
for all $|w|<|p|\delta$, 
\begin{equation*}
C^{-1} \min \{|w|^\eta, |w|^{-\eta}\}
\leq \left| \frac{L(|p| \, |w|^{-1})}{L(|p|)} \right| 
\leq C \max \{|w|^\eta, |w|^{-\eta}\}.
\end{equation*}
For $w \in [-1,1]$, we get:
\begin{equation*}
L(p)\frac{1-\cos(w)}{|w|^\alpha L(|p| \, |w|^{-1})} 
\leq \frac{C|w|^{2-\alpha-\eta}}{2},
\end{equation*}
while for $1<|w|<|\sqrt{\Sigma}^{-1}p|\delta$:
\begin{equation*}
L(p)\frac{1-\cos(w)}{|w|^\alpha L(|p| \, |w|^{-1})} 
\leq \frac{2C}{|w|^{\alpha-\eta}}.
\end{equation*}
In addition, $L(p)(1-\cos(w))|w|^{-\alpha} / L(|p| \, |w|^{-1})$ converges pointwise, as $p$ goes to infinity, 
to $(1-\cos(w))|w|^{-\alpha}$. By the Lebesgue dominated convergence theorem:
\begin{equation*}
H_\delta (p) 
\sim_{p \to \infty} \frac{2|p|^{\alpha-1}}{\pi \vartheta (1+\zeta^2) L(|p|)} \int_0^{+ \infty} \frac{1-\cos(w)}{w^\alpha} \dd w.
\end{equation*}
Since $\lim_{p \to \infty} H_\delta (p) = + \infty$ and the right hand-side does not depend on $\delta$, 
by~\eqref{eq:DifferenceThetaH},
\begin{equation*}
g (p) 
\sim_{p \to \infty} \frac{2|p|^{\alpha-1}}{\pi \vartheta (1+\zeta^2) L(|p|)} \int_0^{+ \infty} \frac{1-\cos(w)}{w^\alpha} \dd w.
\end{equation*}
Finally, using an integration by parts and \cite[Chap.~XVII.4, (4.11)]{Feller:1966}, we get:
\begin{equation*}
\int_0^{+ \infty} \frac{1-\cos(w)}{w^\alpha} \dd w 
= \frac{\pi}{2 \Gamma (\alpha) \sin \left( \frac{(\alpha-1) \pi}{2} \right)}.
\end{equation*}

\medskip
\item \textbf{Case $d=\alpha=1$.} First, by the same computations and the monotone convergence theorem:
\begin{equation*}
\sum_{n \geq 0} \mu (S_n = 0)
= \frac{1}{2 \pi} \int_{\Tbb} \Psi (u) \dd u 
= \frac{1}{2 \pi \vartheta (1+\zeta^2)} \int_U \frac{1}{|u| L(|u|^{-1})} \dd u (1+o(1))+O(1),
\end{equation*}
where $U$ is any small neighborhood of $0$ in $\Tbb$. But Halmos' recurrence theorem \cite{Halmos:1956} 
and the conservativity of $(\widetilde{A}, \tilde{\mu}, \widetilde{T})$ implies that the left hand-side 
is infinite, so the right hand-side is also infinite, and $\lim_0 I = + \infty$.

\smallskip

Let us go back to the study of $g$. If $\alpha = d = 1$, a neighborhood of size $1/|p|$ of the origin makes for a negligible 
part of the mass of $g(p)$. We must look at a larger scale, where the oscillations makes the cosine ultimately vanish 
(much as with Riemann-Lebesgue's lemma).

\smallskip

Let $R >0$. using again Potter's bound (in the same way as in the previous case), we get that:
\begin{equation*}
\sup_{p \in \Z} L(|p|) \int_{B(0,R)} \frac{1-\cos(w)}{|w| L(|p| \, |w|^{-1})} \dd w 
<+\infty,
\end{equation*}
whence:
\begin{align*}
\frac{\pi \vartheta (1+\zeta^2)}{2} H_\delta (p) 
& = \int_{\frac{R}{|p|}}^\delta \frac{1-\cos(|p| v)}{|v| L(|v|^{-1})} \dd v + O(L(|p|)^{-1}) \\
& = I(R/|p|)-I(\delta)- \int_{\frac{R}{|p|}}^\delta \frac{\cos(|p| v)}{|v| L(|v|^{-1})} \dd v + O(L(|p|)^{-1}).
\end{align*}
By~\cite[Theorem~1.5.9a]{BinghamGoldieTeugels:1987}, $I(R/|p|) \gg L(|p|)^{-1}$. Set $F(v) := 1/(v L(v^{-1}))$, 
which is monotonous on a neighborhood of $0$. Remark that, by~\cite[Theorem~1.5.9a]{BinghamGoldieTeugels:1987} again, 
$|p|^{-1} F(\delta) \ll |p|^{-1} F(R/|p|) \ll I(R/|p|)$ as $p$ goes to infinity. Then, using the Riemann-Stieltjes 
version of the integration by parts~\cite[Theorem~7.6]{Apostol:1974}:
\begin{align}
\left| \int_{\frac{R}{|p|}}^\delta \frac{\cos(|p| v)}{|v| L(|v|^{-1})} \dd v \right|
& = \left| \frac{1}{|p|} \left[ \sin(|p| v) F(v) \right]_{\frac{R}{|p|}}^\delta - \frac{1}{|p|} \int_{\frac{R}{|p|}}^\delta \sin(|p|v) \dd F(v) \right| \nonumber \\
& \leq \frac{2}{|p|} \left( F(R/|p|)+F(\delta) \right). \label{eq:BorneRiemannLebesgue}
\end{align}
Hence, $H_\delta (p) \sim \frac{2}{\pi \vartheta (1+\zeta^2)} I(R/|p|) \sim \frac{2}{\pi \vartheta} I(|p|^{-1})$ as $p$ goes to infinity. 
Since $\lim_{p \to \infty} H_\delta (p) = + \infty$ and $I(|p|^{-1})$ does not depend on $\delta$, 
using the remark following~\eqref{eq:DifferenceThetaH}, we get the claim of the proposition.

\medskip
\item \textbf{Case $d=\alpha=2$.} The method is much the same as for $d=\alpha=1$, but the 
oscillations happen along one axis in the plane. Hence, there is cancellation in almost all directions, 
but not uniformly. Using again Potter's bound, we get that:
\begin{align*}
\frac{\pi \sqrt{\det (\Sigma)}}{2} H_\delta (p) 
& = \frac{1}{2\pi} \int_{B(0,\delta) \setminus B(0,\frac{R}{|\sqrt{\Sigma}^{-1} p|})} \frac{1-\cos(|\sqrt{\Sigma}^{-1} p| \langle v, e_1\rangle)}{|v|^2 L(|v|^{-1})} \dd v + O(L(|p|)^{-1}) \\
& = I(R/|\sqrt{\Sigma}^{-1} p|)-I(\delta)-\frac{1}{2\pi} \int_0^{2\pi} \int_{\frac{R}{|\sqrt{\Sigma}^{-1} p|}}^\delta \frac{\cos(|\sqrt{\Sigma}^{-1}p| r \cos(t))}{r L(r^{-1})} \dd r \dd t + O(L(|p|)^{-1}).
\end{align*}
Fix $\eta > 0$. On $\{|\cos (t)|>\eta \}$, as in~\eqref{eq:BorneRiemannLebesgue},
\begin{align*}
\left| \int_0^{2\pi} \right. & \left.\int_{\frac{R}{|\sqrt{\Sigma}^{-1} p|}}^\delta \frac{\cos(|\sqrt{\Sigma}^{-1}p| v \cos(t))}{v L(v^{-1})} \dd v \dd t \right| \\
& \hspace{4em} \leq \frac{2 \Leb(|\cos (t)|>\eta)}{\eta |\sqrt{\Sigma}^{-1}p|} \left( F(R/|\sqrt{\Sigma}^{-1}p|)+F(\delta) \right) + \Leb(|\cos (t)|<\eta) I(|\sqrt{\Sigma}^{-1}p|/R).
\end{align*}
Since this holds for all $\eta >0$, and since $F(\delta)/|\sqrt{\Sigma}^{-1}p| \ll F(R/|\sqrt{\Sigma}^{-1}p|)/|\sqrt{\Sigma}^{-1}p| \ll I(|\sqrt{\Sigma}^{-1}p|^{-1})$, 
we get that $H_\delta (p) \sim  \frac{2}{\pi \sqrt{\det (\Sigma)}} I(R/|\sqrt{\Sigma}^{-1}p|) \sim \frac{2}{\pi \sqrt{\det (\Sigma)}} I(|p|^{-1})$ as $p$ goes to infinity. Again, 
this is what we claimed in the proposition.
\end{itemize}
\end{proof}

\section{Theorems~\ref{thm:ConvergenceNp} and~\ref{thm:GM}: context and proofs}
\label{sec:proofGM}

Theorem~\ref{thm:gene} yields a limit theorem using only spectral methods. If the factor 
$(A, \mu, T)$ is Gibbs-Markov, then we also have the limit theorems from~\cite{Thomine:2013, Thomine:2014}. 
Comparing the expressions of the limits yields Corollary~\ref{cor:AlphaG}.

\smallskip

Using the structure of Gibbs-markov map, we can leverage Corollary~\ref{cor:AlphaG} 
to get an estimate of the probability that an excursion from $A \times \{0\}$ hits $A \times \{p\}$, 
with $p \in \Z^d$. This is the content of Theorem~\ref{thm:ConvergenceNp}. Finally, 
Theorem~\ref{thm:ConvergenceNp} allows us to improve the main theorems from~\cite{Thomine:2014}, 
yielding Theorem~\ref{thm:GM}. In turn, this improves  Corollary~\ref{cor:AlphaG}, 
yielding Corollary~\ref{cor:InvarianceInduction}.

\smallskip

We present our strategy in Subsection~\ref{subsec:Strategy}. In Subsection~\ref{sec:AnnexeGibbsMarkov}, 
we present Gibbs-Markov maps, and their main properties of interest. Subsections~\ref{subsec:Tightness} 
and~\ref{subsec:ConvergenceLoiPetitesBoules} deal with the tightness and convergence in distribution 
of the (renormalized) number of hits of $A \times \{p\}$ by an excursion, and from there the convergence 
in moments. Finally, Theorem~\ref{thm:ConvergenceNp} and Theorem~\ref{thm:GM} are proved in 
Subsections~\ref{subsec:ProbabiliteAtteinte} and~\ref{subsec:PreuveGM} respectively.

\subsection{Strategy: Working with excursions}
\label{subsec:Strategy}

Our end goal is Theorem~\ref{thm:GM}. Let us describe the strategy behing our proof.

\smallskip

The method used in~\cite{Thomine:2014} to get a distributional limit theorem for observables 
of a Markov $\Z^d$-extension was the following. 
To keep things simple, we ignore L\'evy stable distributions and stay in dimension $d=1$. 
Let $(\widetilde{A},\tilde{\mu},\widetilde{T})$ be an ergodic and conservative Markov $\Z$-extension 
of a Gibbs-Markov map $(A, \mu, T)$ with a square integrable step function $F$ with asymptotic variance 
$\sigma_{GK}^2 (F, A, \mu, T):= \Ebb_\mu[F^2]+2\sum_{k\geq 1}\Ebb_\mu[F \cdot F\circ T^k]$. 

\smallskip

As in Subsection~\ref{sec:ProbabiliteAtteinteExcursions}, let $\varphi_{\{0\}}$ be the first return time to 
$A \times \{0\}$, and $\widetilde{T}_{\{0\}}$ be the induced map on $A\times \{0\} \simeq A$. 
Recall that, for any measurable function $f : A \times \Z^d \to \R$ and 
almost every $x \in A$, we define 
\begin{equation*}
f_{\{0\}}(x)
:=\sum_{k=0}^{\varphi_{\{0\}}(x)-1}f\circ\widetilde{T}^k(x,0),
\end{equation*}
that is, $f_{\{0\}} (x)$ is the sum of $f$ along the excursion from $A \times \{0\}$ 
starting from $(x,0)$.

\smallskip

For every $n \geq 0$ and $x\in A$, let $\tau_n(x)$ be the number of visits 
of $(\widetilde{T}^k(x,0))_{k \geq 1}$ to $A \times \{0\}$ before time $n$. Then, for $x$ in $A$ :
\begin{equation*}
S_n^{\widetilde{T}} f(x,0)
= \sum_{k=0}^{n-1} f \circ \widetilde{T}^k (x,0) 
\simeq \sum_{k=0}^{\tau_n(x)-1} f_{\{0\}} \circ \widetilde{T}_{\{0\}}^k (x),
\end{equation*}
where, under reasonable assumptions on $f$ and on the extension, the error terms are negligible  
for large $n$. If $f$ is integrable and has zero integral, then so does $f_{\{0\}}$. If in addition 
$|f|_{\{0\}}$ belongs to $\Lbb^p$ for some $p > 2$ and if $f$ is regular enough, then $(\tau_n)$ and 
$(\sum_{k=0}^{N-1} f_{\{0\}} \circ \widetilde{T}_{\{0\}}^k)$ are asymptoticaly independent~\cite[Theorem~1.7]{Thomine:2014}, 
and we have a generalized central limit theorem~\cite[Corollary~6.9]{Thomine:2015}, 
which has the following form when $d=1$:
\begin{equation*}
\lim_{n \to + \infty} \frac{S_n^{\widetilde{T}} f}{n^{\frac{1}{4}}} 
= \left(\frac{2}{\pi \sigma_{GK}^2 (F, A, \mu, T)}\right)^{\frac{1}{4}} \sigma_{GK} (f_{\{0\}}, A, \mu, \widetilde{T}_{\{0\}}) L ,
\end{equation*}
where the convergence is strong in distribution and where $L$ is a parameter $1/\sqrt{2}$ centered Laplace 
random variable and where\footnote{Assuming $(A, \mu, \widetilde{T}_{\{0\}})$ is mixing, otherwise the formula differs slightly}:
\begin{equation}\label{eq:EquationVariance}
\sigma_{GK}^2 (f_{\{0\}}, A, \mu, \widetilde{T}_{\{0\}})
= \Ebb_\mu\left[ f_{\{0\}}^2 \right] + 2 \Ebb_\mu\left[ f_{\{0\}} \cdot f_{\{0\}} \circ \widetilde{T}_{\{0\}}^n \right].
\end{equation}
Similar limit theorems hold in dimension two or when the jumps are in the basin of attraction 
of a L\'evy stable distribution.

\smallskip

Due to \cite[Theorem 1.11]{Thomine:2014}, we already know that the limit theorem holds for observables $f$ which are H\"older 
and such that $|f|_{\{0\}} \in \Lbb^q$ for some $q > 2$. However, this condition is hard to check, 
and we would like to get a condition which may be stronger, but more manageable. Our idea is to 
leverage what we know about the observables $f_p:\widetilde{A} \rightarrow \{\pm 1\}$, which we recall 
are defined for $p\in\Z^d$ by $f_p(x,q):= (\mathbf{1}_{\{p\}}-\mathbf{1}_{\{0\}})(q)$.

\smallskip

Note that $f_{p,\{0\}} (x) = N_p (x) -1$, where $N_p (x)$ is the number of visits to 
$A \times \{p\}$ starting from $(x,0)$ before coming back to $A \times \{0\}$. In addition, 
for any observable $f$ and any $q \geq 1$,
\begin{equation}
\label{eq:DominationParTempsLocaux}
\norm{|f|_{\{0\}}}{\Lbb^q (A, \mu)} 
\leq \sum_{p \in \Z^d} \norm{f(\cdot, p)}{\Lbb^\infty (A, \mu)} \norm{N_p}{\Lbb^q (A, \mu)}.
\end{equation}
Hence, we are led to the study of $\norm{N_p}{\Lbb^q (A, \mu)}$
for $q>2$. Note that $\norm{N_p}{\Lbb^q (A, \mu)} = \norm{f_{p,\{0\}}}{\Lbb^q (A, \mu)}+O(1)$.

\smallskip

First, we will see that $\norm{N_p}{\Lbb^2 (A, \mu)} \sim \sigma_{GK} (f_{p,\{0\}}, A, \mu, \widetilde{T}_{\{0\}})$. 
Moreover, comparing the conclusions of Theorem~\ref{thm:gene} of the present paper with a previous result, 
we obtain that $\sigma_{GK}^2 (f_{p,\{0\}}, A, \mu, \widetilde{T}_{\{0\}}) = 2 (g(p)-1)$ for every $p$.
The control on higher moments ($q >2$) of $f_{p,\{0\}}$ helps us 
to extend Theorem~\ref{thm:gene} to a wider class of observables, thanks to the argument in \cite{CsakiCsorgoFoldesRevesz:1992}. 

\smallskip

Our main issue is then to control $\norm{N_p}{\Lbb^q (A, \mu)}$ for any $q > 2$ with 
the weaker norm $\norm{N_p}{\Lbb^2 (A, \mu)}$. For random walks, there is a simple argument, 
which we will replicate in the context of Gibbs-Markov maps. Recall that 
$\alpha (p)^{-1} := \mu (N_p > 0)$ is the probability to visit $A \times \{p\}$ 
before coming back to $A \times \{0\}$, when starting from $0$.

\smallskip

To identify the distribution of $N_p$, it is enough to consider the Markov chain corresponding 
to the times at which the random walk is in $\{0,p\}$, which is given by:

\bigskip

\begin{center}
\begin{tikzpicture}[scale=2.5,->,>=stealth',shorten >=1pt,auto,inner sep=2pt,semithick,bend angle=20]
\tikzstyle{every state}=[draw=none]

\node[state] (0) at (-1,0)      {$0$};
\node[state] (p) at (1,0)      {$p$};

\path (0) edge [in=150,out=210,loop] node[left] {$1-\alpha (p)^{-1}$} (0)
          edge [bend right,looseness=0.5] node[below] {$\alpha (p)^{-1}$} (p)
      (p) edge [in=330,out=30,loop] node[right] {$1-\alpha (-p)^{-1}$} (p)
          edge [bend right,looseness=0.5] node[above] {$\alpha (-p)^{-1}$} (0);
\end{tikzpicture}
\end{center}

\bigskip

Since the random walk spends as much time in $A \times \{0\}$ and $A \times \{p\}$, we get $\alpha (p)=\alpha (-p)$.
Hence, the random variable $N_p | \{N_p > 0\}$ has a geometric distribution of parameter 
$\alpha(p)^{-1}$. So $\norm{f_{p,\{0\}}}{\Lbb^q}$ is determined by $\norm{f_{p,\{0\}}}{\Lbb^2}$ for all $q$.

\smallskip

In the context of Markov $\Z^d$-extensions of Gibbs-Markov maps, we cannot expect to 
know the explicit distribution of $N_p$; however, the same results will hold asymptotically, 
which is enough for our purposes. The main idea is that $N_p-1 | \{N_p > 0\}$ is the hitting time of 
the event $\{\widetilde{T}_{\{0,p\}} \in A \times \{0\} \}$, which becames small as $p$ goes to infinity, 
so $\alpha (-p)^{-1} N_p | \{N_p > 0\}$ converges in distribution to an exponential 
random variable of parameter $1$. Exponential tightness gives the convergence of the moments of $N_p$, 
which is what we want.

\subsection{Recalls on Gibbs-Markov maps}
\label{sec:AnnexeGibbsMarkov}

Throughout this Section, $(A, \pi, \lambda, \mu, T)$ denotes a Gibbs-Markov map. 
These models provide a large enough family of dynamical systems, including many important examples, most 
notably inductions of Markov maps with respect to a stopping time. 
Together with the construction of Young towers~\cite{Young:1998}, Gibbs-Markov maps appear in a variety of subjects, 
including intermittent chaos~\cite{Sarig:2002,Gouezel:2008e,MelbourneTerhesiu:2012}, Anosov flows~\cite{Bowen:1973}, 
or hyperbolic billiards~\cite{Young:1998}. Their definition is flexible enough to allow $\Z^d$-extensions 
with large jumps~\cite{AaronsonDenker:2001}. Yet, Gibbs-Markov maps have a very strong structure which 
makes them tractable. We refer the reader to~\cite[Chapter~4]{Aaronson:1997} and~\cite[Chapitre~1]{Gouezel:2008e} 
for more general references on Gibbs-Markov maps, and to Subsection~\ref{sec:AnnexeGibbsMarkov} 
for some more specialized results. Let us recall their definition.

\begin{definition}[Measure-preserving Gibbs-Markov maps]
\label{def:ApplicationGibbsMarkov}
\ \\
Let $(A, d, \Bcal, \mu)$ be a probability, metric, bounded Polish space. 
Let $\pi$ be a partition of $A$ in subsets of positive measure (up to a null set). 
Let $T : \ A \to A$ be a $\mu$-preserving map, that is exact and Markov with respect to the 
partition $\pi$. Such a map is said to be Gibbs-Markov if it also satisfies:
\begin{itemize}
\item Big image property: $\inf_{a \in \pi} \mu (Ta) > 0$;
\item Expansivity: there exists $\lambda > 1$ such that $d (Tx, Ty) \geq \lambda d(x, y)$ 
for all $a \in \pi$ and $(x,y) \in a \times a$;
\item Bounded distortion: there exists $C > 0$ such that, for all $a \in \pi$, for almost every 
$(x,y) \in a \times a$: 
\end{itemize}
\begin{equation}
\left| \frac{\dd \mu}{\dd \mu \circ T} (x) - \frac{\dd \mu}{\dd \mu \circ T} (y) \right| 
\leq C d(Tx, Ty) \frac{\dd \mu}{\dd \mu \circ T} (x).
\end{equation}
\end{definition}

A measure-preserving Gibbs-Markov map is thus the data $(A, \pi, d, \mu, T)$ of five objects: a topological space, a
partition, a distance, a measure and a measure-preserving transformation. We will sometimes abuse notations, 
and say for instance that $(A, \mu, T)$ is a Gibbs-Markov map. 

\smallskip

Later on, we shall use liberally many fine properties of Gibbs-markov maps. We put them 
together in this subsection, which is divided in three parts:
\begin{itemize}
\item \textit{Fundamental definitions and facts}: what is a Gibbs-Markov map, and what are stopping times.
\item \textit{Good Banach spaces}: the Banach spaces we work with, and the properties of the transfer operator.
\item \textit{Extensions and induction}: what happens when we induce a Markov extension $(\widetilde{A}, \tilde{\mu}, \widetilde{A})$ 
on a nice set, and a distortion estimate.
\end{itemize}

\subsubsection{Fundamental definitions and facts} 

Let $(A, \pi, d, \mu, T)$ be a Gibbs-Markov map. For all $x$ and $y$ in $A$, we define the \textit{separation time} 
of $x$ and $y$ as:
\begin{equation*}
s(x,y) 
:= \inf \{n \geq 0: \ \forall \ a \in \pi, \ T^n x \notin a \text{ or } T^n y \notin a \}.
\end{equation*}
Let $\lambda$ be the expansion constant of the Gibbs-Markov map. Then $(A, \pi, \tau^{-s}, \mu, T)$ 
is Gibbs-Markov for all $\tau \in (1, \lambda]$. Without loss of generality, 
we assume that the distance $d$ belongs to this family of distances, and (if needed) 
we specify the parameter $\tau$ instead of the distance $d$. This simplifies greatly the 
induction processes.

\smallskip

For $n \geq 0$, a \textit{cylinder} of length $n$ is a non-trivial element of $\pi_n := \bigvee_{k=0}^{n-1} T^{-k} \pi$. 
It is given by a unique finite sequence $(a_k)_{0 \leq k < n}$ of elements of $\pi$ such that 
$T (a_k) \cap a_{k+1}$ is non-neglectable for all $0 \leq k < n-1$. Such a cylinder shall be 
denoted by $[a_0, \ldots, a_{n-1}]$.

\smallskip

With any Markov maps comes a natural filtration given by $\Fcal_n := \sigma(\pi_n)$ 
for all $n \geq 0$. From this filtration we define \textit{stopping times}.

\begin{definition}[Stopping time]
\label{def:TempsDArret}

Let $(A, \pi, d, \mu, T)$ be a Gibbs-Markov map. Let $\varphi: \ A \to \N \cup \{+ \infty\}$ be measurable. 
We say that $\varphi$ is a stopping time if $\{\varphi \leq n\} \in \Fcal_n$ 
for all $n \geq 0$.

\smallskip

If $\varphi$ is a stopping time which is almost surely positive and finite, the associated countable 
partition of $A$ is given by:
\begin{equation*}
\pi_\varphi 
:= \bigcup_{n \geq 1} \{ \overline{a} \in \pi_n: \ \mu (\overline{a}) > 0 \text{ and } \overline{a} \subset \{\varphi = n\} \},
\end{equation*}
and the associated transformation is:
\begin{equation*}
T_\varphi (x) 
:= T^{\varphi(x)} (x),
\end{equation*}
which is well-defined almost everywhere if $\varphi$ is finite almost everywhere.
\end{definition}

One of the great advantages of the class of Gibbs-Markov maps is that it is stable by induction, 
and that ergodic Gibbs-Markov maps admits some iterate which is mixing on ergodic components, 
as the following results assert.

\begin{lemma}\cite[Proposition~4.6.2]{Aaronson:1997}

Let $(A, \pi, \lambda, \mu, T)$ be a Gibbs-Markov map, and $\varphi$ be a stopping time for 
the associated filtration $(\Fcal_n)_{n \geq 0}$.

\smallskip

Assume that $\varphi$ is almost surely positive and finite, and that $T_\varphi$ preserves $\mu$. 
Then $(A, \pi_\varphi, \lambda, \mu, T_\varphi)$ is a mesure-preserving Gibbs-Markov map.
\end{lemma}

\begin{proposition}\cite[Proposition~1.3.14]{Gouezel:2008e}
\label{prop:IterationMelangeante}

Let $(A, \pi, \lambda, \mu, T)$ be an ergodic Gibbs-Markov map. Then there exists an integer $M \geq 1$ and a $\sigma (\pi)$-measurable 
partition $(A_k)_{k \in \Z / M \Z}$ of $A$ such that:
\begin{itemize}
\item $T (A_k) = A_{k+1}$ for all $k \in \Z /M\Z$;
\item each $(A_k, \pi_M, \lambda, \mu (\cdot|A_k), T^M)$ is a mixing Gibbs-Markov map.
\end{itemize}
\end{proposition}

\subsubsection{Good Banach spaces} 

Let $P: \Lbb^1(A, \mu)\rightarrow\Lbb^1(A, \mu)$
be the transfer operator associated with $T$. For any bounded measurable 
function $h : A \to \R$, let:
\begin{equation*}
P_h : \ \left\{
\begin{array}{lll}
\Lbb^1(A, \mu) & \to & \Lbb^1(A, \mu) \\
f & \mapsto & P (hf)
\end{array}
\right. .
\end{equation*}

For any $a \in \pi$ and any measurable function $f : A \to \R$, 
we define the Lipschitz exponent of $f$ on $a$ by:
\begin{equation*}
|f|_{\Lip(a,d)} 
:= \Einf \{C \geq 0: |f(x)-f(y)| \leq C d(x,y) \ \forall  (x,y) \in a^2 \}.
\end{equation*}

\begin{definition}
\label{def:LocalementGlobalementLipschitz}

Let us define the following two norms:
\begin{equation*}
\begin{array}{lll}
\norm{f}{\Lip^1 (A, \pi, d, \mu)} & := & \norm{f}{\Lbb^1 (A, \mu)} + \sum_{a \in \pi} \mu (a) |f|_{\Lip(a,d)} ;\\
\norm{f}{\Lip^\infty (A, \pi, d, \mu)} & := & \norm{f}{\Lbb^\infty (A, \mu)} + \sup_{a \in \pi} |f|_{\Lip(a,d)}.
\end{array}
\end{equation*}
The spaces $\Lip^1 (A, \pi, d, \mu)$ and $\Lip^\infty (A, \pi, d, \mu)$ are the spaces of measurable functions 
whose respective norms are finite. The space $\Lip^\infty$ is the space of all \textit{globally Lipschitz functions}, 
while $\Lip^1$ is the space of all \textit{locally Lipschitz functions}. 

\smallskip

A family of observables is uniformly globally (respectively, locally) Lipschitz if the $\Lip^\infty$ norm 
(respectively, the $\Lip^1$ norm) is bounded on this family.

\smallskip

Let $\theta \in (0,1]$. If we replace $d$ by $d^\theta$, we get spaces of globally or locally 
$\theta$-H\"older observables. Any result about Lipschitz observables can be generalized freely to $\theta$-H\"older 
observables.

\end{definition}

The transfer operator $P$ acts quasi-compactly on $\Lip^\infty$. If the transformation is mixing, then 
the transfer operator has a spectral gap, which implies an exponential decay of correlation for Lipschitz 
(and, by extension, H\"older) observables~\cite[Corollaire~1.1.21]{Gouezel:2008e}:

\begin{proposition}[Exponential decay of correlations]
\label{prop:DecroissanceExponentielle}

Let $(A, \pi, d, \mu, T)$ be a mixing Gibbs-Markov map. Then there exist constants $C, \kappa > 0$ such that, 
for all $n \geq 0$, for all $g \in \Lip^\infty (A)$,
\begin{equation*}
\norm{P^n g - \int_A g \dd \mu \cdot \mathbf{1}}{\Lip^\infty (A)} 
\leq C e^{- \kappa n}\norm{g}{\Lip^\infty (A)}.
\end{equation*}
\end{proposition}

In addition, $P$ maps continuously $\Lip^1$ into $\Lip^\infty$ \cite[Lemme~1.1.13]{Gouezel:2008e}. 
This feature (that $P$ maps a large space of integrable functions into a space of bounded functions) 
is specific to Gibbs-Markov maps.

\subsubsection{Extensions and induction} 

Let $(A, \pi, \lambda, \mu, T)$ be a measure-preserving Gibbs-Markov map. Let $G$ be a discrete countable Abelian 
group with counting measure $\nu$. Let $F : \ A \to G$ be $\sigma (\pi)$-measurable. If 
$(A \times G, \mu_{A \times G}, \widetilde{T})$ is conservative and ergodic, then 
for any non-empty subset $S \subset G$ and any $p \in G$, the function:
\begin{equation*}
\varphi_{p,S} : \ \left\{
\begin{array}{lll}
A \times S & \to & \N_+ \\
x & \mapsto & \inf \{n \geq 1 : \ \widetilde{T}^n (x,p) \in A \times S\}
\end{array}
\right.
\end{equation*}
is a stopping time which is almost surely positive and finite.

\smallskip

Let $S \subset G$ be non-empty and finite. Set:
\begin{itemize}
\item a partition $\pi_S := \{a \times \{p\}: \ p \in S, \ a \in \pi_{\varphi_{p,S}} \}$;
\item a measure $\mu_S := \nu(S)^{-1} \mu \otimes \nu_{|S}$;
\item a transformation:
\begin{equation*}
T_S : \ \left\{
\begin{array}{lll}
A \times S & \to & A \times S \\
(x,p) & \mapsto & \widetilde{T}^{\varphi_{p,S} (x)} (x)
\end{array}
\right. .
\end{equation*}
\end{itemize}

\begin{proposition}[Inductions of extensions of Gibbs-Markov maps are Gibbs-Markov]
\label{prop:ExtensionsGMGM}

Let $(\widetilde{A}, \tilde{\mu}, \widetilde{T})$ be an ergodic and conservative Markov extension 
of a Gibbs-markov map $(A, \pi, \lambda, \mu, T)$.

\smallskip

Then, for any non-empty finite subset $S \subset G$, the dynamical system $(A \times S, \pi_S, \lambda, \mu_S, T_S)$ 
is a measure-preserving ergodic Gibbs-Markov map.
\end{proposition}

\begin{proof}

Up to straightforward modifications, the proof is the same as in \cite[Proposition~4.6.2]{Aaronson:1997}.
\end{proof}

We can then define a transfer operator $P_S$ associated to any non-trivial and finite $S \subset G$.

\smallskip

For the remainder of the section, we assume that $(A, \pi, \lambda, \mu, T)$ is a measure-preserving and ergodic Gibbs-Markov map, 
$G$ a discrete countable Abelian group with counting measure $\nu$, and $F : \ A \to G$ a $\sigma (\pi)$-measurable 
function. We assume that $(\widetilde{A}, \tilde{\mu}, \widetilde{T})$ is conservative and ergodic.

\smallskip

In our proof, we will sometimes have to control the distortion of $T_\varphi : x \mapsto T^{\varphi (x)} (x)$ 
for various stopping times $\varphi$. This is done with the next lemma, which generalizes 
\cite[Lemme~1.1.13]{Gouezel:2008e}. We write $P^{(\varphi)}$ for the transfer operator associated with $T_\varphi$.

\begin{lemma}\label{lem:EncoreUneBorneUniforme}

Let $(A, \pi, \lambda, \mu, T)$ be a measure-preserving Gibbs-Markov map. Then there exists a constant 
$K > 0$ with the following property. Let $\varphi$ be a stopping time which is finite with positive probability 
as well as almost surely positive. Let $A \subset \{\varphi < + \infty\}$ be $\sigma(\pi_\varphi)$-measurable 
and non-trivial. Then:
\begin{equation*}
\norm{\frac{P^{(\varphi)} \mathbf{1}_A}{\mu (A)}}{\Lip^\infty (A, \pi, \lambda)} 
\leq K.
\end{equation*}
\end{lemma}

\begin{proof}

Let $n \geq 1$, and let $\overline{a}$ be a cylinder of length $n$ for the Gibbs-Markov map $(A, \pi, \lambda, \mu, T)$. 
By a strengthening of the Distortion lemma, e.g.\ \cite[Lemme~1.1.13]{Gouezel:2008e}, there is a constant $K$ 
independent of $n$ and $\overline{a}$ such that:
\begin{equation*}
\norm{P^n (\mathbf{1}_{\overline{a}})}{\Lip^\infty (A, \pi, \lambda)} 
\leq K \mu (\overline{a}).
\end{equation*}
By additivity, this inequality remains true whenever $\overline{a}$ is $\sigma (\pi_n)$-measurable. 
For all $n \geq 1$, let $A_n := A \cap \{\varphi = n\}$. Then $(A_n)_{n \geq 1}$ is a partition 
of $A$. In addition, each $A_n$ is $\sigma (\pi_n)$-measurable, and 
$P^{(\varphi)} = P^n$ for functions supported by $A_n$, so that:
\begin{equation*}
\frac{P^{(\varphi)} \mathbf{1}_A}{\mu (A)}
= \mu (A_\varphi)^{-1} \sum_{n \geq 1} P^n \mathbf{1}_{A_n}.
\end{equation*}
By additivity again, the $\Lip^\infty$ norm of the right-hand side is at most:
\begin{equation*}
K \mu (A_\varphi)^{-1} \sum_{n \geq 1} \mu (A_n)
= K \mu (A)^{-1} \mu(A) 
= K. \qedhere
\end{equation*}
\end{proof}

\subsubsection{Fulfillment of the spectral hypotheses} 

The spectral hypotheses~\ref{hyp:HHH} are used in our main theorems, and Gibbs-Markov maps 
appear in a variety of applications. We provide here a simple sufficient criterion to ensure 
that the spectral hypotheses are satisfied for Gibbs-markov maps. The hypothesis of aperiodicity 
will be central.

\begin{definition}[Aperiodic extensions]\label{def:Aperiodicite}

Let $(A, \mu, T)$ be a dynamical system preserving a probability measure. Let $d \geq 0$ and $F : A \to \Z^d$ 
a measurable function. The corresponding extension $(\widetilde{A}, \tilde{\mu}, \widetilde{T})$ 
is said to be \textit{aperiodic} if the coboundary equation:
\begin{equation*}
F = k+u \circ T-u \quad \text{mod} \ \Lambda,
\end{equation*}
where $\Lambda$ is a proper sublattice of $\Z^d$, $k \in (\Z^d)_{/\Lambda}$, and $u : A \to (\Z^d)_{/\Lambda}$ 
is measurable.
\end{definition}

We prove the following:

\begin{proposition}\label{prop:GMHHH}

Let $(\widetilde{A}, \tilde{\mu}, \widetilde{T})$ be an aperiodic Markov $\Z^d$-extension 
of a Gibbs-Markov map $(A,\pi,\lambda,\mu,T)$ with step function $F$. Assume that the extension is 
ergodic, conservative, and either of the following hypotheses:
\begin{itemize}
\item $d=1$ and $F$ is in the domain of attraction of an $\alpha$-stable distribution, with $\alpha \in (1,2]$.
\item $d=1$ and $\int_A e^{iuF}\dd \mu = e^{-\vartheta|u|[1-i\zeta\sgn(u)]L(|u|^{-1})} + o\left(|u| L(|u|^{-1})\right)$ at $0$, 
for some real numbers $\vartheta>0$ and $\zeta\in\R$ and some function $L$ with slow variation. 
\item $d=2$ and $F$ is in the domain of attraction of a non-degenerate Gaussian random variable.
\end{itemize}

\smallskip

Then the Hypotheses~\ref{hyp:HHH} are satisfied with $\Bcal := \Lip^\infty$.
\end{proposition}

\begin{proof}

The recurrence of the extension is among the hypotheses. Since the extension is ergodic, 
so is $(A, \mu, T)$. The existence of an integer $M \geq 1$ and a decomposition of $A$ 
into $M$ measurable subsets $(A_i)_{i \in \Z/M\Z}$ on which $T^M$ is mixing follows~\cite[Th\'eor\`eme~1.1.8]{Gouezel:2008e}.

\smallskip

We choose the Banach space $\Lip^\infty \subset \Lbb^\infty (A, \mu) \subset \Lbb^1 (A, \mu)$. 
Then $\mathbf{1} \in \Lip^\infty$, and $P$ cts continuously on $\Lip^\infty$. In addition, 
the subsets $A_i$ are $\sigma (\pi)$-measurable, so for all $f \in \Lip^\infty$:
\begin{equation*}
\norm{\mathbf{1}_{A_i}f}{\Lip^\infty (A, \pi, d, \mu)} 
= \norm{\mathbf{1}_{A_i}f}{\Lbb^\infty (A, \mu)} + \sup_{a \in \pi} |\mathbf{1}_{A_i}f|_{\Lip(a,d)} 
\leq \norm{f}{\Lip^\infty (A, \pi, d, \mu)},
\end{equation*}
so the multiplication by $\mathbf{1}_{A_i}$ acts continuously on $\Lip^\infty$.

\smallskip

We use Proposition~\ref{prop:CSHHH} to check the third item. The function $F$ is constant 
on elements of the partition $\pi$, so, with the notations of~\cite{Gouezel:2008e}  $D_\tau f (a) \equiv 0$. 
Hence, by~\cite[Corollaire~4.1.3]{Gouezel:2008e}, the application $u \mapsto P_u$, 
as a function with values in $\Lcal (\Lip^\infty, \Lip^\infty)$, is continuous in $0$. 
But multiplication by $e^{i \langle u, F \rangle}$ is continuous on $\Lip^\infty$, and 
$P_v (f) -P_u (f) = (P_{v-u}-P)(e^{i \langle u,F \rangle \}} f)$. Hence, $u\mapsto P_u$ 
is continuous for all $u$.

\smallskip

The action of $P$ on $\Lip^\infty$ is quasi-compact: the spectrum of $P$ is included 
in the closed unit ball, its intersection with the unit circle is exactly the set of $M$th 
roots of the unity, and the remainder of the spectrum lies in a ball of smaller radius. 
Hence, the eigendecomposition of $P$ is continuous for small parameters $u$. 
The hypotheses of Proposition~\ref{prop:CSHHH} follow, except for the last one 
(that $P_u$ has no eigenvalue of modulus one for $u\neq 0$).

\smallskip

Since the extension is assumed to be aperiodic, the spectral radius of $P_u$ 
acting on $\Lip^\infty$ is strictly smaller than $1$ for $u \neq 0$, by~\cite[Lemma~2.6]{Thomine:2015}. 
We point out that the later lemma uses the hypotheses that $T$ be mixing and $F$ 
integrable, but these assumptions are not used in its proof. We have checked all 
the assumptions of Proposition~\ref{prop:CSHHH}, and thus the third item 
of Hypotheses~\ref{hyp:HHH}.

\smallskip

The expansion of the main eigenvalue for Gibbs-Markov maps is done in~\cite{AaronsonDenker:2001} 
in the $1$-dimensional case. If $F \in \Lbb^2$, then it is an instance of the central limit theorem 
by spectral methods, as in~\cite{Nagaev:1957}; otherwise, the expansion ultimately satisfies:
\begin{equation}\label{eq:ExpansionValeurProprePrincipale}
1-\lambda_u 
\sim 1-\int_A e^{i\langle u, F\rangle}\dd \mu,
\end{equation}
and the formulas comes from~\cite{Feller:1966}. 

\smallskip

Note that, if $\alpha \in (1,2]$, Birkhoff's theorem and the conservativity of the extension 
imply that $F$ has no drift, which finishes this case. For $\alpha =1$, the expansion 
of $\int_A e^{iu F}\dd \mu$ is part of the hypothesis.
\end{proof}

\subsection{Tightness}
\label{subsec:Tightness}

In this subsection and the next, for any metric space $(E,d)$, any $x\in E$ and any $R>0$, we write 
$\overline{B}_E(x,R)$for the closed ball in $(E,d)$ of center $x$ and of radius $R$, and $S_E(x,R)$ 
for the corresponding sphere.

\smallskip

Recall that, for all $p \in G$, for all $x \in A$, we put $N_p (x) = |\{0 \leq k < \varphi_{\{0\}} (x) : \ \widetilde{T}^k (x,0) \in A \times \{p\}\}|$ 
and $N_{0,p} (x) = \inf \{n \geq 0: \ T_{\{0\}}^n (x) \in \{N_p >0\} \}$. 
The goal of this section is to obtain an upper bound for the tail distribution of $N_{0,p}$. 
This estimate will be used later to prove the tightness of $\alpha (p)^{-1} N_p | \{N_p > 0\}$.

\smallskip

Since $T$ is ergodic, we consider $M \in\N_+$ and
$(A_k)_{k \in \Z/M\Z}$ as in Proposition~\ref{prop:IterationMelangeante}. For all $k \in \Z/M\Z$ and $f \in \Lbb^1 (A, \mu)$, 
let $\Pi_k$ be the projection $f \mapsto \int_{A_k} f \dd \mu \cdot \mathbf{1}_{A_k}$.
For all $K>0$, we set:
\begin{equation*}
\Scal_K 
:= \left\{ h: \ A \to [0,1]\right\} \cap \overline{B}_{\Lip^\infty (A)} (\mathbf{0}, K).
\end{equation*}

\begin{proposition}\label{prop:Tension}

Let $(A, \pi, \lambda, \mu, T)$ be an ergodic Gibbs-Markov map. Then for all $K > 0$, 
there exist constants $C, \kappa > 0$ such that for all $h \in \Scal_K$, for all $n \geq 0$,
\begin{equation}\label{eq:Tension}
\norm{P_h^n (\mathbf{1})}{\Lbb^1 (A, \mu)} 
\leq C e^{- \kappa \norm{\mathbf{1}-h}{\Lbb^1 (A, \mu)} n}.
\end{equation}
\end{proposition}

\begin{proof}
\begin{itemize}
\item First, let us assume that $(A, \mu, T)$ is mixing. We only need to prove the assertion for $K \geq 1$. 
Let $h \in \Scal_K$. 

\smallskip

If $h < 1/2$ somewhere, since $\Scal_K$ is convex and $\mathbf{1} \in \Scal_K$, the function $h' := (\mathbf{1}+h)/2$ 
also belongs to $\Scal_K$ and satisfies $h' \geq 1/2$. In addition, $P_h^n (\mathbf{1}) \leq P_{h'}^n (\mathbf{1})$ 
for all $n$, so any upper bound for $\norm{P_{h'}^n (\mathbf{1})}{\Lbb^1 (A, \mu)}$ is also an upper bound for $\norm{P_h^n (\mathbf{1})}{\Lbb^1 (A, \mu)}$. 
Moreover, $\norm{\mathbf{1}-h}{\Lbb^1 (A, \mu)} = 2 \norm{\mathbf{1}-h'}{\Lbb^1 (A, \mu)}$. Hence, if we get 
the bound~\eqref{eq:Tension} for $h'$, up to dividing $\kappa$ by $2$, we also get the bound~\eqref{eq:Tension} 
for $h$. Hence, without loss of generality, we assume from now on that $h \geq 1/2$.

\smallskip

Let $f \in \overline{B}_{\Lip^\infty (A)} (\mathbf{1}, 1/2) \cap S_{\Lbb^1 (A,\mu)} (\mathbf{0},1)$. Then, on the one hand, for all $h \in \Scal_K$,
\begin{equation}\label{eq:DecroissanceExponentielleBorneInf}
\left| \int_A P_h (f) \dd \mu \right|
= \left| \int_A h f \dd \mu \right| 
\geq \int_A h \dd \mu - \int_A h |\mathbf{1}-f| \dd \mu 
\geq \frac{\norm{h}{\Lbb^1 (A, \mu)}}{2}. 
\end{equation}
On the other hand,
\begin{align}
\label{eq:DecroissanceExponentielleBorneSup}
\begin{split}
\left| \int_A P_h (f) \dd \mu \right| 
& \leq \int_A |f| \dd \mu + \int_A (\mathbf{1}-h) |\mathbf{1}-f| \dd \mu - \int_A (\mathbf{1}-h) \dd \mu \\
& \leq \norm{f}{\Lbb^1 (A, \mu)} -\frac{\norm{\mathbf{1}-h}{\Lbb^1 (A, \mu)}}{2} \\
& = 1-\frac{\norm{\mathbf{1}-h}{\Lbb^1 (A, \mu)}}{2}.
\end{split}
\end{align}
From \eqref{eq:DecroissanceExponentielleBorneInf}, we compute:
\begin{align*}
\norm{\frac{P_h (f)}{\int_A P_h (f) \dd \mu}}{\Lip^\infty (A)}
& \leq \frac{2 \norm{P}{L(\Lip^\infty (A))} \norm{hf}{\Lip^\infty (A)}}{\norm{h}{\Lbb^1 (A,\mu)}} \\
& \leq \frac{2\norm{P}{L(\Lip^\infty (A))}K}{\norm{h}{\Lbb^1 (A, \mu)}}\norm{f}{\Lip^\infty (A)} \\
& \leq 4 \norm{P}{L(\Lip^\infty (A))} K \norm{f}{\Lip^\infty (A)} \\
& \leq 6 \norm{P}{L(\Lip^\infty (A))} K.
\end{align*}

Due to Proposition~\ref{prop:DecroissanceExponentielle}, there exists $m \geq 1$ such that, for any $h$ fitting our assumptions, 
for all $f \in \overline{B}_{\Lip^\infty (A)} (\mathbf{1}, 1/2)$,
\begin{equation*}
\norm{\frac{P^{m-1} P_h (f)}{\int_A P_h (f) \dd \mu}-\mathbf{1}}{\Lip^\infty (A)} 
\leq \frac{1}{2}.
\end{equation*}
We fix such a value of $m$. Then, the following map is well-defined:
\begin{equation*}
F : \ \left\{
\begin{array}{lll}
\overline{B}_{\Lip^\infty (A)} (\mathbf{1}, 1/2) \cap S_{\Lbb^1 (A,\mu)} (\mathbf{0},1) & \to & \overline{B}_{\Lip^\infty (A)} (\mathbf{1}, 1/2) \cap S_{\Lbb^1 (A,\mu)} (\mathbf{0},1) \\
f & \mapsto & \frac{P^{m-1} P_h (f)}{\int_A P_h (f) \dd \mu}
\end{array}
\right. .
\end{equation*}
Furthermore, by virtue of \eqref{eq:DecroissanceExponentielleBorneSup}, for all $n \geq 0$,
\begin{equation*}
\left| \int_A \left( P^{m-1} P_h \right)^n f \dd \mu \right|
= \left| \int_A F^n (f) \dd \mu \cdot \prod_{k=0}^{n-1} \int_A P_h (F^k (f)) \dd \mu \right|
\leq \left(1-\frac{\norm{\mathbf{1}-h}{\Lbb^1 (A, \mu)}}{2}\right)^n.
\end{equation*}
Remark that $0 \leq P_h^n f \leq P^n f$ for all non-negative $f \in \Lbb^1$, 
for all $n \geq 0$ and all $h \in \Scal_K$. In addition, $F$ preserves the subset of real-valued 
functions. Fix $h \in \Scal_K$. Then, for all $n \geq 0$,
\begin{equation*}
0 
\leq \int_A P_h^{nm} (\mathbf{1}) \dd \mu
\leq \int_A \left( P^{m-1} P_h \right)^n (\mathbf{1}) \dd \mu
\leq \left(1-\frac{\norm{\mathbf{1}-h}{\Lbb^1 (A, \mu)}}{2}\right)^n 
\leq e^{-\frac{\norm{\mathbf{1}-h}{\Lbb^1 (A, \mu))}}{2}n},
\end{equation*}
so that:
\begin{equation*}
\norm{P_h^n (\mathbf{1})}{\Lbb^1 (A, \mu)} 
\leq \sqrt{e} \max_{0 \leq k < m} \sup_{h \in \Scal_K} \norm{P_h^k}{L(\Lip^\infty (A))} e^{-\frac{\norm{\mathbf{1}-h}{\Lbb^1 (A,\mu)}}{2m}n}.
\end{equation*}
We have proved that the conclusion of the lemma holds if $(A, \mu, T)$ is assumed to be mixing.

\item Finally, assume that $(A, \mu, T)$ is ergodic but not necessarily mixing. Let $(A_k)_{k \in \Z/M\Z}$ 
be its decomposition in components on which $T^M$ is mixing, and write $\mu_k := M \mu_{|A_k}$. Let $K \geq 0$, and let 
$h \in \Scal_K$. Let $k_0$ be such that $\norm{\mathbf{1}-h}{\Lbb^1 (A_{k_0}, \mu_{k_0})} \geq \norm{\mathbf{1}-h}{\Lbb^1 (A, \mu)}$. 
Note that $h \cdot \mathbf{1}_{A_{k_0}}$ is in $\Scal_K$ when we replace $\Lip^\infty (A, \pi, \lambda)$ by $\Lip^\infty (A_{k_0}, \pi_M, \lambda)$. 
Let $\tilde{P}_h (f) := P^M (hf)$ for $f \in \Lip^\infty (A_{k_0}, \pi_M, \lambda)$.
Then, there exist positive constants $C_0, \kappa_0$ depending only on $K$ such that, for all $n \geq 0$,
\begin{equation*}
\norm{\tilde{P}_h^n (\mathbf{1})}{\Lbb^1 (A_{k_0}, \mu_{k_0})} 
\leq C_0 e^{- \kappa_0 \norm{\mathbf{1}-h}{\Lbb^1 (A_{k_0}, \mu_{k_0})} n} 
\leq C_0 e^{- \kappa_0 \norm{\mathbf{1}-h}{\Lbb^1 (A, \mu)} n}.
\end{equation*}
But then, for all $k \in \Z/M\Z$, for all $n \geq 1$,
\begin{align*}
\norm{P_h^{nM} (\mathbf{1})}{\Lbb^1 (A_k, \mu_k)} 
& \leq \norm{P_h^{(n-1)M} (\mathbf{1})}{\Lbb^1 (A_{k_0}, \mu_{k_0})} \\
& \leq \norm{\tilde{P}_h^{n-1} (\mathbf{1})}{\Lbb^1 (A_{k_0}, \mu_{k_0})} \\
& \leq C_0 e^{- \kappa_0 \norm{\mathbf{1}-h}{\Lbb^1 (A, \mu)} (n-1)} \\
& \leq C_0 e^{\kappa_0} e^{- \kappa_0 \norm{\mathbf{1}-h}{\Lbb^1 (A, \mu)} n},
\end{align*}
so that, for all $n \geq 0$:
\begin{equation*}
\norm{P_h^n (\mathbf{1})}{\Lbb^1 (A, \mu)} 
= \frac{1}{M} \sum_{k \in \Z/M\Z} \norm{P_h^n (\mathbf{1})}{\Lbb^1 (A_k, \mu_k)} 
\leq C_0 e^{2 \kappa_0} e^{- \frac{\kappa_0}{M} \norm{\mathbf{1}-h}{\Lbb^1 (A, \mu)} n}. \qedhere
\end{equation*}
\end{itemize}
\end{proof}

Proposition~\ref{prop:Tension} yields an upper bound on the probability that the orbits do not 
visit a given subset of $A$ before a given time.

\begin{corollary}\label{cor:Tension}

Let $(A, \pi, \lambda, \mu, T)$ be an ergodic Gibbs-Markov map. Let $G$ be a set, and 
$(a_p)_{p \in G}$ be a family of non-trivial $\sigma(\pi)$-measurable subsets. Let $C, \kappa > 0$ 
be constants associated with $\Scal_1$ in Proposition~\ref{prop:Tension}.
Let $K > 0$. Let $(\mu_p)_{p \in G}$ be a family of probability measures on $A$ 
such that $\mu_p \ll \mu$ and $\norm{\de \mu_p / \de \mu}{\Lip^\infty (A)} \leq K$ 
for all $p$. Then, for all $n \geq 0$,
\begin{equation}\label{eq:TensionBis}
\mu_p \left(\bigcap_{k=0}^{n-1} \{T^k (x) \notin a_p \} \right) 
\leq KC e^{- \kappa \mu(a_p) n}.
\end{equation}
\end{corollary}

\begin{proof}

We compute:
\begin{align*}
\mu_p \left(\bigcap_{k=0}^{n-1} \{T^k (x) \notin a_p \} \right) 
& = \int_A \prod_{k=0}^{n-1} \mathbf{1}_{a_p^c} \circ T^k \cdot \frac{\de \mu_p}{\de \mu} \dd \mu \\
& \leq K \int_A P_{\mathbf{1}-\mathbf{1}_{a_p}}^n (\mathbf{1}) \dd \mu.
\end{align*}
But $\mathbf{1}-\mathbf{1}_{a_p} \in \Scal_1$ for all $p$. All remains is to use Proposition~\ref{prop:Tension}.
\end{proof}

\subsection{Convergence in distribution}
\label{subsec:ConvergenceLoiPetitesBoules}

Let $(A, \mu, T)$ be a sufficiently hyperbolic measure-preserving dynamical system, and let $(A_p)$ 
be a family of measurables subsets such that $\lim_{p \to \infty} \mu (A_p)=0$. Let $\varphi_p$ 
be the first hitting time of $A_p$. As $p$ goes to infinity, hitting this set 
becomes a rare event. Knowing that a trajectory has not hit the set until some time 
gives us little information about later times, which implies that any limit distribution exhibits a loss 
of memory characteristic of the exponential distributions. Hence, one can usually prove that 
$\mu (A_p) \varphi_p$ converges in distribution to a exponential random variable of parameter $1$. 
There is an extensive litterature on the subject; we refer the interested reader 
to the reviews~\cite{Coelho:2000, Saussol:2009, Haydn:2013}. Note that this family of results can usually be strenghtened, 
for instance to show convergence to a Poisson process~\cite[Th\'eor\`eme~3.6]{Saussol:1998}. 
More promisingly, there are also ways to get a rate of convergence~\cite{GalvesSchmitt:1997}, 
which may be adapted to get rates of convergence in Theorem~\ref{thm:ConvergenceNp}.

\smallskip

In the previous Subsection, we showed that, under any probability measure with uniformly bounded density, 
the tail of the hitting time of a $\sigma(\pi)$-measurable set decays exponentially, at a speed which is 
at most inversely proportional to the size of the set. Now, we shall prove that, as the 
size of the sets goes to $0$, the distribution of the renormalized hitting time is asymptotically exponential.
This is the content of Proposition~\ref{prop:ConvergenceVersExponentielle}. Due to some specificities 
of our situation (the hitting sets are not exactly cylinders, and the distribution changes with the sets), 
we prove the convergence ourselves, instead of using some already established theorem.

\smallskip

Afterwards, we shall prove Lemma~\ref{lem:BorneSommeLp}, which is useful in the proof of Theorem~\ref{thm:GM} 
and whose proof uses ideas very similar to the proof of Proposition~\ref{prop:ConvergenceVersExponentielle}.

\begin{proposition}\label{prop:ConvergenceVersExponentielle}

Let $(A, \pi, \lambda, \mu, T)$ be an ergodic Gibbs-Markov map. Let $G$ be a locally compact space, 
and $(a_p)_{p \in G}$ be a family of non-trivial $\sigma(\pi)$-measurable subsets such that 
$\lim_{p \to \infty} \mu(a_p) = 0$. For all $p \in G$ and $x \in A$, let $N_p (x) := \inf \{k \geq 0 : T^k (x) \in a_p\}$. 
Let $(\mu_p)_{p \in G}$ be a family of probability measures on $A$ 
such that $\mu_p \ll \mu$ for all $p$, and:
\begin{equation*}
\sup_{p \in G} \norm{\frac{\de \mu_p}{\de \mu}}{\Lip^\infty (A)} 
< + \infty.
\end{equation*}
Then the family of random variables $(\mu (a_p) N_p)_{p \in G}$ 
defined on the probability space $(A, \mu_p)$ converges in distribution to an exponential 
random variable of parameter $1$.
\end{proposition}

\begin{proof}

At first, we assume that the system is mixing. We work with the distribution function 
$\mathbf{1}-F_{p,f}$ of $N_p$ under the distribution $f \de \mu$, that is, 
for all $t \geq 0$:
\begin{equation*}
F_{p,f} (t) 
= \int_A \mathbf{1}_{\{N_p\ge t\}} f \dd \mu.
\end{equation*}
In a first step, we prove that $F_{p,f}$ does not depend too much on the density $f$. This will 
imply the loss of memory: in the second step, we prove that any limit distribution of 
$\mu (a_p) N_p$ is exponential, and that the limit points do not depend on the choice of $f$.
Then, we have to identify the parameter of the limit distribution, which is done in the third and 
fourth steps. In the third step, we prove that some $\Z/2\Z$-extension of the system is ergodic, at least for large 
$p$'s and, in the fourth step, we use Kac's formula to prove that, for a good choice of $f$ (depending on $p$), the 
expectation of $\mu (a_p) N_p$ is $1$.
Finally, in the last step we extend this result to dynamical systems which are merely ergodic.
We assume in the first four steps that $(A, \mu, T)$ is mixing.

\medskip
\begin{itemize}
\item \textbf{Step 1 (mixing case): Loss of memory.} First, let us prove that $F_{p,f}$ 
does not depend on $f$ as $p$ goes to infinity.
Let $h_p := \mathbf{1}-\mathbf{1}_{a_p}$. Then $h_p \in \Scal_1$ for all $p$. Let $K \geq 1$. 
Let $f \in \overline{B}_{\Lip^\infty (A)} (\mathbf{0},K)$ with $f \geq 0$ and $\int_A f \dd \mu = 1$. 
Let $n, k \in\N$ and $p \in G$. Note that $F_{p,f} (n) = \norm{P_{h_p}^n (f)}{\Lbb^1 (A, \mu)}$. 
Since each $P_{h_p}$ is a weak contraction when acting on $\Lbb^1 (A, \mu)$, 
\begin{align*}
\left| F_{p,f} (n+k) - F_{p, P^k f} (n) \right| 
& = \left| \int_A P_{h_p}^{n+k} (f) - P_{h_p}^n (P^k f) \dd \mu \right| \\
& \leq \left| \int_A P_{h_p}^k (f) - P^k f \dd \mu \right| \\
& \leq \mathbf{1}- F_{p,f} (k).
\end{align*}
In addition,
\begin{equation*}
\left| F_{p, P^k f} (n) - F_{p,1} (n) \right| 
\leq \norm{P^k f-1}{\Lbb^\infty} 
\leq K \norm{P^k - \Pi_0}{L(\Lip^\infty (A))},
\end{equation*}
and:
\begin{equation*}
\mathbf{1}-F_{p,f} (k) 
\leq k \norm{f}{\Lbb^\infty (A, \mu)} \mu(a_p) 
\leq Kk\mu(a_p).
\end{equation*}
Hence, we finally get:
\begin{equation*}
\left| F_{p,f} (n+k) - F_{p, 1} (n) \right| 
\leq Kk\mu(a_p)+K \norm{P^k - \Pi_0}{L(\Lip^\infty (A))}.
\end{equation*}
Since $(A, \mu, T)$ is a mixing Gibbs-Markov map, $\norm{P^k - \Pi_0}{L(\Lip^\infty (A))}$ 
converges to $0$ as $k$ goes to infinity (Proposition~\ref{prop:DecroissanceExponentielle}). Taking 
$n = \lfloor \mu(a_p)^{-1} t \rfloor$ and $k := \lfloor \sqrt{\mu(a_p)^{-1}} \rfloor$ yields:
\begin{equation}
\label{eq:PerteDeMemoire}
F_{p,f} \left( \lfloor \mu(a_p)^{-1}t \rfloor + \lfloor \sqrt{\mu (a_p)}) \rfloor \right)
= F_{p, 1} (\lfloor \mu(a_p)^{-1} t \rfloor) + o(1) \text{ as } p \to \infty,
\end{equation}
uniformly for $f$ in $\overline{B}_{\Lip^\infty (A)} (0,K)$ and $t \geq 0$.

\medskip
\item \textbf{Step 2 (mixing case): Limit distributions.} 
Now, we prove that any limit distribution of $\mu(a_p) N_p$ is $\delta_0$ or exponential, 
and that the limit distributions do not depend on the choice of the measures $\mu_p$.
For every $p\in G$, we set $g_p$ for the density of $\mu_p$ with 
respect to $\mu$. By Corollary~\ref{cor:Tension}, there exist positive constants 
$C$, $\kappa$ such that, for all $t \geq 0$ and for all $p \in G$,
\begin{equation*}
\mu_p (\mu(a_p)N_p \geq t) 
= \int_A \mathbf{1}_{\{\mu(a_p)N_p \geq t\}} g_p \dd\mu 
\leq CKe^{-\kappa t}.
\end{equation*}
Hence, the sequence $(\mu(a_p)N_p)_{p \in G}$ defined on $(A, \mu_p)$ is tight. Let $F$ 
be the tail distribution function of one of its limit points, and let $G_F \subset G$ 
be such that the distribution function of $\mu(a_p)N_p$ converges to $F$ for $p \in G_F$. 
By Equation~\eqref{eq:PerteDeMemoire}, $F$ does not depend on $f$. Note that $F$ is non-increasing 
and c\`adl\`ag. 

\smallskip

If $F (t) = 0$ for all $t >0$, then the limit distribution is $\delta_0$, and we are done. 
Let us assume that there exists $T > 0$ with $F(T) > 0$, and let $t \in [0,T)$. 
Then $F_{p,1} (\lceil \mu(a_p)^{-1}t\rceil) > 0$ for all large enough $p \in G_F$.
We apply Lemma~\ref{lem:EncoreUneBorneUniforme} with the stopping 
time $n_{p}(t):=\lceil \mu(a_p)^{-1}t\rceil$ and the event $A := \bigcap_{k=0}^{n_p(t)-1} T^{-k} a_p^c$, 
which has positive probability if $p$ is large enough. There exists a constant $K'$ 
such that $P_{h_p}^{n_p(t)} (\mathbf{1})/ F_{p,\mathbf{1}} (n_p(t))$ belongs to $\overline{B}_{\Lip^\infty (A)} (\mathbf{0},K')$ 
for all large engouh $p$. But then, for all $k\in\N_+$ and for all $p \in G_F$:
\begin{equation*}
F_{p,\mathbf{1}} \left(n_p(t)+\lfloor \sqrt{\mu(a_p)^{-1}}\rfloor+k\right)
= F_{p,\mathbf{1}} (n_p(t)) \cdot F_{p,\frac{P_{h_p}^{n_p(t)} (\mathbf{1})}{F_{p,\mathbf{1}} (n_p(t))}} (\lfloor \sqrt{\mu(a_p)^{-1}}\rfloor+k).
\end{equation*}
Let $t' \geq 0$ and $k = \lceil \mu(a_p)^{-1}t' \rceil$. 
Letting $p$ go to infinity in $G_F$, by Equation~\eqref{eq:PerteDeMemoire},
\begin{equation*}
F(t+t') 
= F(t)F(t').
\end{equation*}
In addition, trivially, $F=1$ on $\R_-$. Hence, $\mathbf{1}-F$ is the distribution function 
of an exponential random variable with parameter in $[0,\infty]$.

\medskip
\item \textbf{Step 3 (mixing case): Ergodicity of a $\Z/2\Z$-extension.}
We have proved that any limit distribution of $(\mu(a_p)N_p)_{p \in G}$ is exponential; now, we show that 
its parameter must be $1$. To this end, we first prove that a certain $\Z/2\Z$-extension is ergodic. This fact 
shall allow us to apply Kac's formula in the next step, and from there to identify the parameter 
of the limit exponential distribution.

\smallskip

Consider the dynamical system:
\begin{equation*}
T_p : \left\{
\begin{array}{lll}
A \times \Z/2\Z & \to & A \times \Z/2\Z \\
(x,q) & \mapsto & \left\{ 
\begin{array}{ll}
(T(x),q) & \text{ if } x \notin a_p, \\
(T(x),q+1) & \text{ otherwise}
\end{array}
\right.
\end{array}
\right. .
\end{equation*}
Let $\pi_p$ be the canonical projection from $A \times \Z/2\Z$ onto $A$, which is a factor map. We shall prove that 
this extension is ergodic for all large enough $p$. The idea is that otherwise, we could divide 
$A$ into two subsets which communicate only through $a_p$; as the $a_p$ get smaller, this would make 
the communication more difficult, and the mixing arbitrarily slow, which is absurd.

\smallskip

Assume that $(A \times \Z/2\Z, \mu \otimes (\delta_0+\delta_1)/2, T_p)$ is not ergodic. 
Let $I_p$ be a $T_p$-invariant non-trivial measurable subset. Then, since $\pi_p(I_p) = \pi_p \circ T_p (I_p) = T \circ \pi_p(I_p)$, 
we see that $\pi_p(I_p)$ is a non-trivial $T$-invariant subset, so $\pi_p(I_p) = A$. 
Doing the same with $I_p^c$, we see that there exists a measurable partition $(I_{p,0}, I_{p,1})$ 
of $A$ such that $I_p = I_{p,0} \times \{0\} \cup I_{p,1} \times \{1\}$. In addition, neither 
$A \times \{0\}$ nor $A \times \{1\}$ are $T_p$-invariant, so $I_p$ cannot be either, 
and neither $I_{p,0}$ nor $I_{p,1}$ are trivial. Finally, since the $\Z/2\Z$-extension is 
still Gibbs-Markov, its partition into ergodic components is coarser than its underlying 
partition, so both $I_{p,0}$ and $I_{p,1}$ are $\sigma(\pi)$-measurable.

\smallskip

The map $T_p$ sends $I_{p,0} \cap a_p$ into $I_{p,1}$ and $I_{p,1}\cap a_p$ into $I_{p,0}$.
By the big image 
property of Gibbs-Markov maps, there exists a constant $m > 0$ 
such that $\mu(I_{p,i}) \geq m$ for all $p \in G$ and $i \in \Z/2\Z$.
Let $f_p := \mu(I_{p,0})^{-1} \mathbf{1}_{I_{p,0}}$. Then $(f_p)_{p \in G}$ is uniformly bounded in 
$\Lip^\infty (A)$ by $m^{-1}$. Hence, there exist constants $C', \kappa' > 0$ such that 
$\norm{P^n f_p-\mathbf{1}}{\Lip^\infty(A)} \leq C'e^{-\kappa' n}$ for all $p$, $n$. Hence, 
\begin{equation*}
\int_{I_{p,1}} P^n f_p \dd \mu 
\geq m(1  -C'e^{-\kappa' n}).
\end{equation*}
But we know that:
\begin{eqnarray*}
\int_{I_{p,1}} P^n f_p \dd \mu 
& = & \mu (T^{-n} I_{p, 1} |  I_{p,0}) \leq \sum_{k=0}^{n-1} \mu\left(\left.T^{-(k+1)}I_{p,1} \right|T^{-k}I_{p,0} \right) \\
& \leq & n \mu(T^{-1} I_{p, 1} |  I_{p,0}) \leq n\frac{\mu(a_p\cap I_{p,0})}{\mu(I_{p,0})} \leq m^{-1} \mu(a_p) n.
\end{eqnarray*}
There is a contradiction for some $n \geq 0$ and all large enough $p \in G$. 

\medskip
\item \textbf{Step 4 (mixing case): Computation of the parameter of the exponential distribution.}
Now, let us apply Kac's formula. For all large enough $p$, the system $(A \times \Z/2\Z, \mu \otimes (\delta_0+\delta_1)/2, T_p)$ is ergodic. 
Let $\varphi_p$ be the first return time for $T_p$ to $A \times \{0\}$ starting from $A \times \{0\}$. 
By Kac's formula,
\begin{equation*}
\int_A \varphi_p \dd \mu 
=2.
\end{equation*}
But $\varphi_p \equiv 1$ on $a_p^c$, and $\varphi_p \equiv 1+N_p \circ T$ on $a_p$. Hence,
\begin{equation*}
1 
= \int_{a_p} N_p \circ T \dd \mu 
= \int_A N_p \cdot P (\mathbf{1}_{a_p}) \dd \mu 
= \int_A (\mu (a_p) N_p) \cdot \frac{P (\mathbf{1}_{a_p})}{\mu (a_p)} \dd \mu.
\end{equation*}
Let $X$ be a limit in distribution of $(\mu (a_p) N_p)_{p \in G}$, and let $G_X \subset G$ 
be such that $(\mu (a_p) N_p)_{p\in G_X}$ converges to $X$ in distribution.
We already know that $X$ has an exponential distribution of parameter at most $\kappa$.
By Lemma~\ref{lem:EncoreUneBorneUniforme}, using the stopping time $1$, 
there exists a constant $K$ such that, for all $p \in G$, the density $P (\mathbf{1}_{a_p}) / \mu (a_p)$ 
lies in $\overline{B}_{\Lip^\infty (A)}(0,K)$. Hence, due to \eqref{eq:PerteDeMemoire}, 
the limit distribution of $(\mu (a_p) N_p)$ on $(A, \mu(a_p)^{-1} P (\mathbf{1}_{a_p}))$ 
is the limit distribution of $(\mu (a_p) N_p)$ on $(A, \mu)$, that is, the distribution of 
$X$. Furthermore, the tail of 
$(\mu (a_p) N_p)$ on $(A, \mu(a_p)^{-1} P (\mathbf{1}_{a_p}))$ is dominated by a 
decaying exponential, so all the moments converge to those of $X$. In particular, 
$\Ebb [X] = 1$, so $X$ follows an exponential distribution of parameter $1$. 

\medskip
\item \textbf{Step 5: General case.}
We have proved the proposition under the assumption that $(A, \mu, T)$ is mixing. 
Now, let us assume that the system is only ergodic, but not mixing. Let $M \geq 1$ and 
$(A_k)_{k \in \Z/M\Z}$ be as in Proposition~\ref{prop:IterationMelangeante}.
Let $(a_p, \nu_p)$ be a sequence satisfying the hypotheses of the proposition. 
Let $k \in \Z/M\Z$, and let $(\overline{\nu}_p)$ be a sequence of probability measures 
on $A_k$, absolutely continuous with respect to $\mu_k := \mu (\cdot | A_k)=M\mu(\cdot\cap A_k)$, and with densities uniformly bounded in 
$\Lip^\infty (A_k, \pi_M, \lambda)$. We define:
\begin{equation*}
\overline{a}_p 
:= \{x \in A_k : \ \exists \ 0 \leq i < M, \ T^i (x) \in a_p \}
\in \pi_M.
\end{equation*}
Note that $\mu_k (\overline{a}_p)\leq M\sum_{i=0}^ {M-1}\mu(A_{k+i}\cap a_p) = M \mu(a_p)$. Let $0 \leq i_1 < i_2 < M$. 
Then, by \cite[Lemme~1.1.13]{Gouezel:2008e}, $P^{i_2-i_1}$ maps continuously $\Lip^1$ into $\Lip^\infty$, and:
\begin{align*}
\mu_k \left(T^{-i_1} (a_p \cap A_{k+i_1}) \cap \right. & \left. T^{-i_2} (a_p \cap A_{k+i_2}) \right) \\
& \leq  \int_{A_{k+i_1}} P^{i_2-i_1} \mathbf{1}_{a_p \cap A_{k+i_1}} \cdot \mathbf{1}_{a_p \cap A_{k+i_2}} \dd \mu_{k+i_1} \\
& \leq C \mu_{k+i_1} (a_p ) \mu_{k+i_2} (a_p ).
\end{align*}
and so, by Bonferroni's inequality,
\begin{align*}
\mu_k (\overline{a}_p) 
& \geq \sum_{0 \leq i < M} \mu_k (T^{-i} (a_p \cap A_{k+i}))- \sum_{0 \leq i_1 <  i_2 < M} \mu_k \left(T^{-i_1} (a_p \cap A_{k+i_1}) \cap T^{-i_2} (a_p \cap A_{k+i_2}) \right) \\
& \geq M \mu (a_p) - \frac{CM^4}{2} \mu (a_p)^2.
\end{align*}
Hence, $\mu_k (\overline{a}_p) \sim M \mu (a_p)$.

\smallskip

Let $\overline{N}_p$ be the first hitting time of $\overline{a}_p$ for $T^M$. 
Note that $|N_p - M \overline{N}_p| \leq M-1$ on $A_k$. 
Since Proposition~\ref{prop:ConvergenceVersExponentielle} holds for mixing transformations, 
the sequence $(\mu_k (\overline{a}_p) \overline{N}_p)_{p \in G}$ defined on $(A_k, \overline{\nu}_p)$ 
converges in distribution to an exponential random variable of parameter $1$. 
But $\mu_k (\overline{a}_p) \sim M \mu (a_p)$ and $M \overline{N}_p = N_p + O(1)$, 
so $(\mu (a_p) N_p)_{p \in G}$ defined on $(A_k, \overline{\nu}_p)$ converges in distribution 
to the same exponential random variable of parameter $1$.

\smallskip

Finally, let $(\nu_p)_p$ be a sequence of probability measures on $A$ whose densities $(h_p)_p$ with respect to 
$\mu$ are bounded in $\Lip^\infty (A, \pi, \lambda)$. For any $x \in A$, let $0 \leq i < M$ be such that 
$T^i (x) \in A_k$, and set $P(x) := (x, T^i (x)) \in A \times A_k$. Then $\nu_p\mapsto\bar\nu_p:=P_* \nu_p$ is a transference plan 
between $\nu_p$ and a probability measure $\overline{\nu}_p$ on $A_k$, with density:
\begin{equation*}
\overline{h}_p
:= \frac{\de \overline{\nu}_p}{\de \mu_k}
= \frac{1}{M} \sum_{i=0}^{M-1} P^i (\mathbf{1}_{A_{k-i}} h_p).
\end{equation*}
This transference plan yields a coupling between $N_p$ (seen as a random variable on $(A, \nu_p)$) 
and $N_p$ (seen as a random variable on $(A_k, \overline{\nu}_p)$). For the sake of clarity, we shall call 
the second random variable $\widetilde{N}_p$.

\smallskip

The sequence $(\overline{h}_p)$ is bounded in $\Lip^\infty (A_k, \pi_M, \lambda)$. Hence, 
$(\mu (a_p) \widetilde{N}_p)_{p \in G}$ converges in distribution to an exponential random variable 
of parameter $1$.

\smallskip

Let $x \in A$. Let $0 \leq i < M$ be such that $T^i (x) \in A_k$. If $N_p (x) \geq i$, then 
$N_p (x) = i+ N_p (T^i (x))$, so $N_p = i+\widetilde{N}_p$. The event $\{N_p < i\}$ has probability 
$O (\mu (a_p))$, and $|\mu(a_p) N_p - \mu(a_p) \tilde{N}_p| \leq M \mu (a_p)$ outside of this event, 
so $(\mu (a_p) N_p)_{p \in G}$ has the same limit in distribution as $(\mu (a_p) \widetilde{N}_p)_{p \in G}$. 
\qedhere
\end{itemize}
\end{proof}

\begin{remark}
In our applications, $a_p$ will be the set of points $x\in A$ such that the trajectory 
$(S_n F (x))_{n \geq 0}$ of $(x,0)$ under the action of $\widetilde{T}$ 
goes to $A \times \{p\}$ before coming back to $A \times \{0\}$. If the $\Z^d$-extension 
is ergodic, then the $\Z/2\Z$-extension used in the proof is also automatically ergodic, as it is the induced 
system on $A \times \{0, p\}$. Hence, the stage in the proof above where we proved that such a 
$\Z/2\Z$-extension is ergodic for all large enough $n$ is not necessary for our applications. 
This detour however made for a cleaner and more general statement in the proposition.
\end{remark}

The following lemma allows us to control the $\Lbb^q (A, \mu)$ norm of the Birkhoff sum of an observable 
until $N_p$.

\begin{lemma}\label{lem:BorneSommeLp}

Let $(A, \pi, \lambda, \mu, T)$ be an ergodic Gibbs-Markov map. Let $G$ be a locally compact space, 
and $(a_p)_{p \in G}$ be a family of non-trivial $\sigma(\pi)$-measurable subsets such that 
$\lim_{p \to \infty} \mu(a_p) = 0$. For all $p \in G$ and $x \in A$, let $N_p (x) := \inf \{k \geq 0 : T^k (x) \in a_p\}$. 
Let $(\mu_p)_{p \in G}$ be a family of probability measures on $A$ 
such that $\mu_p \ll \mu$ for all $p$. Let $C >1$. Then for all $q \in [1, \infty)$, for all $f \in \Lbb^q (A, \mu)$, 
for all large enough $p \in G$,
\begin{equation}
\norm{\sum_{k=0}^{N_p-1} f \circ T^k}{\Lbb^q (A, \mu_p)} 
\leq C q \alpha (p) \sup_{p \in G} \norm{\frac{\de \mu_p}{\de \mu}}{\Lbb^\infty (A, \mu)} \norm{f}{\Lbb^q (A, \mu)}.
\end{equation}
\end{lemma}

\begin{proof}

Let $\varepsilon >0$. Let $f \in \Lbb^q (A, \mu)$, which we can assume without loss of generality to be non-negative. 
Fix $p \in G$ and $\Ncal > (1+\varepsilon) N > 0$, such that $\varepsilon N$ is a multiple of the period $M$ of the Gibbs-Markov map.
Define $N_p' (x) := \inf \{n \geq 0: \ n \notin [N, 2N), \ T^n (x) \in a_p\} \geq N_p$. Then:
\begin{align*}
\norm{\sum_{k=0}^{\Ncal \wedge N_p-1} f \circ T^k}{\Lbb^q (A, \mu)} 
& \leq \norm{\sum_{k=0}^{\Ncal \wedge N_p'-1} f \circ T^k}{\Lbb^q (A, \mu)} \\
& \leq \norm{\sum_{k=0}^{((1+\varepsilon)N) \wedge N_p'-1} f \circ T^k}{\Lbb^q (A, \mu)} + \norm{\sum_{k=(1+\varepsilon)N}^{\Ncal \wedge N_p'-1} f \circ T^k}{\Lbb^q (A, \mu)} \\
& \leq (1+\varepsilon) N \norm{f}{\Lbb^q (A, \mu)} + \norm{\sum_{k=2N}^{(2N+\Ncal) \wedge N_p'-1} f \circ T^k}{\Lbb^q (A, \mu)}.
\end{align*}
Now, focus on the right hand-side. We get:
\begin{equation*}
\Ebb_\mu \left[ \left( \sum_{k=2N}^{((1+\varepsilon)N+\Ncal) \wedge N_p'-1} f \circ T^k \right)^p \right] 
= \Ebb_\mu \left[ \left( \sum_{k=0}^{\Ncal \wedge N_p-1} f \circ T^k \right)^p P^{(1+\varepsilon)N} (\mathbf{1}_{N_p \geq N}) \right].
\end{equation*}
By Lemma~\ref{lem:EncoreUneBorneUniforme}, applied to the stopping time whose value is $N-1$ if $N_p < N$ (and $+\infty$ otherwise), 
and to the set $A := \{N_p < N\}$, we get $\norm{P^{N-1} (\mathbf{1}_{N_p \geq N})}{\Lip^\infty (A)} \leq K$. Hence:
\begin{align*}
\norm{P^{(1+\varepsilon)N} (\mathbf{1}_{N_p \geq N})}{\Lbb^\infty (A, \mu)} 
& \leq \left(1+KC\rho^{-\frac{\varepsilon N}{M}}\right) \norm{P^{N-1} (\mathbf{1}_{N_p \geq N})}{\Lip^1 (A)} \\
& = \left(1+KC\rho^{-\frac{\varepsilon N}{M}}\right) \norm{\mathbf{1}_{N_p \geq N}}{\Lbb^1 (A, \mu)} \\
& = \left(1+KC\rho^{-\frac{\varepsilon N}{M}}\right) \mu (N_p \geq N),
\end{align*}
whence:
\begin{equation*}
\norm{\sum_{k=0}^{\Ncal \wedge N_p-1} f \circ T^k}{\Lbb^q (A, \mu)} 
\leq (1+\varepsilon)N \norm{f}{\Lbb^q (A, \mu)} + \left(1+KC\rho^{-\frac{\varepsilon N}{M}}\right)^{\frac{1}{q}} \mu (N_p \geq N)^{\frac{1}{q}} \norm{\sum_{k=0}^{\Ncal \wedge N_p-1} f \circ T^k}{\Lbb^q (A, \mu)}.
\end{equation*}
We choose $N (p) \sim \varepsilon \alpha (p)$. Then $\rho^{-\frac{\varepsilon^2 N(p)}{M}}$ converges to $0$, 
while by Proposition~\ref{prop:ConvergenceVersExponentielle}, $\mu (N_p \geq N(p))$ 
converges to $e^{-\varepsilon} < 1$. For all large enough $p$, this yields:
\begin{equation*}
\norm{\sum_{k=0}^{\Ncal \wedge N_p-1} f \circ T^k}{\Lbb^q (A, \mu)} 
\leq \frac{\varepsilon \alpha (p) (1+\varepsilon+o(1))}{1-e^{-\frac{\varepsilon}{q}}} \norm{f}{\Lbb^q (A, \mu)}.
\end{equation*}
The $o(1)$ is independent from $\Ncal$. We choose $\varepsilon$ small enough that $\varepsilon (1+ 2\varepsilon) < C q (1-e^{-\frac{\varepsilon}{q}})$, 
and then take the limit as $\Ncal$ goes to infinity. 
Finally, notice that $\de \mu_p / \de \mu$ is uniformly bounded (in $\Lbb^\infty (A, \mu)$ norm and in $p$), 
so that this inequality, up to the constant $\sup_{p \in G} \norm{\frac{\de \mu_p}{\de \mu}}{\Lbb^\infty (A, \mu)}$, 
extends to $\norm{\sum_{k=0}^{N_p-1} f \circ T^k}{\Lbb^q (A, \mu_p)}$.
\end{proof}

\subsection{Hitting probabilities and limit theorems}
\label{subsec:ProbabiliteAtteinte}

In this subsection, we work with ergodic, discrete Abelian, Markov extensions of Gibbs-Markov maps. Let $G$ be an 
infinite countable Abelian group. Let $(A, \pi, \lambda, \mu, T)$ be a Gibbs-Markov map, and let 
$F : \ A \to G$ be $\sigma (\pi)$-measurable. We shall assume that the associated extension 
$(\widetilde{A}, \tilde{\mu}, \widetilde{T})$ is conservative and ergodic.

\smallskip

First, we shall relate the probability that an excursion from $0$ hits a specific point $p$ 
with the moments of the time spent in $p$. This is where the results from Subsections~\ref{subsec:Tightness} 
and \ref{subsec:ConvergenceLoiPetitesBoules} are used directly.

\smallskip

For $p \in G$, let $A_p := \{x \in A : \widetilde{T}_{\{0, p\}} (x,0) \in A \times \{p\}\}$ be the set of points $x$ such that 
the excursion starting from $(x,0)$ reaches $A \times \{p\}$ before $A \times \{0\}$. Let $\alpha (p) := \mu (A_p)^{-1}$. 
The function $\alpha$ is well-defined because the extension is conservative and ergodic. The next lemma asserts that it 
converges to infinity as $p$ goes to infinity.

\begin{lemma}
Let $G$ be an infinite countable Abelian group.
Let $(\widetilde{A}, \tilde{\mu}, \widetilde{T})$ be a conservative and ergodic Markov $G$-extension 
of a measure-preserving dynamical system $(A, \mu, T)$. Then $\lim_{p \to \infty} \alpha (p) = + \infty$.
\end{lemma}

\begin{proof}
Let $(K_n)_{n \geq 0}$ be an exhaustion of $G$ by an increasing sequence of finite subsets of $G$. 
For all $x \in A$ such that $\varphi_{\{0\}} (x)$ is finite, set:
\begin{equation*}
N (x) 
:= \max_{0 \leq k < \varphi_{\{0\}} (x)} \min \{n \geq 0 : \ \widetilde{T}^k (x,0) \in A \times K_n\}.
\end{equation*}
Then $A = \bigcup_{n \geq 0} N^{-1} (n)$ up to set of measure 0, so that $\lim_{n \to + \infty} \mu \left(N>n \right) 
= 0$.
But, if $p \notin K_n$, then $A_p \subset \{ N>n\}$, 
so $\lim_{n \to + \infty} \sup_{p \in K_n^c} \mu(A_p) = 0$,
i.e. $\lim_{n \to + \infty} \inf_{p \in K_n^c} \alpha(p) = +\infty$.
\end{proof}

Let us go back to the study of the local time. Recall that, for $p \in G $ and $x \in A$, we set:
\begin{equation*}
f_{p,\{0\}} (x) 
:= N_p (x) -1 
= \left( \sum_{k = 0}^{\varphi_{\{0\}} (x)-1} \mathbf{1}_{\{S_k F (x) = p\}} \right)-1,
\end{equation*}
which is the difference between the time spent in $A \times \{p\}$ and $A \times \{0\}$ in the excursion starting from 
$(x,0)$. Our next goal in this subsection is to evaluate the tail and moments of $f_{p,\{0\}}$ as $p$ goes to infinity.

\begin{proposition}\label{prop:DistributionLp}

Let $(A, \pi, d, \mu, T)$ be a Gibbs-Markov map, and $G$ be a countable Abelian group. Let 
$(\widetilde{A}, \tilde{\mu}, \widetilde{T})$ be a conservative and ergodic Markov $G$-extension of $(A, \pi, d, \mu, T)$. 

\smallskip

The conditional distributions $\alpha(-p)^{-1} N_p | \{N_p > 0\}$ have exponential tails, 
uniformly in $p$. In addition, $\alpha(-p)^{-1} N_p$, seen as a random variable on $(A, \mu (\cdot | A_p))$, 
converges in distribution and in moments to an exponential distribution of parameter $1$.
\end{proposition}

\begin{proof}

The random variable $N_p (x)$ counts the time the process starting from $(x,0)$ 
spends in $p$ before going back to $0$. On $A_p$, it is positive. For $x$ in $A_p$, 
let $T_p(x)$ be such that $\widetilde{T}_{\{0,p\}} (x,0) = (T_p (x),p)$. Then, on $A_p$,
\begin{align*}
N_p (x) 
& = \inf \{k \geq 1 : \ \widetilde{T}_{\{0,p\}}^k (T_p(x),p) \in A\times\{0\} \} \\
& = 1+\inf \{k \geq 0 : \ \widetilde{T}_{\{0,p\}}^k (T_p(x),p) \in A_{-p} \times\{p\} \}.
\end{align*}

But, if $y \notin A_{-p}$, then the first return time of $(\widetilde{T}^k (y,p))$ to $A \times \{0,p\}$ 
is the first return time of $(\widetilde{T}^k (y,p))$ to $A \times \{p\}$. Hence, 
$\widetilde{T}_{\{0,p\}} (y,p) = (\widetilde{T}_{\{0\}} (y),p)$, and:
\begin{align*}
N_p (x) 
= 1+\inf \{k \geq 0 : \ \widetilde{T}_{\{0\}}^k (T_p(x))  \in A_{-p} \}.
\end{align*}
Let $N_p^{(0)}$ be the hitting time of $A_{-p}$ for the process $(\widetilde{T}_{\{0\}}^k (x))_{k \geq 0}$. Then the random variable $N_p$ 
seen on $(A, \alpha (p) \mathbf{1}_{A_p} \dd \mu)$ has the same distribution as the random variable $\mathbf{1}+N_p^{(0)}$ seen on 
$(A, \alpha (p) P_{\{0,p\}} \mathbf{1}_{A_p} \dd \mu)$. We write $\pi_{\{0\}}:=\pi_{\varphi_{\{0\}}}$. In addition, each $A_{-p}$ is non trivial (as the extension is 
conservative and ergodic), and each $A_{-p}$ is $\sigma (\pi_{\{0\}})$-measurable (because $\sigma (\pi_{\{0\}})$ 
contains all the information about the sites visited in an excursion, and in particular whether $-p$ is visited or not).

\smallskip

Due to Lemma~\ref{lem:EncoreUneBorneUniforme} with the stopping time $\varphi_{\{0,p\}}$, 
the sequence of densities $(\alpha (p) P_{\{0,p\}} \mathbf{1}_{A_p})_{p \in G \setminus \{0\}}$ 
is uniformly bounded in $\Lip^\infty (A, \pi_{\{0,p\}},\lambda)$. Since $\pi_{\{0\}} \leq \pi_{\{0,p\}}$, 
it is also uniformly bounded in $\Lip^\infty (A, \pi_{\{0\}},\lambda)$.
By Proposition~\ref{prop:ConvergenceVersExponentielle}, the sequence of random variables $\mu(A_{-p}) N_p(\cdot)$ seen on 
$(A, \alpha (p) P_{\{0,p\}} \mathbf{1}_{A_p} \dd \mu)$ converges in distribution to an exponential random variable of 
parameter $1$. By Corollary~\ref{cor:Tension}, this sequence of random variables is also exponentially tight, 
so it converges in moments, which proves the first part of Proposition~\ref{prop:DistributionLp}.
Since $(\alpha(-p))_{p \in G}$ goes to infinity as $p$ goes to infinity, $(\alpha(-p)^{-1} N_p)_{p \in G}$, with respect to $(\mu (\cdot | A_p))_{p \in G}$, 
converges in distribution and in moments to an exponential 
random variable of parameter $1$.
\end{proof}

Proposition~\ref{prop:DistributionLp} yields directly a rough description of the distribution of $f_{p,\{0\}}$ for large $p$'s: 
it is $-1$ with probability $1-\alpha (p)^{-1}$, and an exponential random variable of parameter $\alpha (-p)$ 
on the remaining set. This is part of Theorem~\ref{thm:ConvergenceNp}.

\begin{proof}[Proof of Theorem~\ref{thm:ConvergenceNp}]

Let $(\widetilde{A},\tilde{\mu},\widetilde{T})$ be a conservative and ergodic Markov 
$\Z^d$-extension of a Gibbs-Markov map $(A, \mu, T)$. We prove the second item, then the 
third, and we finish by the first item.

\smallskip

Let $p \in \Z^d \setminus \{0\}$. The dynamical system $(A \times \{0,p\}, \mu \otimes (\delta_0+\delta_p)/2, T_{\{0,p\}})$ 
is ergodic. In this new system, by Kac's formula,
\begin{equation*}
\int_{A \times \{0\}} (1+N_p) \dd \frac{\mu \otimes (\delta_0+\delta_p)}{2} 
= \int_{A \times \{0\}} \varphi_{\{0\}} \dd \frac{\mu \otimes (\delta_0+\delta_p)}{2} 
= 1.
\end{equation*}
Hence, 
\begin{equation*}
1+\Ebb_\mu [N_p] 
= \int_A 1+N_p \dd \mu 
= 2,
\end{equation*}
so $\Ebb_\mu [N_p] = 1$. But, by Proposition~\ref{prop:DistributionLp},
\begin{equation*}
\Ebb_\mu [N_p] 
= \frac{\alpha (-p)}{\alpha(p)} \Ebb_\mu [\alpha(-p)^{-1} N_p|A_p] 
\sim_{p \to \infty} \frac{\alpha (-p)}{\alpha(p)}.
\end{equation*}
Hence, $\alpha (p) \sim \alpha (-p)$ as $p$ goes to infinity. Together with Proposition~\ref{prop:DistributionLp}, 
this yields the second item of Theorem~\ref{thm:ConvergenceNp}.

\smallskip

Let $q > 1$, and apply Proposition~\ref{prop:DistributionLp} to the moments of order $q$ of $N_p$. This yields:
\begin{equation*}
\norm{f_{p,\{0\}}+1}{\Lbb^q (A, \mu)}^q
= \int_{A_p} N_p^q \dd \mu 
= \alpha(p)^{-1} \alpha(-p)^q \norm{\alpha(-p)^{-1} N_p}{\Lbb^q (A, \mu (\cdot | A_p))}^q
\sim \alpha(p)^{q-1} \alpha(-p)^q \Ebb [\Ecal^q],
\end{equation*}
where $\Ecal$ is a random variable with an exponential distribution of parameter $1$. 
Finally, we use the fact that $\alpha(-p) \sim \alpha (p)$ and that $\Ebb [\Ecal^q] = \Gamma (1+q)$ 
to get the third item of Theorem~\ref{thm:ConvergenceNp}.

\smallskip

We have proved that $\alpha(p) \sim \alpha(-p) \sim \Ebb_\mu [f_{p,\{0\}}^2]/2$. 
Due to Proposition~\ref{prop:DistributionLp}, $\alpha(-p)^{-1} \Ebb_\mu[N_p|N_p>0] \sim \alpha(p)^{-1} \Ebb_\mu[N_p|N_p>0] = \Ebb_\mu[N_p] = 1$.
Due to Proposition~\ref{prop:ConvergenceVersExponentielle} and~\ref{prop:Tension}, the random variable $N_{0,p}$, 
which is the first hitting time of $A_p$ for $T_{\{0\}}$, once divided by $\alpha (p)$, converges in distribution and in moments to an exponential random 
variable of parameter $1$. Hence,
\begin{equation*}
\Ebb_\mu [N_{0,p}] 
\sim \alpha (p),
\end{equation*}

\smallskip

Now let us prove the link with $\sigma_{GK}^2 (f_{p,\{0\}}, A, \mu, T_{\{0\}})$.
Note that $f_{p,\{0\}}$ is constant on elements on $\pi_{\{0\}}$, and that $\norm{f_{p,\{0\}}}{\Lbb^1 (A,\mu)} \leq 1+\norm{N_p}{\Lbb^1 (A, \mu)} \leq 2$. 
Hence, $\norm{f_{p,\{0\}}}{\Lip^1 (A, \pi_{\{0\}},\lambda,\mu)}$, as a function of $p$, is bounded. Since $P_{\{0\}}$ sends 
$\Lip^1 (A,\pi_{\{0\}},\lambda,\mu)$ continuously into $\Lip^\infty (A,\pi_{\{0\}},\lambda,\mu)$, and $P_{\{0\}}^M$ contracts exponentially 
fast on the subspace of functions in $\Lip^\infty (A,\pi_{\{0\}},\lambda,\mu)$ with zero average on each of the $M$ ergodic components of $T^M$, 
all the terms $\int_A f_{p,\{0\}} \circ \widetilde{T}_{\{0\}}^k \cdot f_{p,\{0\}} \circ \widetilde{T}_{\{0\}}^\ell \dd \mu$ 
with $k \neq \ell$ have a bounded contribution. Hence,
\begin{equation*}
\sup_{p \in G} \left|\sigma_{GK}^2 (f_{p,\{0\}}, A, \mu, \widetilde{T}_{\{0\}}) - \Ebb_\mu [f_{p,\{0\}}^2] \right|
<+\infty.
\end{equation*}
Note that, if the system is a random walk, then $P_{\{0\}}$ sends any function which is constant on elements of the partition to its average, 
which is $0$ for $f_{p,\{0\}}$. In this case, the supremum above is actually $0$.
\end{proof}

\subsection{Proof of Theorem~\ref{thm:GM}}
\label{subsec:PreuveGM}

In this section we prove Theorem~\ref{thm:GM}. Our goal is mostly to get a more explicit integrability condition 
in the statement of~\cite[Theorem~6.8]{Thomine:2015}. We first give a lemma which gives a good sufficient condition 
for this integrability condition to hold.

\begin{lemma}\label{lem:sommabi}

Let $(\widetilde{A},\tilde{\mu},\widetilde{T})$ be a conservative and ergodic Markov $G$-extension of 
a Gibbs-Markov map $(A, \mu, T)$. Let $f : \ A \times G \to \R$ be measurable. 
Let $q \in [1, \infty)$. Assume that:
\begin{equation}\label{eq:ConditionSommabilite}
\sum_{p \in G} \alpha(p)^{1-\frac{1}{q}} \norm{f(\cdot, p)}{\Lbb^q (A, \mu)} 
< + \infty.
\end{equation}
Then $f_{\{0\}} \in \Lbb^q (A, \mu)$. 
\end{lemma}

\begin{proof}

Now, consider a function $f$ satisfying the condition~\eqref{eq:ConditionSommabilite}. 
Without loss of generality, we can assume $f$ to be non-negative. Note that:
\begin{equation*}
\norm{f_{\{0\}}}{\Lbb^q (A, \mu)} 
= \norm{\sum_{p \in \Z^d} (f \mathbf{1}_p)_{\{0\}}}{\Lbb^q (A, \mu)} 
\leq \sum_{p \in \Z^d} \norm{(f \mathbf{1}_p)_{\{0\}}}{\Lbb^q (A, \mu)}.
\end{equation*}
Then, for all $p \in G \setminus \{0\}$,
\begin{align*}
\norm{(f \mathbf{1}_p)_{\{0\}}}{\Lbb^q (A, \mu)}^q 
& = \int_{A_p} \left( \sum_{k=0}^{N_{-p}-1}f \circ \widetilde{T}_{\{0\}}^k \circ \widetilde{T}_{\{0,p\}} \right)^q \dd \mu \\
& = \frac{1}{\alpha(p)} \int_A \left( \sum_{k=0}^{N_{-p}-1} f \circ \widetilde{T}_{\{0\}}^k \right)^q \dd \widetilde{T}_{\{0,p\}*} \mu (\cdot | A_p).
\end{align*}
By Lemma~\ref{lem:EncoreUneBorneUniforme}, $\widetilde{T}_{\{0,p\}*} \mu (\cdot | A_p) \ll \mu$, with a density which is bounded in 
$\Lip^\infty$ norm, and a fortiori in $\Lbb^\infty (A, \mu)$ norm. We can thus apply Lemma~\ref{lem:BorneSommeLp} : there exists 
a constant $C$, independent from $p$, such that:
\begin{equation*}
\norm{(f \mathbf{1}_p)_{\{0\}}}{\Lbb^q (A, \mu)}^q 
\leq C^q \frac{\alpha (-p)^q}{\alpha(p)} \norm{f \mathbf{1}_p}{\Lbb^q (A, \mu)}^q.
\end{equation*}
Since $\alpha (p) \sim_{p \to \infty} \alpha (-p)$ by Theorem~\ref{thm:ConvergenceNp}, 
up to taking a larger value of $C$, 
\begin{equation*}
\norm{(f \mathbf{1}_p)_{\{0\}}}{\Lbb^q (A, \mu)} 
\leq C \alpha (p)^{1-\frac{1}{q}} \norm{f (\cdot, p) }{\Lbb^q (A, \mu)},
\end{equation*}
whence:
\begin{equation*}
\norm{f_{\{0\}}}{\Lbb^q (A, \mu)} 
\leq C \sum_{p \in \Z^d} \alpha (p)^{1-\frac{1}{q}} \norm{f (\cdot, p)}{\Lbb^q (A, \mu)}. \qedhere
\end{equation*}
\end{proof}

Finally, we prove Theorem~\ref{thm:GM}.

\begin{proof}[Proof of Theorem~\ref{thm:GM}]

Let $(A, \pi, \lambda, \mu, T)$ be an ergodic Gibbs-Markov map. Let $F : \ A \to \Z^d$ 
be $\sigma (\pi)$-measurable, integrable, and such that $\int_A F \dd \mu=0$. 
Assume that the distribution of $F$ with respect to $\mu$ is in the domain of attraction of an 
$\alpha$-stable distribution, and that the Markov extension $(\widetilde{A},\tilde{\mu},\widetilde{T})$ 
is conservative and ergodic.

\smallskip

We first assume that the extension $(\widetilde{A},\tilde{\mu},\widetilde{T})$ is aperiodic.

\medskip
\item \textbf{Aperiodic case.}
By Proposition~\ref{prop:GMHHH}, this extension satisfies Hypothesis~\ref{hyp:HHH}. 
We can thus apply Theorem~\ref{thm:gene}. Let $\beta: \Z^d \rightarrow \R$ be such that:
\begin{itemize}
\item $\beta$ has finite support;
\item $\sum_{p\in\Z^d} \beta(p)=0$.
\end{itemize}
Let $f(x,p):=\beta(p)$. Then:
\begin{equation*}
\frac{S_n^{\widetilde{T}} f}{\sqrt{\sum_{k=0}^{n-1} \mu (S_k=0)}}
\Rightarrow \sigma_{GK} (f, \widetilde{A}, \tilde{\mu}, \widetilde{T}) \Ycal,
\end{equation*}
where $\Ycal$ is a standard $MLGM(1-\frac{d}{\alpha})$ random variable and the convergence is strong in distribution.

\smallskip

We can also apply~\cite[Theorem~6.8]{Thomine:2015}, with $r \equiv 1$. 
The regularity conditions are satisfied, since $f$ and $r$ are constant on the subsets of the Markov partition. 
The integrability condition ``$|f|_{\{0\}} \in \Lbb^p (A, \mu)$ for some $p>2$'' 
is satisfied thank to~\cite[Lemma~6.6]{Thomine:2015}. Hence,
\begin{equation*}
\frac{S_n^{\widetilde{T}} f}{\sqrt{\sum_{k=0}^{n-1} \mu (S_k=0)}}
\Rightarrow \sigma (f) \Ycal,
\end{equation*}
where $\Ycal$ is a standard $MLGM(1-\frac{d}{\alpha})$ random variable and the convergence is strong in distribution, 
and where:
\begin{equation*}
\sigma (f) 
= \lim_{N \to + \infty} \frac{1}{N} \int_A \left( \sum_{k=0}^{n-1} f_{\{0\}} \circ \widetilde{T}_{\{0\}}^k \right)^2 \dd \mu.
\end{equation*}
Following the proof of Lemma~\ref{lem:InvarianceGKFacile}, $\sigma (f_{\{0\}}) = \sigma_{GK} (f_{\{0\}}, A, \mu, \widetilde{T}_{\{0\}})$.

\smallskip

Hence, for any function $\beta$ on $\Z^d$ with finite support and which sums to $0$,
\begin{equation*}
\sigma_{GK} (f, \widetilde{A}, \tilde{\mu}, \widetilde{T}) 
= \sigma_{GK} (f_{\{0\}}, A, \mu, \widetilde{T}_{\{0\}}).
\end{equation*}
Take $\beta := \mathbf{1}_p-\mathbf{1}_0$. Then, for all $q \in (1, \infty)$, 
\begin{equation*}
g(p) 
\sim_{p \to \infty} \frac{\sigma_{GK}^2 (f_p, \widetilde{A}, \tilde{\mu}, \widetilde{T})}{2}
= \frac{\sigma_{GK}^2 (f_{p,\{0\}}, A, \mu, \widetilde{T}_{\{0\}})}{2} 
\sim_{p \to \infty} \alpha (p),
\end{equation*}
where we used Theorem~\ref{thm:ConvergenceNp} to get the last equivalence. Note that we 
already obtain Corollary~\ref{cor:AlphaG}.

\smallskip

Let $\varepsilon >0$. Let $\delta >0$ and $q >2$ be small enough such that:
\begin{equation}\label{eq:ConditionEpsilonDeltaQ}
(\alpha-d+\delta)\left(2-\frac{2}{q}\right)
\leq \alpha-d+2\varepsilon.
\end{equation}
By Proposition~\ref{prop:renouvellement} and Potter's bound, 
$g(p) = O((1+|p|)^{\alpha-d+\delta})$, so $\alpha (p) = O((1+|p|)^{\alpha-d+\delta})$.

\smallskip

We are now ready to apply again~\cite[Theorem~6.8]{Thomine:2015}. Let $f:\widetilde{A} \to \R$ be such that:
\begin{itemize}
\item the family of function $(f (\cdot, p))_{p \in \Z^d}$ is uniformly locally $\eta$-H\"older for some $\eta > 0$;
\item $\int_{\widetilde{A}} (1+|p|)^{\frac{\alpha-d}{2}+\varepsilon} \norm{f(\cdot,p)}{\Lbb^q (A,\mu)} \dd \tilde{\mu}(x,p) <+\infty$ for some $\varepsilon >0$ and $q>2$; 
\item $\int_{\widetilde{A}} f \dd \tilde{\mu}=0$.
\end{itemize}
To apply~\cite[Theorem~6.8]{Thomine:2015}, we only need to check that:
\begin{itemize}
\item $\Ebb_\mu (\sup_{p \in \Z^d} D(f(\cdot, p)) < + \infty$;
\item $|f|_{\{0\}} \in \Lbb^q (A, \mu)$;
\end{itemize}
where $D(f) (x)$ is the Lipschitz norm of $f$ restricted to the Markov subset to which $x$ belongs. 

\smallskip

Without loss of generality, we can use the metric $d^\eta$ on $A$, so that $(f (\cdot, p))_{p \in \Z^d}$ 
is uniformly locally Lipschitz. Then $D(f(\cdot, p))$ is, by hypothesis, bounded uniformly in $p$. 
Hence, $\sup_{p \in \Z^d} D(f(\cdot, p))$ is bounded, and a fortiori integrable: the first point holds.

\smallskip

All is left is to check the second point. We adapt an argument by Cs\'aki, Cs\"org\H{o}, F\"oldes and R\'ev\'esz~\cite[Lemma~3.1]{CsakiCsorgoFoldesRevesz:1992} 
to control the norm of $|f|_{\{0\}}$. Up to choosing a smaller value of $q$, there exists $\delta >0$ which satisfies the 
condition~\eqref{eq:ConditionEpsilonDeltaQ}. Then:
\begin{align*}
\sum_{p \in \Z^d} \alpha(p)^{1-\frac{1}{q}} \norm{f(\cdot, p)}{\Lbb^q (A, \mu)} 
& \leq C \sum_{p \in \Z^d} (1+|p|)^{(\alpha-d+\delta)(1-\frac{1}{q})} \norm{f(\cdot, p)}{\Lbb^q (A, \mu)} \\
& \leq C \sum_{p \in \Z^d} (1+|p|)^{\frac{\alpha-d}{2}+\varepsilon} \norm{f(\cdot, p)}{\Lbb^q (A, \mu)} 
< + \infty.
\end{align*}
By Lemma~\ref{lem:sommabi}, $|f|_{\{0\}} \in \Lbb^q (A, \mu)$. This proves the theorem 
for aperiodic extensions.

\medskip
\item \textbf{Non-aperiodic case.}
For the remainder of this proof, we do not assume that the extension is aperiodic. 

\smallskip

Let $\Lambda \subset \Z^d$ be the set of the essential values of the extension~\cite{Aaronson:1997}. 
Since the extension is ergodic, $\Lambda \simeq \Z^d$ is a cocompact lattice. Let $0 \in B  \subset \Z^d$ 
be a fundamental domain for this lattice. Let:
\begin{itemize}
\item $A_B := A \times B$;
\item $\mu_B := |B|^{-1} \mu \otimes \sum_{b \in B} \delta_b$;
\item $T_B (x,b) = (T(x), b+F(x) [\Lambda])$.
\end{itemize}
Since $(\widetilde{A},\tilde{\mu},\widetilde{T})$ is ergodic, $(A_B, \mu_B, T_B)$ 
is a measure-preserving ergodic dynamical system, which is Gibbs-Markov. There 
exists $F_\Lambda : A_B \to \Lambda$, constant on the elements of the Gibbs-Markov partition, 
such that $(\widetilde{A},\tilde{\mu},\widetilde{T})$ is isomorphic to the 
extension $(\widetilde{A}_B,\tilde{\mu}_B,\widetilde{T}_B)$ with step function $F_\Lambda$. 
The later extension is a conservative, ergodic, aperiodic Markov extension of a Gibbs-Markov map. 

\smallskip

The function $f$ still satisfies our assumptions for the new system (it is uniformly locally H\"older, 
decays at a sufficient rate at infinity, and has zero integral). Thus, we can apply 
the version of Theorem~\ref{thm:GM} for aperiodic systems; this yields:
\begin{equation*}
\frac{S_n^{\widetilde{T}} f}{\sqrt{\sum_{k=0}^{n-1} \mu (S_k \in B)}}
\Rightarrow \sigma_{GK} (f_B, A_B, \mu_B, \widetilde{T}_{B,\{0\}}) \Ycal,
\end{equation*}
where $\Ycal$ is a standard $MLGM(1-\frac{d}{\alpha})$ random variable, the convergence is strong in distribution, and:
\begin{equation*}
\sigma_{GK}^2 (f_B, A_B, \mu_B, \widetilde{T}_{B,\{0\}})
:= \lim_{n \to + \infty} \int_{A_B} f_B^2 \dd\mu_B + 2 \sum_{k=1}^n \int_{A_B} f_B \cdot f_B \circ \widetilde{T}_{B,\{0\}}^k \dd \mu_B,
\end{equation*}
where the limit is taken in the Ces\`aro sense.

\smallskip

The proof of~\cite[Lemma~3.7.4]{Aaronson:1997} can be adapted to ergodic Gibbs-markov maps (instead of continued fraction mixing maps), 
by replacing $\widetilde{T}_A^k$ with $M^{-1} \sum_{k=0}^{M-1} \widetilde{T}_A^k$, which can be done up to a uniformly 
bounded error term. As $\widetilde{T}_B$ is an ergodic Gibbs-markov map, $A \times B$ is thus also a Darling-Kac set, 
and a set on which R\'enyi's inequality is satisfied. By~\cite[Theorem~3.3.1]{Aaronson:1997},
\begin{equation}\label{eq:AperiodicToPeriodic1}
\lim_{n \to + \infty} \frac{\sum_{k=0}^{n-1} \mu (S_k \in B)}{\sum_{k=0}^{n-1} \mu (S_k = 0)} 
= \lim_{n \to + \infty} \frac{\sum_{k=0}^{n-1} \tilde{\mu} (A \times \{0\} \cap \widetilde{T}^{-k} (A \times B)) }{\sum_{k=0}^{n-1} \tilde{\mu} (A \times \{0\} \cap \widetilde{T}^{-k} (A \times \{0\})) } 
= |B|.
\end{equation}
Using the induction invariance of the Green-Kubo formula (Lemma~\ref{lem:InvarianceGKFacile}) with the observable 
$f_B$ on $(A_B, \mu_B, \widetilde{T}_{B,\{0\}})$, noting that the induced transformation 
on $A \times \{0\}$ is $\widetilde{T}_{\{0\}}$, we get:
\begin{align}
\sigma_{GK}^2 (f_B, A_B, \mu_B, \widetilde{T}_{B,\{0\}}) 
& = \lim_{n \to + \infty} \int_A f^2 \dd\mu_B + 2 \sum_{k=1}^n \int_A f \cdot f \circ \widetilde{T}_{\{0\}}^k \dd \mu_B \nonumber \\
& = |B|^{-1} \sigma_{GK}^2 (f, A, \mu, \widetilde{T}_{\{0\}}), \label{eq:AperiodicToPeriodic2}
\end{align}
where the limit is taken in the Ces\`aro sense. Equations~\eqref{eq:AperiodicToPeriodic1} and~\eqref{eq:AperiodicToPeriodic2} 
together yield the claim.
\end{proof}

\section{Applications}
\label{sec:Applications}

In this section, we prove our claims of Subsection~\ref{subsec:Applications}, 
starting with the geodesic flow and finishing with the billiards.

\subsection{Periodic planar billiard in finite horizon}
\label{subsec:ApplicationBillard}

Recall that the billiard table is $\R^2 \setminus \bigcup_{i\in\Ical,\ p\in \Z^2} (p+O_i)$, where 
$(O_i)_{i \in \Ical}$ corresponds to a finite family of open convex subsets of $\Tbb^2$, whose boundaries 
are non-overlapping, $\Ccal^3$, and with non-vanishing curvature. For the collision map, the phase space is $\Omega := \partial Q \times [-\pi/2, \pi/2]$. 
The invariant measure is the Liouville measure $\cos (\phi) \dd x \dd \phi$ in $(x, \phi)$, 
where $x$ is the curvilinear coordinate on $\partial Q$.

\smallskip

A particle has configuration $(x,\phi,i,p)$ if it is located in $p+\partial O_i$, 
with curvilinear coordinate $x$ on $\partial O_i$ (for some counterclockwise
curvilinear parametrization of $\partial O_i$) and if its reflected vector $V$ makes the angle $\phi$
with the inward normal vector to $\partial O_i$. The billiard map $\widetilde{T}_0:\Omega\to \Omega$
maps a configuration in $\Omega$ to the configuration corresponding to the next collision time.
This transformation preserves the Liouville measure $\tilde{\nu}$, which has infinite mass.

\smallskip

We consider a particle starting from the \textit{original cell} $\Ccal_0 = \bigcup_{i \in \Ical} O_i$ 
with initial distribution $\nu := \tilde{\nu} (\cdot|\Ccal_0)$.

\smallskip

The associated compact billiard is the system $(M,\nu,T_0)$, with $M := \Ccal_0$ 
and $\widetilde{T}_0 (x,\phi,i,p) = (T_0(x,\phi,i),p+H(x,\phi,i))$
Then $(\Omega,\tilde{\nu}/\tilde{\nu}(\Ccal_0), \widetilde{T}_0)$ is the $\Z^2$-extension 
of $(M,\nu,T_0)$ with step function $H:M\to\Z^2$ corresponding to the 
change of cells. The quantity $S_n^T H (y) := \sum_{k=0}^{n-1} H \circ T_0^k (y)$ corresponds 
to the index of the cell containing $\widetilde{T}_0^n (y)$, for all $y \in \Ccal_0$.

\smallskip

Let $\varepsilon >0$, and $\beta:\Z^2 \to \R$ be such that $\sum_{p\in\Z^2} \beta(p)=0$. 
We associate the value $\beta(p)$ to the cell $\Ccal_p$, and put for $y \in \Ccal_0$:
\begin{equation*}
\Ycal_n (y)
:= \sum_{k=1}^n \beta \left(S_k^{T_0} H(y) \right).
\end{equation*}

\begin{proof}[Proof of Corollaries~\ref{coro:bill1} and~\ref{coro:bill2}]

Due to Young's towers~\cite{Young:1998}, we know that there exists a dynamical
system $(A,\mu,T)$ such that $(A,\mu,T)$ and $(M,\nu,T_0)$ are both factors of another dynamical system
$(\widehat{A},\hat{\mu},\widehat{T})$. This means that there exist two maps 
$\hat{\pi}:(\widehat{A},\hat{\mu},\widehat{T})\to(A,\mu,T)$ and 
$\pi:(\widehat{A},\hat{\mu},\widehat{T})\to(M,\nu,T_0)$ such that:
\begin{align*}
\hat{\pi}\circ\hat{T} & = T \circ \hat{\pi}, \\
\pi\circ \widehat{T} & = T_0 \circ \pi,\\
\hat{\pi}_* \hat{\mu} & = \mu, \\
\pi_* \hat{\mu} & = \nu.
\end{align*}
Moreover, there exist $F:A\to\Z^d$ and $\beta:A\to \Z^2$ such that 
$F\circ\hat{\pi}=H\circ\pi$ and $\hat{\beta} \circ \hat{\pi}=\beta \circ\pi$.

\smallskip

The properties of the family of transfer operators $P_u=P(e^{i\langle u, F \rangle})$ 
for such step function $F$ have been studied: see for instance~\cite{SzaszVarju:2004,DolgopyatSzaszVarju:2008,Pene:2009a,Pene:2009b},
in which local limit theorems with various remainder terms have been established. The matrix $\Sigma$ 
corresponds to the asymptotic variance matrix of $(S_n^T F/\sqrt{n})_{n \geq 1}$ with respect to $\mu$, 
which is the same as the asymptotic variance matrix of $(S_n^{T_0}H/\sqrt{n})_{n \geq 1}$ with respect to $\nu$, and is given by:
\begin{equation*}
\Sigma 
= \sum_{k\in\Z} C(H,H\circ T^k),
\end{equation*}
where $C(H,H\circ T^k)$ denotes the matrix of covariances of $H$ and $H\circ T^k$ with respect to $\nu$.
Recall that $(S_n^{T_0}H/\sqrt{n})_{n \geq 1}$ converges in distribution to
a centered gaussian random variable with variance matrix $\Sigma$.

\smallskip

Let $\Zcal_n: M\to \R$ be defined by $\Zcal_n(x):=\sum_{k=0}^{n-1} \hat{\beta} (S_k^T F(x))$.
This function satisfies $\Zcal_n \circ\hat{\pi} = \Ycal_n \circ \pi$ on $\hat{A}$.
Applying Theorem~\ref{thm:gene} to the dynamical system $(A,\mu,T)$, step function $F$ (respectively, the first coordinate 
$F_1:A\to \Z$ of $F$) and $\hat{\beta}$ (respectively, $p\mapsto \hat{\beta}(p,0)$), 
we obtain Corollary~\ref{coro:bill2} (respectively, Corollary~\ref{coro:bill1}).
\end{proof}

\subsection{Geodesic flow on periodic hyperbolic manifolds}
\label{subsec:ApplicationFlotGeodesique}

We recall that $M$ is a compact, connected manifold with a Riemannian metric of 
negative sectional curvature, and $\varpi:N \to M$ be a connected $\Z^d$-cover of $M$, with $d \in \{1, 2\}$. 
The manifold $T^1 N$ is endowed with the $\sigma$-finite lift $\mu_N$ of a Gibbs 
measure $\mu_M$ corresponding to a reversible H\"older potential. The geodesic flow on $T^1 N$ 
is denoted by $(g_t)_{t \in \R}$. 

\smallskip

Let $(A,\mu,T)$ be a Markov section for the geodesic flow on $T^1 M$, as constructed by Bowen \cite{Bowen:1973}, \cite[Theorem~3.12]{Bowen:1975}. 
The section $A$ is constructed by carefully choosing a finite number of pieces of strong unstable manifolds $(W^u (a))_{a \in \pi}$, 
then, for all $x \in W^u (a)$, adding a piece of strong stable manifold $W^s (x)$ to get rectangles. We shall denote by 
$p_+$ the projection onto unstable manifolds, defined by $p_+ (y)=x$ whenever $y \in W^s (x)$ and $x \in W^u (a)$.
Let $r$ be the return time to $A$; by Bowen's construction, $r(x)$ depends only on the future (the non-negative coordinates) 
of $x$. Finally, we put $A_+ := \bigcup_{a \in \pi} W^u (a)$ as the state space of the one-sided transformation.

\smallskip

The set $\widetilde{A} := \varpi^{-1} (A)$ is a section for the geodesic flow on $T^1 N$, 
with return time $\tilde{r} = r \circ \varpi$. The induced map on $\widetilde{A}$ is the $\Z^d$-extension of the natural extension 
of a Gibbs-Markov map, with step function $F$. Without loss of generality, we may refine the Markov partition on $A$ 
so that $F$ depends only on the first coordinate of the shift; then, the extension $(\widetilde{A}, \tilde{\mu}, \widetilde{T})$ is Markov.
The geodesic flow on $T^1 N$ is thus isomorphic to the suspension flow over $(\widetilde{A}, \tilde{\mu}, \widetilde{T})$ 
with roof function $\tilde{r}$. In particular, $T^1 N \simeq \{(x,q,t): \ x \in A, \ q \in \Z^d, \ t \in [0,r(x))\}$.

\smallskip

Let $f : T^1 N \to \R$ be H\"older. The following lemma asserts that, up to adding a coboundary, 
we can assume that $f$ depends only on the future, which allows us to work with Gibbs-Markov maps 
instead of their natural extension. While this lemma is classic~\cite[Lemma~1.6]{Bowen:1975}, 
we give a statement which is valid in the context of $\Z^d$-extensions.

\begin{lemma}
\label{lem:InversibleCobords}

Let $(A, \pi, \lambda, \mu, T)$ be the natural extension of an ergodic Gibbs-Markov 
map\footnote{The metric being defined by $\lambda^{-s}$, where $s$ is the two-sided separation time.}. 
Let $(A \times \Z^d, \tilde{\mu}, \widetilde{T})$ be a Markov $\Z^d$-extension with step function $F$. 
Let $f$ be a measurable real-valued function on $A \times \Z^d$. 
Assume that:
\begin{equation*}
\norm{D(f)}{\infty} 
:= \norm{f}{\infty} + \sup_{q \in \Z^d} \sup_{a \in \pi} |f|_{\Lip(a \times\{q\})} 
< +\infty.
\end{equation*}
Then there exists a function $u$ which is bounded by $\lambda (\lambda-1)^{-1} \norm{D(f)}{\infty}$, 
uniformly $1/2$-H\"older, and such that the function $f_+ := f + u \circ T-u$ is $\Bcal_+$-measurable, 
with $\Bcal_+ := \left( \bigvee_{n \geq 0} T^{-n} \pi \right) \otimes \Pcal (\Z^d)$.
\end{lemma}

\begin{proof}

Let $\tilde{p}_+ (x,q) := (p_+(x),q)$ be defined on $\widetilde{A}$. We put: 
\begin{equation*}
u 
:= \sum_{n=0}^{+ \infty} f \circ \widetilde{T}^n - f \circ \widetilde{T}^n \circ \tilde{p}_+.
\end{equation*}
The proof then proceeds as in~\cite[Lemma~6.11]{Thomine:2015}: the function $u$ satisfies the conclusion 
of the lemma. Most changes in the proof of~\cite[Lemma~6.11]{Thomine:2015} are straightforward; 
the only observation needed is that, if $x$ and $y$ are in the same cylinder of length $n$ in $A$, 
then $\widetilde{T}^k (x,q)$ and $\widetilde{T}^k (y,q)$ are in the same set $A \times \{S_k F(x)\}$ for 
$|k| \leq n$, so that we can use the Lipschitz estimate for each $f(\cdot, S_k F(x))$.
\end{proof}

We are now ready to prove Proposition~\ref{prop:FlotGeodesique}.

\begin{proof}[Proof of Proposition~\ref{prop:FlotGeodesique}]

The proof follows the one in~\cite[Proposition~6.12]{Thomine:2015}, with a few significant modifications. The first step is to 
eliminate to past, that is, add a coboundary to get an observable which depends only on the future, to be able 
to use~\cite[Proposition~6.1]{Thomine:2015}. Let $\eta \in (0,1]$. Let $f : T^1 N \to \R$ be a $\eta$-H\"older observable, 
which satisfies the hypotheses of the proposition. We put:
\begin{itemize}
\item $f_{\widetilde{A}} (x,q) := \int_0^{r(x)} f(x,q,s) \dd s$;
\item $u_{\widetilde{A}}$ the function obtained from $f_{\widetilde{A}}$ by the construction of Lemma~\ref{lem:InversibleCobords};
\item $f_{+,\widetilde{A}} := f_{\widetilde{A}} + u_{\widetilde{A}} \circ \widetilde{T}-u_{\widetilde{A}}$;
\item $f_+ (x,q,t) := r(x)^{-1} f_{+,\widetilde{A}} (x,q)$.
\end{itemize}
By Lemma~\ref{lem:InversibleCobords}, the function $u_{\widetilde{A}}$ is $\eta/2$-H\"older and bounded. 
Then, using the fact that $f_{+,\widetilde{A}}-f_{\widetilde{A}}$ is a coboundary,
\begin{equation}\label{eq:BorneObservableCobord}
\sup_{t \geq 0} \norm{\int_0^t f \circ g_s \dd s - \int_0^t f_+ \circ g_s \dd s}{\infty} 
\leq 2 \norm{u_{\widetilde{A}}}{\infty} + 2 \norm{f_{+,\widetilde{A}}}{\infty} + 2 \norm{r}{\infty} \norm{f}{\infty} 
< + \infty.
\end{equation}
Hence, it is enough to prove the limit theorem for $f_+$. Note that $f_+$ is a coboundary if and only if 
$f$ is a coboundary. 

\smallskip

Let $\varphi_{A\times\{0\}}$ be the first return time to $A \times \{0\}$ 
for the geodesic flow, and $\overline{\varphi}_{A \times \{0\}}$ the first return time 
to $A \times \{0\}$ for $\widetilde{T}$. The proof then proceeds as in~\cite{Thomine:2015}, 
with the same weakened criterion: we only need to check that, for some $\delta >0$,
\begin{equation}\label{eq:FlotGeodesiqueConditionLpF+}
\sup_{0 \leq t \leq \varphi_{A \times \{0\}}} \left| \int_0^t f_+ \circ g_s \dd s \right| 
\in \Lbb^{2+\delta} (A \times \{0\});
\end{equation}

\smallskip

Now, we shall go back to the initial (invertible) system to use the integrability assumption on $f$. 
Equations~\eqref{eq:BorneObservableCobord} and~\eqref{eq:FlotGeodesiqueConditionLpF+} together 
yield:
\begin{equation}\label{eq:FlotGeodesiqueConditionLp}
\sup_{0 \leq t \leq \varphi_{A \times \{0\}}}  \left| \int_0^t f \circ g_s \dd s \right| 
\in \Lbb^{2+\delta} (A \times \{0\}).
\end{equation}

Finally, once again, we go to the non-invertible factor. 
Let $\overline{f}_{\widetilde{A}} (x,q) := \norm{\int_0^{r(\cdot)} |f| (\cdot, q, t) \dd t}{\infty}$ . Then:
\begin{equation*}
\sup_{0 \leq t \leq \varphi_{A \times \{0\}}} \left| \int_0^t f \circ g_s (x,0,0) \dd s \right| 
\leq \sum_{n=0}^{\overline{\varphi}_{A \times \{0\}} (x,0)-1} \overline{f}_{\widetilde{A}} \circ \widetilde{T}^n (x).
\end{equation*}
The function $\overline{f}_{\widetilde{A}}$ is an upper bound on $|f|_{\widetilde{A}}$ which depends only on $q$, 
and thus not on the past. Hence, it factorizes as a function of $A_+ \times \Z^d$. 
In addition, the integrability assumptions yields:
\begin{equation*}
\sum_{q \in \Z^d} |q|^{1-\frac{d}{2}+\varepsilon} \norm{\overline{f}_{\widetilde{A}} (\cdot, q)}{\infty} 
< + \infty.
\end{equation*}
By Lemma~\ref{lem:sommabi}, $\overline{f}_{\widetilde{A}}$ belongs to $\Lbb^{2+\delta} (A \times \{0\})$ 
for all small enough $\delta >0$, which yields Equation~\eqref{eq:FlotGeodesiqueConditionLp}.
\end{proof}

\appendix

\section{About Green-Kubo's formula}
\label{sec:GreenKubo}

The spirit behind Corollary~\ref{cor:InvarianceInduction}, and thus of our alternative 
proof of Spitzer's theorem~\cite[Chap.~III.11, P5]{Spitzer:1976}, is that Green-Kubo's formula 
satisfies an invariance by induction which is reminiscent of Kac's theorem. We shall draw this parallel 
here, as well as prove a specific instance of this phenomenon which is useful in the proof of Theorem~\ref{thm:GM}. 
In what follows, the measure may be finite or $\sigma$-finite.

\smallskip

Given an ergodic, conservative, measure-preserving dynamical system $(A, \mu, T)$ and a 
measurable subset $B \subset A$ such that $\mu (B) >0$, one may define the system induced 
on $B$ by $(B, \mu_{|B}, T_B)$. Given any measurable observable $f : A \to \C$, 
we also define the induced observable $f_B$ by:
\begin{equation*}
f_B (x) 
= \sum_{k=0}^{\varphi_B (x)-1} f(T^k (x)),
\end{equation*}
where $\varphi_B$ is the first return time to $B$. Then, a generalization of 
Kac's theorem~\cite{Kac:1947} asserts that the integral is invariant by induction.

\begin{theorem}[Kac's theorem: induction invariance of the integral]

Let $(A, \mu, T)$ be an ergodic, conservative, measure-preserving dynamical system. 
Let $B \subset A$ be a measurable subset with $0<\mu(B)<+\infty$. Then, for all 
$f \in \Lbb^1 (A, \mu)$,
\begin{equation}\label{eq:KacGeneral}
\int_A f \dd \mu 
= \int_B f_B \dd \mu.
\end{equation}
\end{theorem}

A consequence is that the map $f \mapsto f_B$ is a weak contraction from $\Lbb^1 (A, \mu)$ 
to $\Lbb^1 (B, \mu)$. There are two different ways to prove this theorem:
\begin{itemize}
\item Using the fact that the system is measure-preserving~\cite{Kakutani:1943}: up to going to the natural extension, 
we can define $\varphi_{-1, B} (x) := \inf \{n \geq 0 : \ T^{-n} (x) \in B\}$, and then using, for all $n \geq 0$,
\begin{equation*}
\int_A f \mathbf{1}_{\varphi_{-1, B} = n} \dd \mu 
= \int_B f \circ T^n \mathbf{1}_{\varphi_B \geq n} \dd \mu.
\end{equation*}
\item Using a convergence theorem, such as Hopf's ergodic theorem~\cite[\S$14$, Individueller Ergodensatz f\"ur Abbildungen]{Hopf:1937}, 
and the preservation of the measure for the induced system. Setting $g := \mathbf{1}_B$, one can identify the almost sure limit 
of $(S_n^T f) / (S_n^T g)$ with that of $(S_n^{T_B} f_B)/ n$, and conclude.
\end{itemize}

Green-Kubo's formula\footnote{The discussion can be generalized by taking two different observables: 
what is invariant is actually the underlying bilinear form.}, at least at a formal level, behaves the same. 
For any square-integrable function $f$ with zero integral,
\begin{equation}\label{eq:GKGeneral}
\int_A f^2 \dd \mu + 2 \sum_{n=1}^{+ \infty} \int_A f \cdot f \circ T^n \dd \mu 
= \int_B f_B^2 \dd \mu + 2 \sum_{n=1}^{+ \infty} \int_B f_B \cdot f_B \circ T_B^n \dd \mu.
\end{equation}
The reader may compare Equations~\eqref{eq:KacGeneral} and~\eqref{eq:GKGeneral}. As 
with Kac's theorem, we may choose different strategies to prove rigorously such an identity. 
Using the fact that the system is meaure-preserving, and cutting in a well-chosen way the integrals 
above, one can see that they are formally the same. However, to get a rigorous proof, one 
would have to use Fubini's theorem, which fails in this case. This is not surprising, as the 
infinite sum may not be well-defined, or a periodicity of the system may require us to take 
a weaker kind of limit (e.g.\ in the Ces\`aro sense, as in Theorem~\ref{thm:GM}). One cannot 
hope~\eqref{eq:GKGeneral} to hold for any dynamical system, or any square-integrable centered function.

\smallskip

Another strategy is to use a distributional limit theorem: for sufficiently hyperbolic systems 
and nice enough observables, Green-Kubo's formula is the asymptotic variance in a central limit theorem. 
Working at two different time scales (with the initial system and with the induced system), 
one can prove that this invariance holds. A very simple example is given by the following lemma.

\begin{lemma}\label{lem:InvarianceGKFacile}

Let $(A, \pi, \lambda, \mu, T)$ be a Gibbs-Markov map. Let $f \in \Lbb^2 (A, \mu)$ be a real-valued function such that:
\begin{itemize}
\item $f$ is locally H\"older: $\sum_{a \in \pi} \mu (a) |f|_{\Lip (a)} < +\infty$;
\item $\int_A f \dd \mu = 0$.
\end{itemize}

Let $B \subset A$, with $\mu(B)>0$, be $\sigma(\pi)$-measurable. Assume that $\varphi_B$ is 
essentially constant. Then:
\begin{equation*}
\lim_{n \to + \infty} \int_A f^2 \dd \mu + 2 \sum_{k=1}^n \int_A f \cdot f \circ T^k \dd \mu 
= \lim_{n \to + \infty} \int_B f_B^2 \dd \mu + 2 \sum_{k=1}^n \int_B f_B \cdot f_B \circ T_B^k \dd \mu,
\end{equation*}
where both limits are taken in the Ces\`aro sense.
\end{lemma}

\begin{proof}

Let $M := \varphi_B$ almost everywhere. Under the assumptions, the Birkhoff sums (for $T$) of $f$ satisfy a central limit theorem 
(see e.g.\ \cite[Th\'eor\`eme~4.1.4]{Gouezel:2008e}, and use the Taylor expansion of $(I-P)^{-1}$): 
\begin{equation*}
\frac{S_n^T f}{\sqrt{n}} 
\to \sigma \Ncal,
\end{equation*}
where the convergence is in distribution on $(A, \mu)$, $\Ncal$ 
follows a standard Gaussian distribution, and:
\begin{equation*}
\sigma^2 
= \lim_{n \to + \infty} \frac{1}{n} \int_A (S_n^T f)^2 \dd \mu.
\end{equation*}
By~\cite[Theorem~1]{Zweimuller:2007}, the same central limit theorem holds strongly in distribution, 
that is, when the initial measured space is $(A, \nu)$, with $\nu \ll \mu$. This holds in particular 
on $(B, M \mu_{|B})$.

\smallskip

Under the same assumptions, the Birkhoff sums (for $T_B$) of $f_B$ satisfy a central limit theorem. 
Then:
\begin{equation*}
\frac{S_n^{T_B} f_B}{\sqrt{n}} 
\to \sigma' \Ncal,
\end{equation*}
where the convergence is in distribution on $(B, M \mu_{|B})$, $\Ncal$ 
follows a standard Gaussian distribution, and:
\begin{equation*}
(\sigma')^2 
= \lim_{n \to + \infty} \frac{M}{n} \int_B (S_n^{T_B} f_B)^2 \dd \mu.
\end{equation*}
Note that $S_n^{T_B} f_B = S_{Mn}^T f$, whence $\sigma' = \sqrt{M} \sigma$. 
This yields:
\begin{equation*}
\lim_{n \to + \infty} \frac{1}{n} \int_A (S_n^T f)^2 \dd \mu 
= \lim_{n \to + \infty} \frac{1}{n} \int_B (S_n^{T_B} f_B)^2 \dd \mu.
\end{equation*}

Finally, note that:
\begin{align*}
\frac{1}{N} \int_A (S_N^T f)^2 \dd \mu 
& = \frac{1}{N} \sum_{k, n=0}^{N-1} \int_A f \circ T^k \cdot f \circ T^n \dd \mu \\
& = \frac{1}{N} \sum_{n=0}^{N-1} \left[ \int_A f^2 \dd \mu + 2 \sum_{k=1}^n \int_A f \cdot f \circ T^k \dd \mu \right].
\end{align*}
Hence, $\sigma^2$ is the Ces\`aro-limit of $(\int_A f^2 \dd \mu + 2 \sum_{k=1}^n \int_A f \cdot f \circ T^k \dd \mu)_{n \geq 1}$. 
The same manipulation with $\frac{1}{n} \int_B (S_n^{T_B} f_B)^2 \dd \mu$ yields the lemma.
\end{proof}

In this article, the proof of Corollary~\ref{cor:InvarianceInduction} relies on this approach: 
we obtain two distributional limit theorems by working at two different time scales, and then identify the limits. 
However, as can be seen, obtaining these limit theorems gets much more challenging when working 
with null recurrent processes.

\end{document}